\documentclass[12pt,reqno]{amsart}
\usepackage[foot]{amsaddr}
\usepackage[utf8]{inputenc}
\usepackage{enumerate,enumitem,amsmath,amsthm,latexsym,amssymb,statmath,comment,times,bm}
\usepackage{mathtools}
\usepackage{dsfont}
\usepackage{mathrsfs}

\usepackage{aligned-overset}

\usepackage{algorithm}
\usepackage[noEnd = false,commentColor=blue]{algpseudocodex}
\makeatletter
\newenvironment{breakablealgorithm}{
   \begin{center}
        \vspace{3mm}
        \refstepcounter{algorithm}
        \hrule height.8pt depth0pt \kern2pt
        \renewcommand{\caption}[2][\relax]{
           {\raggedright\textbf{\fname@algorithm~\thealgorithm :} ##2\par}%
           \ifx\relax##1\relax 
             \addcontentsline{loa}{algorithm}{\protect\numberline{\thealgorithm}##2}%
           \else 
             \addcontentsline{loa}{algorithm}{\protect\numberline{\thealgorithm}##1}%
           \fi
           \kern2pt\hrule\kern2pt
        }
        }{
        \kern2pt\hrule\relax
        \vspace{3mm}
    \end{center}
}
\makeatother

\algnewcommand\algorithmicinputnew{\textbf{Input:}}
\algnewcommand\algorithmicoutputnew{\textbf{Output:}}
\algnewcommand\MyInput{\item[\algorithmicinputnew]}%
\algnewcommand\MyOutput{\item[\algorithmicoutputnew]}%

\usepackage{hyperref}
\usepackage[capitalise]{cleveref}
\crefname{algocfline}{Algorithm}{Algorithms}
\crefname{equation}{}{}
\crefname{prop}{Proposition}{Propositions}
\crefname{enumi}{}{}
\crefname{figure}{Figure}{Figures}
\crefname{claim}{Claim}{Claims}
\crefname{subsection}{Subsection}{Subsections}

\newtheorem{theorem}{Theorem}[section]
\newtheorem{corollary}[theorem]{Corollary}
\newtheorem{lemma}[theorem]{Lemma}
\newtheorem{claim}[theorem]{Claim}
\newtheorem{proposition}[theorem]{Proposition}

\newtheorem{observation}[theorem]{Observation}
\theoremstyle{definition}

\newcommand\bfrac[2]{\left(\frac{#1}{#2}\right)}

\newcommand{\prob}[1]{\mathbb{P}\left[#1\right]}
\newcommand{\variance}[1]{\mathbb{V}\left[#1\right]}
\newcommand{\expec}[1]{\mathbb{E}\left[#1\right]}
\newcommand{\condexpec}[2]{\mathbb{E}\left[#1 \;\middle|\; #2\right]}
\newcommand{\condprob}[2]{\mathbb{P}\left[#1\;\middle|\; #2\right]}

\def\cP{\mathcal{P}}

\def\cE{\mathcal{E}}

\def\cT{\mathcal{T}}
\def\cF{\mathcal{F}}
\def\cQ{\mathcal{Q}}
\def\tcF{\widetilde{\mathcal{F}}}
\def\cU{\mathcal{U}}

\def\tL{\widetilde{L}}
\def\tX{\widetilde{X}}
\def\tD{\widetilde{D}}
\def\tM{\widetilde{M}}
\def\tB{\widetilde{B}}
\def\tk{\widetilde{k}}
\def\tS{\widetilde{S}}
\def\tC{\widetilde{C}}
\def\tY{\widetilde{Y}}

\def\hL{\widehat{L}}
\def\hX{\widehat{X}}

\def\eps{\varepsilon}

\def\dist{\mathrm{dist}}

\usepackage{xcolor} 
\def\gr#1{}
\renewcommand{\gr}[1]{\textcolor{teal}{#1}}

\usepackage{calc}
\usepackage{tikz}
\usetikzlibrary{arrows}
\usetikzlibrary{decorations}
\usetikzlibrary{shapes.misc}
\usetikzlibrary{decorations.pathreplacing}
\usetikzlibrary{calc}
\usetikzlibrary{math}

\newenvironment{proofclaim}[1][Proof of Claim]{\begin{proof}[#1]}{\end{proof}}

\newcommand{\ind}[1]{\mathds{1}_{\{#1\}}}
\newcommand{\indev}[1]{\mathds{1}_{#1}}
\newcommand{\tbound}{(10ck)^{2k}}

\newcommand{\tG}{\widetilde{G}}
\newcommand{\uc}{\mathrm{uc}}

\title{The law of the circumference of sparse binomial random graphs}
\author{Michael Anastos$^1$} 
\address{$^1$Institute of Science and Technology Austria (ISTA), 3400 Klosterneurburg, Austria}
\author{Joshua Erde$^2$ \and Mihyun Kang$^2$ \and Vincent Pfenninger$^2$}
\address{$^2$Institute of Discrete Mathematics, Graz University of Technology, Steyrergasse 30, 8010 Graz, Austria}
\email{michael.anastos@ist.ac.at}
\email{\hspace{-3pt}\{erde,kang,pfenninger\}@math.tugraz.at}

\date{\today}

\begin{document}
\begin{abstract}
There has been much interest in the distribution of the \emph{circumference}, the length of the longest cycle, of a random graph $G(n,p)$ in the \emph{sparse} regime, when $p = \Theta\left(\frac{1}{n}\right)$. Recently, the first author and Frieze established a scaling limit for the circumference in this regime, along the way establishing an alternative `structural' approximation for this parameter. In this paper, we give a central limit theorem for the circumference in this regime using a novel argument based on the Efron--Stein inequality, which relies on a combinatorial analysis of the effect of resampling edges on this approximation.
\end{abstract}

\maketitle

\section{Introduction}
\subsection{Motivation and results}
In this paper we will study the \emph{circumference}, the length of the longest cycle, of a \emph{sparse} random graph $G(n,p)$, where $p=\frac{c}{n}$ for some constant $c>0$. Answering a well-known conjecture of Erd\H{o}s~\cite{E75}, Ajtai, Koml\'{o}s and Szemer\'{e}di~\cite{AKS81} showed the existence of a function $\ell \colon (1,\infty)\to (0,\infty)$ such that if $c > 1$, then whp\footnote{With probability tending to $1$ as $n$ tends to infinity.} the circumference of $G\left(n,\frac{c}{n}\right)$ is at least $\ell(c)\cdot n$, where we note that a slightly weaker result was also proved independently by Fernandez de la Vega~\cite{FdlV79}. 
Since then, the problem of proving more precise results about the distribution of the circumference of sparse random graphs has attracted considerable interest. 

Given a graph $G$ let us write $L(G)$ for the circumference of $G$. We define
\[
\ell_{\textrm{min}}(c) ~\coloneqq~ \sup\left\{\alpha > 0 \colon L\left(G\left(n,\frac{c}{n}\right)\right) \geq \alpha n \quad \text{whp}\right\}.
\]
We note that it is easy to show that $\ell \to 1$ as $c \to \infty$ and $\ell \to 0$ as $c \to 1$.
Early work, of Bollob\'{a}s~\cite{B82}, of Bollob\'{a}s, Fenner and Frieze~\cite{BFF84} and of Frieze~\cite{F86}, focused on the asymptotic dependence of the function $1-\ell_{\textrm{min}}$ on $c$ as $c \to \infty$. In particular, Frieze~\cite{F86} determined the value of $1-\ell_{\textrm{min}}(c)$ up to a $1+o_c(1)$ multiplicative factor. On the other hand, the \emph{weakly supercritical regime}, where $\varepsilon \coloneqq c-1$ is such that $\varepsilon^3n \to \infty$ and $\varepsilon$ is sufficiently small, has also been studied~\cite{KW13,L91}. In particular, the current best known asymptotic bounds are
\[
1.581 \varepsilon^2 \leq \ell_{\textrm{min}}(1+\varepsilon) \leq 1.7395\varepsilon^2,
\]
where the lower bound is due to the first author \cite{A22} and the upper bound is due to Kemkes and Wormald \cite{KW13}.

A recent result of the first author and Frieze \cite{AF21} proved that, for large enough $c$, the circumference has a \emph{scaling limit}.
\begin{theorem}[{\cite[Theorem 1.6]{AF21}}]\label{t:AnastosFriezeScaling}
Let $c>1$ be a sufficiently large constant. Then there exists a function $f\colon (1,\infty)\to (0,\infty)$ such that
\[
\frac{L\left(G\left(n,\frac{c}{n}\right)\right)}{n}\quad\to\quad f(c) \qquad \text{almost surely, as $n \rightarrow \infty$.}
\]
\end{theorem}
Whilst the function $f$ in \cref{t:AnastosFriezeScaling} is not explicit, the authors gave a method to compute $f$ to an arbitrary level of precision. Similarly, the bound on $c$ in \cref{t:AnastosFriezeScaling} was not made explicit, but parts of the arguments require that $c \geq 10^6$. However, the result was later shown to hold for all $c \geq 20$ by the first author \cite[Theorem 1]{A23}.

A natural question to ask, once the existence of a scaling limit has been established, is whether we can determine more precisely the distribution of $L\left(G\left(n,\frac{c}{n}\right)\right)$, and in particular whether it satisfies a \emph{central limit theorem}.
Here, we say that a sequence of random variables $(X_n)_{n \in \mathbb{N}}$ \emph{satisfies a central limit theorem} if 
\[
\frac{X_n - \expec{X_n}}{\sqrt{\variance{X_n}}} \quad\indist\quad \mathcal{N}(0,1),
\]
as $n \rightarrow \infty$, that is, if $\frac{X_n - \expec{X_n}}{\sqrt{\variance{X_n}}}$ converges in distribution to the standard normal distribution.

Establishing central limit theorems for various graph parameters of $G(n,p)$ has become an important area of research. A variety of tools have been used including the method of moments \cite{K84,KR83,R88}, Stein's method \cite{Barbour1989}, $U$-statistics and related methods \cite{J94,NW88}, and martingale methods \cite{BJKR90,dJ96}. For example, central limit theorems have been established for various types of subgraph counts \cite{GK16, K84, KR83, R88,S19}, for the order of the giant component \cite{BCOK14,BR12,PW05,PW08}
and for the order of the $k$-core \cite{COCKS19,JL08}. 

All of these parameters are in some sense `local', in that they can be determined by looking at bounded size subgraphs of $G(n,p)$, or have natural `local' approximations. More recently, central limit theorems have been proven for more `global' graph parameters, such as the rank \cite{GKSS23} and the matching number \cite{GKSS24,K17} by analysing the \emph{Karp--Sipser leaf-removal process} \cite{KS81}.

Our main result gives a central limit theorem for the circumference $L\left(G\left(n,\frac{c}{n}\right)\right)$, which has a similar `global' nature, in the sparse regime.
In fact, our methods allow us to prove slightly more, and to give quantitative bounds on the variance of $L\left(G\left(n,\frac{c}{n}\right)\right)$.

\begin{theorem} \label{thm:Law_of_longest_cycle}
    Let $c \geq 20$ be a fixed constant and let $G \sim G\left(n,\frac{c}{n}\right)$.
    Then 
    \[
    \frac{L(G)- \expec{L(G)}}{\sqrt{\variance{L(G)}}} \quad\indist\quad \mathcal{N}(0,1)
    \]
    as $n \to \infty$.
Furthermore, the limit $\displaystyle \sigma \coloneqq \lim_{n \to \infty} \left(\frac{\variance{L(G)}}{n}\right)^{\frac{1}{2}}$ exists and there exist constants $C_1,C_2 >0$ such that $C_1 e^{-10c} \leq \sigma \leq C_2 c^{\frac{1}{2}}$.
\end{theorem}

We note that, as in the result of \cite{GKSS24} on a central limit theorem for the matching number of a sparse random graph, we prove our central limit theorem \emph{non-constructively}, in the sense that the best known bounds on $\expec{L\left(G\left(n,\frac{c}{n}\right)\right)}$ are much less accurate than the typical fluctuations of $L\left(G\left(n,\frac{c}{n}\right)\right)$ implied by \cref{thm:Law_of_longest_cycle}. 

For completeness, we also include the next proposition, whose proof is in \cref{a:generalp}, which describes how the law of the circumference behaves for larger~$p$. The logarithms here and in the remainder of the paper are natural logartithms.

\begin{proposition}\label{p:generalp}
Let $\lambda(n)$ be a sequence of real numbers, $c \coloneqq \log n +\log\log n +\lambda(n)$, and let $G \sim G\left(n,\frac{c}{n}\right)$.
\begin{enumerate}[label = \upshape{(\arabic*)}]
\item\label{i:small} If $c \to +\infty$ and $\lambda(n) \to -\infty $, then as $n \to \infty$
\[
    \frac{L(G)- \expec{L(G)}}{\sqrt{\variance{L(G)}}} \quad\indist\quad \mathcal{N}(0,1).
    \]
\item\label{i:medium} If $\lambda(n) \to \lambda \in \mathbb{R}$, then as $n \to \infty$
 \[
  n- L(G)\quad\indist\quad \mathrm{Poisson}\left(e^{-\lambda}\right).\]
  \item\label{i:large} If $\lambda(n) \to \infty$, then whp $L(G)=n$.
\end{enumerate}
\end{proposition}
\smallskip

\subsection{Proof methods}\label{s:methods}
Let us fix some large constant $c$ and let $G \sim G\left(n,\frac{c}{n}\right)$. In \cite{A23} the first author describes an algorithmic process, similar in nature to the \emph{peeling process} used to produce the $4$-core of a graph, which produces a $3$-colouring\footnote{Note, this is not a \emph{proper} colouring of $V(G)$, that is, the colour classes are not necessarily independent.} of $V(G)$, whose colour classes we will later denote by $S$, $P$, and $R$ (we use the colours \emph{sapphire}, \emph{purple}, and \emph{red}, respecitvely).
The set $S$ is a `robust' version of the $4$-core of $G$, the set $P$ is the external neighbourhood of $S$, that is $P \coloneqq N(S)$, and the set $R$ is the remaining set of vertices, that is, $R \coloneqq V(G) \setminus (S \cup P)$. He goes on to show that there is some simple function $\Phi(G)$ which depends only on the induced subgraph $G[P \cup R]$ such that whp $L(G) = n - \Phi(G)$. 

Indeed, the rough idea of \cite{A23} is as follows : It can be shown that $G$ contains no edges from $S$ to $R$, and so given any cycle $\tC \subseteq G$, if we delete the vertices of $\tC$ which lie in $S$, what remains can be split into a vertex-disjoint union of paths, all of whose endpoints lie in $P$, which we call a \emph{family of $P$-paths}. Conversely, it is shown in \cite{A23} that whp the subgraph $G[S \cup P]$ has such a rich structure that we can extend \emph{any} family of $P$-paths in $G[P \cup R]$ into a cycle which contains every vertex of $S \cup P$. Hence, whp the length of the longest cycle in $G$ will be determined by the minimum number of vertices which are left \emph{uncovered} by a family of $P$-paths in $G[P \cup R]$, which is the quantity denoted by $\Phi(G)$. Since this function is determined by the component structure of $G[P \cup R]$ (as a coloured graph), it is relatively easy to show that, when $c \to \infty$, the only non-neglible contribution comes from vertices of degree at most one, which
can be controlled precisely using standard techniques, leading to \cref{p:generalp}. When $c$ is a fixed constant however, we have to take into account all possible components, which requires more novel ideas.

The colouring process that identifies the sets $S,P$ and $R$ will also be key to our proof, and will be described in more detail in \cref{s:strong4core}. It is shown in \cite{A23} that the probability that $L(G)$ is not equal to $n- \Phi(G)=: \tL$ is sufficiently small that it will be sufficient to prove a central limit theorem for $\tL$, which we will do so via a series of approximations.

For the first approximation, since the value of the function $\Phi$ can be split additively over the components of $G[P \cup R]$, we can think of each $v\in V(G)$ as having a contribution to $\Phi$ which is defined by splitting the contribution from each component equally between its vertices, which we denote by $\phi(v)$, so that $\tL = n- \sum_{v \in V(G)} \phi(v)$. 
We then approximate the colouring process \emph{locally} for each $v\in V$, by running a similar algorithm restricted to the $k$-th neighbourhood of $v$, which we denote by $B(v,k)$. We can consider then the contribution of $v$ to the value of $\Phi$ on $B(v,k)$ with respect to this local colouring, which we denote by $\phi_k(v)$. We will show that once $k$ is large enough so that the component $C_v$ of $G[P \cup R]$ that contains $v$ is a subset of $B(v,k-1)$, then $\phi_k(v) = \phi(v)$. 
Hence, if we expect the typical component in $G[P \cup R]$ to be small, then we should expect that as $k \to \infty$, the quanitity $n -\sum_{v \in V(G)} \phi_k(v) \eqqcolon \tL_k$ is a good approximation to $n-\sum_{v\in V(G)} \phi(v) = \tL$.

One of the key ideas in the paper will be to show that we can control quite precisely how the typical (global and local) colourings change when we resample an edge in $G\left(n,\frac{c}{n}\right)$. This will allow us to use the Efron--Stein inequality (\cref{thm:EfronStein}) to bound the variance of $\tL - \tL_k$, after which it will suffice to prove a central limit theorem for $\tL_k$.

In order to prove a central limit theorem for $\tL_k$, it will be useful to rewrite $\tL_k$ as follows. Since $\phi_k(v)$ is determined by $G[B(v,k)]$, we can view the random variable $\tL_k$ as a weighted sum of \emph{$k$-th neighbourhood counts}. Let $\mathcal{H}$ be the set of rooted graphs $(H,r)$ of \emph{radius} at most $k$, that is, such that $\dist_H(r,v) \leq k$ for all $v \in V(H)$. Then for every $(H,r) \in \mathcal{H}$, there is some constant $\alpha_{(H,r)}$ such that if $(G[B_G(v,k)],v)$ and $(H,r)$ are isomorphic as rooted graphs, then $1-\phi_k(v) = \alpha_{(H,r)}$. Hence, we can rewrite
\begin{align}
\tL_k &\coloneqq n -\sum_{v \in V(G)} \phi_k(v) = \sum_{v \in V(G)} (1 - \phi_k(v))  = \sum_{v\in V(G)} \sum_{(H,r) \in \mathcal{H}} \alpha_{(H,r)}\mathds{1}_{\{(G[B_{G}(v,k)],v) \cong (H,r)\}} \nonumber\\
&= \sum_{(H,r) \in \mathcal{H}} \alpha_{(H,r)} \sum_{v\in V(G)} \mathds{1}_{\{(G[B_{G}(v,k)],v) \cong (H,r)\}}\nonumber \\
&=  \sum_{(H,r) \in \mathcal{H}} \alpha_{(H,r)} \cdot \# \{ \text{$k$-th neighbourhoods isomorphic to $(H,r)$}\}. \label{eq:ballcount}
\end{align}

It is known, see for example \cite{R88}, that for a fixed graph $H$, the number of copies of $H$ in $G(n,p)$ satisfies a central limit theorem when $p$ is sufficiently large and not too close to one, more concretely when $n^2(1-p) \to \infty$ and $np^{m_H} \to \infty$ where
\[
m_H ~\coloneqq~ \max \left\{ \frac{e(H')}{v(H')} \colon H' \subsetneq H  \right\}
\]
is the \emph{maximum density} of $H$.
In fact, it can even be shown (see, e.g., \cite{J94}) that for any finite collections of graphs which simultaneously satisfy the above conditions, the corresponding subgraph counts satisfy a multidimensional central limit theorem, and hence, any linear combination of them satisfies a central limit theorem. Hence, it is not unreasonable to think that something similar might hold for the weighted sum of neighbourhood counts in \eqref{eq:ballcount}.

However, there are some difficulties here. Firstly the sum in \eqref{eq:ballcount} is not over a bounded number of rooted graphs,
and secondly there are rooted graphs $(H,r)$ which will contribute to this sum for which we should not expect the neighbourhood counts to satisfy a central limit theorem in this regime of $p$ as their maximum density $m_H$ is too large, for example whenever $H$ is not a tree.

For this reason we make a second approximation. Letting $\cT$ be the set of rooted trees $(T,r) \in \mathcal{H}$ with $1 \leq v(T) \leq \tbound$ we consider
\begin{equation}
\hL_k ~\coloneqq~ \sum_{(T,r) \in \mathcal{T}} \alpha_{(T,r)} \cdot \# \{ \text{$k$-th neighbourhoods isomorphic to $(T,r)$}\}. \label{eq:ballcounttree}
\end{equation}
It is reasonable to think that $\hL_k$ and $\tL_k$ are close, since our choice of $p = \frac{c}{n}$ means that for each $k \in \mathbb{N}$ and $v \in V(G)$ it is very likely that $(G[B(v,k)],v) \in \mathcal{T}$. Indeed, standard results imply that whp $G$ contains no cycle of length at most $k$, in which case $G[B_G(v,k)]$ is a tree, and since the expected degree of a vertex in $G$ is $\approx c$, the probability that $|B(v,k)|$ is significantly larger than $c^k$ will tend to $0$ as a function of $k$. Using this, together with an analysis of how $\tL_k$ and $\hL_k$ change under resampling edges, we will again be able to use the Efron--Stein inequality to bound the variance of $\tL_k - \hL_k$, after which it will suffice to prove a central limit theorem for $\hL_k$.

Finally, we will be able to show, using a version of Stein's method developed by Barbour, Karo\'nski and Ruci\'nski \cite{Barbour1989}, that $\hL_k$ satisfies a central limit theorem as $n \to \infty$ for any fixed $k \in \mathbb{N}$, from which we will be able to deduce a central limit theorem for the circumference (\cref{thm:Law_of_longest_cycle}) using our sequence of approximations.

However, a small technical issue with the application of Stein's method is that we will require a lower bound on the variance of $\hL_k$, which turns out to be deceptively difficult to prove. In order to do so we will first prove a lower bound on the variance of $\tL_k$ using a delicate partial exposure argument (see \cref{s:varlowerbound}), which we will be able to use to give a sufficient bound on the variance of $\hL_k$, as well as to give the quantitative lower bound on the variance of $L(G)$ in \cref{thm:Law_of_longest_cycle} (the upper bound is much simpler to show).

We note that these last two steps, approximating the weighted neighbourhood sum $\tL_k$ by an appropriate truncation $\widehat{L}_k$ and showing that this truncation satisfies a central limit theorem, do not require much about the specific form of $\tL_k$ beyond that it is a weighted neighbourhood sum whose variance is not too small.

\subsection{Main contributions of the paper}
Our aim in this paper is not just to prove a central limit theorem for the circumference (\cref{thm:Law_of_longest_cycle}), but also to demonstrate a novel method for proving central limit theorems for graph parameters in sparse random graphs. We believe this method will have applications more generally. The rough idea is to take a series of finer and finer \emph{local} appromixations to the parameter, each of which can be shown to satisfy a central limit theorem using Stein's Method, and then to transfer this information `through the limit' to the original parameter. This is achieved using the Efron-Stein inequality (\cref{thm:EfronStein}), which is one of the main novelties in the paper. The Efron--Stein inequality allows us to reduce this problem to a more combinatorial question about the \emph{robustness} of these local approximations under the \emph{resampling} of edges, that is, how do these local approximations change if we add or delete an edge. 
This question is addressed then in \cref{s:strong4coreproperties}. 

However, deducing a central limit theorem for these local approximations is also not entirely straightforward. In particular, one difficulty that arises here is that proving the lower bound on the variance required by Stein's method through direct computation does not appear feasible. 
This is probably typical for any application that is not amenable to simpler methods. To overcome this difficulty in our case, we show in \cref{s:varlowerbound} how we can obtain the desired lower bound through a subtle argument that involves revealing the random edges over several rounds.

\subsection{Structure of the paper}
The paper is structured as follows. In the next section we define precisely the global and local colouring process and our sequence of approximations and give a detailed outline of the proof of \cref{thm:Law_of_longest_cycle}. Then, in \cref{s:prelim} we introduce various notation and tools that we will use throughout the paper. In \cref{s:strong4coreproperties} we prove some results about how the global and local colourings are affected by resampling an edge. In \cref{s:ESvariance} we apply the Efron--Stein inequality, using the information from \cref{s:strong4coreproperties}, to bound the variances of the differences in our approximation steps. \cref{s:varlowerbound} is the dedicated to proving a lower bound on the variance of $\hL_k$. Finally, in \cref{s:Lkprop} we give a central limit theorem for $\hL_k$, together with some quantitative analysis of its variance which will imply the quantitative aspects of \cref{thm:Law_of_longest_cycle}. We finish in \cref{s:discussion} with a general discussion of our methods and some open problems and directions for future research.

\section{Proof of Theorem \ref{thm:Law_of_longest_cycle}}\label{s:Proof}
In this section we will formally define the three approximations we make for $L\left(G\left(n,\frac{c}{n}\right)\right)$ and give a detailed outline of the proof of \cref{thm:Law_of_longest_cycle}. During this section we will state without proof the auxiliary results (\Cref{lem:var_L_k_hL_k_diff_bound,lem:var_X_X_k_diff_bound,lem:conv_by_stein,lem:variancelimit}) which will allow us to prove \Cref{thm:Law_of_longest_cycle} and defer their proofs until later in the paper.

\subsection{The strong \texorpdfstring{$4$}{4}-core and a local approximation}\label{s:strong4core}
Let $G$ be a graph. For $A \subseteq V(G)$, we denote by $N_G(A)$ the \emph{external neighbourhood of $A$ in $G$}, that is, $N_G(A) \coloneqq \{x \in V(G) \setminus A \colon xy \in E(G) \text{ for some } y \in A\}$. For $x \in V(G)$, we also write $N_G(x)$ for $N_G(\{x\})$.
For a subgraph $H$ of $G$ and a set $A \subseteq V(G)$, we let $H - A$ be the subgraph of $H$ obtained by deleting all vertices in $V(H) \cap A$.
For $x, y \in V(G)$, we denote by $\dist_G(x,y)$ the \emph{graph-distance in $G$ from $x$ to $y$}, that is, the length of a shortest path from $x$ to $y$ (or $\infty$ if no such path exists). For $x \in V(G)$ and $k \in \mathbb{N}$, we let $\partial_G(x,k) \coloneqq \{y \in V(G) \colon \dist_G(x,y) = k\}$ and $B_{G}(x,k) \coloneqq \{y \in V(G) \colon \dist_G(x,y) \leq k\}$. 

We say that a set $A \subseteq V(G)$ has the \emph{strong $4$-core property for $G$} if 
\begin{align} \label{Strong_4-core_property}
    |N_G(v) \cap A| \geq 4 \quad \text{ for every } v \in A \cup N_G(A). 
\end{align}
We call the maximal set $A \subseteq V(G)$ that has the strong $4$-core property for $G$ the \emph{strong $4$-core of~$G$} and denote it by $S(G)$. Note that $S(G)$ is well-defined, since if two sets $A_1$ and $A_2$ have the strong $4$-core property for $G$, then so does $A_1 \cup A_2$.
Moreover, we define 
\[
    P(G) \coloneqq N_G(S(G)) \quad \text{ and } \quad R(G) \coloneqq V(G) \setminus (S(G) \cup P(G)). 
\]
If $G$ is clear from context, we simply write $S$, $P$, and $R$ for $S(G)$, $P(G)$, and $R(G)$, respectively. Inspired by the work in \cite{AF21}, the first author introduced the strong $4$-core to study long cycles in sparse random graphs \cite{A23}.

The strong $4$-core can also be obtained by a colouring process as follows.\footnote{We will call this the \emph{global colouring process}.}

\vbox{
\begin{breakablealgorithm}
    \caption{Global colouring process (\cite[Algorithm 1]{A23})} \label{alg:strong_4_core_alg}
\begin{algorithmic}[1] 
    \MyInput{A graph $G$}
    \MyOutput{$S(G)$, the strong $4$-core of $G$}
    \State Initiate the colour of all vertices to be sapphire.
    \State Set $t \coloneqq 1$.
    \While{there exists a sapphire or purple vertex $x$ with fewer than $4$ sapphire neighbours}
        \State Colour $x$ red and all its sapphire neighbours purple.
        \State Set $v_t \coloneqq x$ and $t \coloneqq t+1$.
    \EndWhile
    \State \Return the set of sapphire vertices.
\end{algorithmic}
\end{breakablealgorithm} 
}

Note that the final colouring of $G$ is determined by the sequence $v_1,v_2,\ldots,$ of vertices which are coloured red. Let us show that \cref{alg:strong_4_core_alg} does indeed return the strong $4$-core of $G$.
Let $S'$ be the output of \cref{alg:strong_4_core_alg}. 
Note that any vertex in the strong $4$-core $S$ will never be recoloured. Indeed, suppose for a contradiction that $x$ is the first vertex in $S$ that is recoloured. Then just before $x$ is recoloured either $x$ or one of its neighbours has fewer than $4$ sapphire neighbours, a contradiction to the fact that $S$ has the strong $4$-core property \cref{Strong_4-core_property} for $G$.  
Thus, $S \subseteq S'$. Moreover, since at any step in the algorithm the neighbourhood of the set of sapphire vertices is the set of purple vertices, which can be shown inductively, $S'$ satisfies the strong $4$-core property for $G$ and thus $S' \subseteq S$. Hence, $S' = S$.

We will on occasion make use of this alternative way of finding the strong $4$-core to prove a property about $S$, $P$, or $R$. 
In view of \cref{alg:strong_4_core_alg}, we call vertices in $S$, $P$, and $R$, {\em sapphire}, {\em purple}, and {\em red} vertices, respectively (see \cref{fig:4core}). This means that these expressions will mean, from now on, the colour of the vertices at the end of the global colouring process, unless otherwise stated. 

\begin{figure}[htp]
    \centering
\begin{tikzpicture}[scale=1]
\node at (-0.5,4) {$G$};
\filldraw[color=blue!100, fill=blue!5, rounded corners, very thick](0,1)   rectangle (6,2.5);
\filldraw[color=purple!100, fill=purple!5, rounded corners, very thick](0,2.6)   rectangle (6,3.5);
\filldraw[color=red!100, fill=red!5, rounded corners, very thick](0,3.6)   rectangle (6,4.2);

\draw[thick] (2,2)--(1.75,1.5);
\draw[thick] (2,2)--(1.5,1.5);
\draw[thick] (2,2)--(2.25,1.5);
\draw[thick] (2,2)--(2.5,1.5);
\draw [fill=blue,draw=blue] (1.75,1.5) circle [radius=.07];
\draw [fill=blue,draw=blue] (1.5,1.5) circle [radius=.07];
\draw [fill=blue,draw=blue] (2.25,1.5) circle [radius=.07];
\draw [fill=blue,draw=blue] (2.5,1.5) circle [radius=.07];
\draw node at (2,1.6) {...};

\draw[thick] (4,3)--(3.75,2);
\draw[thick] (4,3)--(3.5,2);
\draw[thick] (4,3)--(4.25,2);
\draw[thick] (4,3)--(4.5,2);
\draw[thick] (4,3)--(4,3.9);

\draw node at (4,2.2) {...};

\draw [decorate,decoration={brace,amplitude=3pt,mirror,raise=3pt},yshift=0pt] (6.05,3.6) -- (6.05,4.2) node [black,midway,xshift=0.5cm] {\footnotesize $R$};
\draw[text width=3cm] (8.55,3.9) node {\footnotesize (red vertices)};
\draw [decorate,decoration={brace,amplitude=3pt,mirror,raise=3pt},yshift=0pt] (6.05,2.6) -- (6.05,3.5) node [black,midway,xshift=0.5cm] {\footnotesize $P$};
\draw[text width=3cm] (8.55,3.05) node  {\footnotesize (purple vertices)};
\draw [decorate,decoration={brace,amplitude=3pt,mirror,raise=3pt},yshift=0pt] (6.05,1) -- (6.05,2.5) node [black,midway,xshift=0.5cm] {{\footnotesize $S$}};
\draw[text width=3cm] (8.55,1.75) node  {\footnotesize (sapphire vertices)};
\draw [fill=blue,draw=blue] (2,2) circle [radius=.07];
\draw [fill=purple,draw=purple] (4,3) circle [radius=.07];
\draw [fill=red,draw=red] (4,3.9) circle [radius=.07];
\draw [fill=blue,draw=blue] (3.75,2) circle [radius=.07];
\draw [fill=blue,draw=blue] (3.5,2) circle [radius=.07];
\draw [fill=blue,draw=blue] (4.25,2) circle [radius=.07];
\draw [fill=blue,draw=blue] (4.5,2) circle [radius=.07];
\end{tikzpicture}
 \caption{The $3$-colouring of a graph $G$ given by \cref{alg:strong_4_core_alg}. Note that every vertex in $S\cup P$ has at least $4$ neighbours in $S$ and there is no edge between $R$ and $S$.}
   \label{fig:4core}
\end{figure}
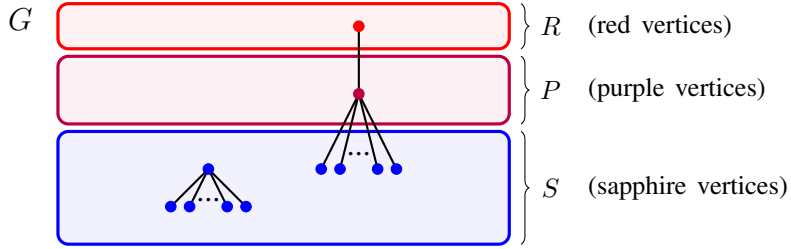

Given a graph $H$ and $W \subseteq V(H)$, we define $\uc(H;W)$ to be the minimal number of vertices in $V(H) \setminus W$ left \emph{uncovered} by a set $\cP$ of vertex-disjoint paths in $H$ all of whose endpoints are in $W$, which we call a \emph{family of $W$-paths in $H$}. Note that, by adding single vertices of $W$ (as paths of length $0$) to $\cP$ if needed, we may assume that $\cP$ covers all vertices in $W$. 
Let 
\[
    \Phi(G) ~\coloneqq~ \uc(G[P \cup R]; P).
\]
For any cycle $\tC$ in $G$ which contains a vertex of $S$, the components of $\tC - S$ form a family of $P$-paths in $G[P \cup R]$ and thus $\tC$ leaves at least $\Phi(G)$ vertices of $G$ uncovered. 
Perhaps surprisingly, the first author showed in \cite{A23} that in $G\sim G(n,p)$ this is essentially tight, by showing that with very high probability the longest cycle in $G$ covers all but $\Phi(G)$ many vertices, provided that $np\geq 20$. So, let us define $$\tL(G) ~\coloneqq~  n - \Phi(G),$$
where when $G$ is clear from context, we write $L$ and $\tL$ for $L(G)$ and $\tL(G)$, respectively.
\begin{theorem}[\cite{A23}, Theorem 11] \label{thm:L_tL_equal}
Let $p \coloneqq p(n) \in [0,1]$ be such that $20 \leq pn \leq 2 \log n$ and $G\sim G(n,p)$. Then, the equality $L(G) = \tL(G)$ holds with probability at least $1 - O(n^{-2})$.
\end{theorem} 

Since on occasion it will be convenient to talk about appropriate \emph{rescalings} of these random variables, let us also define
\[
 X(G) ~\coloneqq~ \frac{L(G) - \expec{L(G)}}{n^{\frac{1}{2}}} \quad \text{ and } \quad   \tX(G) ~\coloneqq~ \frac{\tL(G) - \expec{\tL(G)}}{n^{\frac{1}{2}}}.
\]
If $G$ is clear from context, we write $X$ and $\tX$ for $X(G)$ and $\tX(G)$, respectively. We note that an immediate consequence of \Cref{thm:L_tL_equal} is that $X - \tX \inprob 0$ as $n \to \infty$ for large enough $c$.

\begin{corollary}\label{cor:tXvsX}
For any $c \geq 20$, $\eps > 0$ and $\delta >0$ there exists $n_0\coloneqq n_0(c,\eps,\delta)$ such that for all $n \geq n_0$ the following holds. Let $G\sim G(n,\frac{c}{n})$, then
\[
\prob{|\tX(G) - X(G)| > \delta} \leq \frac{\eps}{6}.
\]
\end{corollary}
\begin{proof}
Let $L=L(G)$ and $\tL=\tL(G)$. 
By \cref{thm:L_tL_equal}, we have
        \begin{align*}
            \prob{\left|\tX - X \right| > \delta} &= \prob{\left|\tL - L - \expec{\tL-L}\right| > \delta n^\frac{1}{2}} \leq \prob{\left|\tL - L\right| + \expec{\left|\tL-L\right|}  > \delta n^\frac{1}{2}} \\
            &\leq \prob{\left|\tL - L\right|  > \frac{\delta n^\frac{1}{2}}{2}} \leq \prob{L \neq \tL} = O(n^{-2}) \leq \frac{\eps}{6},
        \end{align*}
        where in the second inequality we used that $\expec{\left|\tL-L\right|} \leq n \prob{L \neq \tL} = O(n^{-1}) \leq \frac{\delta n^\frac{1}{2}}{2}$. 
\end{proof}

Since we want to approximate $\Phi(G)$ by considering $G$ \textit{locally} from the point of view of each vertex, it makes sense to rewrite $\Phi(G)$ as a sum over $V(G)$ as follows. 
First, we note that the minimisation giving $\Phi(G)$ can be done separately for each connected component of $G[P \cup R]$. Indeed, $\Phi(G) = \sum_C \uc(G[C];P\cap C)$ where the sum is taken over all components $C$ of $G[P \cup R]$.  
For a vertex $v \in P \cup R$, we let $C_v(G)$ be the connected component of $G[P \cup R]$ that contains $v$, where we write $C_v$ for $C_v(G)$ if $G$ is clear from context. For $v \in S$, we set $C_v \coloneqq \varnothing$.  We can now write 
\[
    \Phi(G) ~=~ \sum_{v \in V(G)} \phi_G(v),
\]
where 
\[
    \phi_G(v) ~\coloneqq~ 0, \text{ if } v \in S \quad\text{ and } \quad \phi_G(v) ~\coloneqq~ \frac{\uc(G[C_v];P \cap C_v)}{|C_v|}, \text{ if } v \in P \cup R.
\]

Our aim is now to find a local approximation for $\phi_G(v)$.
To that end, we define the local versions of $S, P, R,$ and $\phi_G(v)$ as follows.
Let $v \in V(G)$ and $k \in \mathbb{N}$.
Let $A \subseteq B_G(v,k-1)$ be the maximal set such that 
\begin{align}\label{e:localstrong4core}
    |N_G(v) \cap (A \cup \partial_G(v,k))| \geq 4 \quad \text{ for every } v \in (A \cup N_G(A)) \setminus \partial_G(v, k).
\end{align}
We then define the \emph{$k$-local strong $4$-core} $S_k(G,v)$ to be $A \cup \partial_{G}(v,k)$. 

We note that $S_k(G,v)$ can also be obtained through a colouring process on $B_G(v,k)$ as follows.\footnote{We will call this the \emph{local colouring process}.}

\vbox{
\begin{breakablealgorithm}
    \caption{Local colouring process} \label{alg:strong_4_core_alg_local}
\begin{algorithmic}[1] 
    \MyInput{A graph $G$, a vertex $v \in V(G)$, and an integer $k \geq 1$.}
    \MyOutput{$S_k(G,v)$, the $k$-local strong $4$-core of $G$ at $v$.}
    \State Initiate the colour of all vertices in $B_G(v,k)$ to be sapphire.
    \State Set $t \coloneqq 1$.
    \While{there exists a sapphire or purple vertex $x \in B_G(v,k-1)$ with fewer than $4$ sapphire \\ \hspace*{10.6mm} neighbours}
        \State Colour $x$ red and all its sapphire neighbours in $B_G(v,k-1)$ purple.
        \State Set $v_t \coloneqq x$ and $t \coloneqq t+1$.
    \EndWhile
    \State \Return the set of sapphire vertices in $B_G(v,k)$.
\end{algorithmic}
\end{breakablealgorithm} 
}

We further define 
\begin{align*}
    P_k(G, v) &~\coloneqq~ N_G\Big(S_k(G,v) \setminus \partial_G(v,k)\Big) \setminus \partial_G(v,k) \quad \text{and} \\ \quad R_k(G,v) &~\coloneqq~ B_G(v,k) \setminus \Big(S_k(G,v) \cup P_k(G,v)\Big).
\end{align*}
We can think of this as a `local colouring' of $B_G(v,k)$, where vertices in $S_k(G,v)$ are coloured `sapphire', vertices in $P_k(G,v)$ are coloured `purple', and vertices in $R_k(G,v)$ are coloured `red'. As with \cref{alg:strong_4_core_alg}, we note that the `local colouring' is determined by the sequence $v_1,v_2,\ldots$ of vertices which are coloured red. 

We will later show (\Cref{obs:C_small}) that whenever the component of $G[P \cup R]$ containing $v$ is a subset of $B_G(v,k-1)$, then it will agree with the component of $G[P_k(G,v) \cup R_k(G,v)]$ containing $v$ and the global and local colourings on this component will therefore be the same.  In this way, we can think of the local colouring as giving a local approximation to the global colouring at $v$ as $k \to \infty$.

If $v \in P_k(G,v) \cup R_k(G,v)$, we define $C_v^k(G)$ to be the component of $G[P_k(G,v) \cup R_k(G,v)]$ that contains $v$. If $v \in S_k(G,v)$, we set $C_v^{k}(G) \coloneqq \varnothing$. We also write $C_v^k$ for $C_v^k(G)$ if $G$ is clear from context. Define $\phi_{G,k}$ by 
\begin{align*}
    \phi_{G,k}(v) ~\coloneqq~ 0 \text{ if } v \in S_k(G,v) \quad \text{ and } \quad \phi_{G,k}(v) ~\coloneqq~ \frac{\uc(G[C_v^k]; P_k(G,v))}{|C_v^k|}, \text{ otherwise}.
\end{align*}

As above, \Cref{obs:C_small} will imply that $\phi_{G,k}(v) = \phi_G(v)$ whenever the component of $G[P \cup R]$ containing $v$ is contained in $B_G(v,k)$, and so we should expect that, for large $k$, 
\[\tL_k(G) ~\coloneqq~ n - \sum_{v \in V(G)} \phi_{G,k}(v)
\]
will be a good approximation to $\tL(G)   = n-\sum_{v \in V(G)} \phi_G(v)$, which itself should be a good approximation to $L(G)$. Let us also define the rescaling
\[
\tX_k (G) ~\coloneqq~ \frac{\tL_k(G) - \expec{\tL_k(G)}}{n^{\frac{1}{2}}}.
\]
When $G$ is clear from the context we write $\tL_k$ and $\tX_k$ for $\tL_k(G)$ and $\tX_k(G)$, respectively. 
A key step towards our result will be to bound the variance of $\tX - \tX_k$, by using the Efron--Stein inequality (\Cref{thm:EfronStein}). For this step we will need good control over the effect of resampling an edge on $\tX - \tX_k$. We will show the following in \cref{s:ESvariance}.
\begin{lemma} \label{lem:var_X_X_k_diff_bound}
    Let $k \geq 100$ and $c \geq 20$ be fixed. Let $G\sim G(n,\frac{c}{n})$. Then, 
    \[
        \variance{\tX(G)-\tX_k(G)} \leq 64ck^2 \left(\frac{3}{4}\right)^{k-2} + o(1),
    \]
    as $n \rightarrow \infty$.
\end{lemma}

As with \Cref{thm:L_tL_equal}, it is a simple consequence of \Cref{lem:var_X_X_k_diff_bound} that $\tX-\tX_k \inprob 0$ as $k,n \to \infty$, in the sense that for every $\delta >0$,
\[
\lim_{k \to \infty} \lim_{n \to \infty} \prob{|\tX - \tX_k| > \delta} = 0.
\]
\begin{corollary}\label{cor:tXkvstX}
For any $c \geq 20$ and $\eps > 0$ there exist $\delta \coloneqq  \delta(c,\eps)$ and $k_0 \coloneqq  k_0(c,\eps,\delta)$ such that for all $k \geq k_0$ there exists $n_0\coloneqq n_0(c,\eps,\delta,k)$ such that for all $n \geq n_0$ the following holds. Let $G\sim G(n,\frac{c}{n})$, then
\[
\prob{|\tX(G) - \tX_k(G)| > \delta} \leq \frac{\eps}{6}.
\]
\end{corollary}
\begin{proof}
By Chebyshev's inequality and \cref{lem:var_X_X_k_diff_bound}, we have
        \begin{align*}
            \prob{|\tX - \tX_k| > \delta} \leq \frac{\variance{\tX-\tX_k}}{\delta^2} \leq \frac{64ck^2 \left(\frac{3}{4}\right)^{k-2} + o(1)}{\delta^2} \leq \frac{\eps}{6}.
        \end{align*}
\end{proof}
\subsection{Approximation by neighbourhood counts}\label{s:neighbourhoodcounts}

As in \cref{s:methods}, we now define an approximation to $\tL_k$ in terms of a weighted sum of neighbourhood counts. A \emph{rooted graph} is a pair $(H,r)$ where $H$ is a graph and $r \in V(H)$. Two rooted graphs $(H_1,r_1)$ and $(H_2,r_2)$ are \emph{isomorphic} (written $(H_1,r_1) \cong (H_2,r_2)$) if there exists a bijection $f \colon V(H_1) \rightarrow V(H_2)$ such that $f(r_1) = r_2$ and $f(u)f(v) \in E(H_2)$ if and only if $uv \in E(H_1)$ for all $u,v \in V(H_1)$. 

Let $\mathcal{H}$ be the set of rooted graphs $(H,r)$ of radius at most $k$. 
Note that 
\[
   \tL_k(G) ~\coloneqq~ \sum_{v \in V(G)} (1-\phi_{G,k}(v)) = \sum_{v \in V(G)} \sum_{(H,r) \in \mathcal{H}} \alpha_{(H,r)} \mathds{1}_{\{(G[B_G(v,k)],v) \cong (H,r)\}}
\] 
where $\alpha_{(H,r)} \in [0,1]$ is a constant depending only on $(H,r)$. 

We approximate $\tL_k(G)$ by the following `truncated' version.
Define 
\begin{align} \label{eq:t_k_def}
    t_k ~\coloneqq~ \tbound
\end{align}
and 
\[
    \hL_k(G) ~\coloneqq~ \sum_{v \in V(G)}\sum_{(T,r) \in \mathcal{T}} \alpha_{(T,r)} \mathds{1}_{\{(G[B_G(v,k)],v) \cong (T,r)\}},
\]
where $\mathcal{T}$ is the set of rooted trees $(T,r) \in \mathcal{H}$ with $1 \leq v(T) \leq t_k$. We will sometimes also consider the rescaled analogue of $\hL_k$
\[
\widehat{X}_k(G) ~\coloneqq~  \frac{\hL_k(G) - \expec{\hL_k(G)}}{n^{\frac{1}{2}}}
\]
When $G$ is clear from the context we write $\hL_k$ and $\widehat{X}_k$ for $\hL_k(G)$ and $\widehat{X}_k(G)$, respectively. 

As with \Cref{lem:var_X_X_k_diff_bound}, we use the Efron--Stein inequality (\Cref{thm:EfronStein}) to bound the variance of $\tL_k - \hL_k$, which we can then use to show that $\tX_k - \widehat{X}_k \inprob 0$ as $k,n \to \infty$, in the sense that for all $\delta >0$,
\[
\lim_{k \to \infty} \lim_{n \to \infty} \prob{|\tX_k - \widehat{X}_k| > \delta} = 0.
\]
We will show the following in \cref{s:ESvariance}.
\begin{lemma} \label{lem:var_L_k_hL_k_diff_bound}
   Let $k \geq 100$ and $c \geq 20$ be fixed. Let $G\sim G(n,\frac{c}{n})$. Then,
    \begin{align*}
        \variance{\tX_k(G) - \hX_k(G)} \leq \exp\left(-c^2k^2\right) + O\left(\frac{1}{n}\right),
    \end{align*}
    as $n \rightarrow \infty$.
\end{lemma}

\begin{corollary}\label{cor:wideXvstXk}
For any $c \geq 20$ and $\eps > 0$ there exist $\delta \coloneqq  \delta(c,\eps)$ and $k_0 \coloneqq  k_0(c,\eps,\delta)$ such that for all $k \geq k_0$ there exists $n_0\coloneqq n_0(c,\eps,\delta,k)$ such that for all $n \geq n_0$ the following holds. Let $G\sim G(n,\frac{c}{n})$, then
\[
\prob{|\tX_k(G) - \widehat{X}_k(G)| > \delta} \leq \frac{\eps}{6}.
\]
\end{corollary}
\begin{proof}
By Chebyshev's inequality and \cref{lem:var_L_k_hL_k_diff_bound},
        \begin{align*}
            \prob{|\widehat{X}_k-\tX_k| > \delta} \leq \frac{\variance{\widehat{X}_k-\tX_k}}{\delta^2}  \leq \frac{\exp\left(-c^2k^2\right)}{\delta^2} + O\left(\frac{1}{n}\right) \leq \frac{\eps}{6}.
        \end{align*}
\end{proof}

Finally, in \cref{sec:normalapprox} we will use Stein's method to show that $\hL_k$ satisfies a central limit theorem as $n \to \infty$.

\begin{lemma} \label{lem:conv_by_stein}
Let $k \geq 100$ and $c \geq 20$ be fixed and $G\sim G(n,\frac{c}{n})$. Then,
    \[
        \frac{\hL_k(G) - \expec{\hL_k(G)}}{\left(\variance{\hL_k(G)}\right)^{\frac{1}{2}}} \quad\indist\quad \mathcal{N}(0,1) \quad \text{ as } n \to \infty.
    \]
\end{lemma}
However, in order to apply Stein's method, we will require a lower bound on $\variance{\hL_k}$ -- for Stein's method, $\variance{\hL_k} = \omega(n^{\frac{2}{3}})$ would suffice. Furthermore, in order to prove the quantitative aspects of \cref{thm:Law_of_longest_cycle}, it will be important to have some control over the sequence $\sigma_k \coloneqq  \lim_{n \to \infty} \left(\frac{\variance{\hL_k}}{n}\right)^{\frac{1}{2}}$, which we summarise in the following lemma, whose proof is given in \cref{sec:variancelimit}:

\begin{lemma}\label{lem:variancelimit}
Let $c \geq 20$ and $G\sim G(n,\frac{c}{n})$. Then, there exist constants $C_1,C_2 >0$ and $k_0(c) \geq 100$ such that

\vbox{
\begin{enumerate}[label = \upshape{(L\arabic*)}, leftmargin= \widthof{R100000}]
    \item the limit $\sigma_k \coloneqq \lim_{n \to \infty} \left(\frac{\variance{\hL_k(G)}}{n}\right)^{\frac{1}{2}}$ exists for all fixed $k \geq k_0(c)$; \label{item:limit_in_n}
    \item the limit $\sigma \coloneqq \lim_{k \to \infty} \sigma_k$ exists and $\sigma = \lim_{n \to \infty} \left(\frac{\variance{L(G)}}{n}\right)^{\frac{1}{2}}$; and \label{item:limit_in_k}
    \item $C_1 e^{-10c} \leq \sigma  \leq C_2 c^{\frac{1}{2}}$. \label{item:limit_bounds}
\end{enumerate}
}
\end{lemma}

\subsection{Proof of Theorem \ref{thm:Law_of_longest_cycle}}
Fix $c \geq 20$ and recall that we are considering
\[
X(G) ~\coloneqq~ \frac{L(G) - \expec{L(G)}}{n^{\frac{1}{2}}}
\]
where $G \sim G\left(n,\frac{c}{n}\right)$. 

Let us define the approximations $\tL$, $\tL_k$ and $\hL_k$ and their rescalings $\tX$, $\tX_k$ and $\widehat{X}_k$ as in \Cref{s:strong4core,s:neighbourhoodcounts}. Let $\sigma \coloneqq \lim_{k \to \infty} \sigma_k$ be as in \Cref{lem:variancelimit} and let $Z\sim \mathcal{N}(0, 1)$ be a standard normal random variable. Our aim is to show that $X$ converges in distribution to $\sigma Z \sim \mathcal{N}(0, \sigma^2)$, that is,
\begin{equation*}
\text{$\lim_{n\rightarrow \infty}\prob{X \leq x} = \prob{\sigma Z \leq x}$ for all $x \in \mathbb{R}$.}
\end{equation*}
To that end, let $\eps > 0$ and $x \in \mathbb{R}$. It suffices to show that for large enough $n$, 
    \begin{align} \label{eq:goal}
        \prob{\sigma Z \leq x} - \eps \leq \prob{X \leq x} \leq \prob{\sigma Z \leq x} + \eps.
    \end{align}
    Let $\delta \coloneqq \delta(c,\eps)$ be small enough. Let $k_0=k_0(c,\eps,\delta)$ be large enough and let $k \geq k_0(c, \eps, \delta)$. Let $n_0=n_0(c,\eps,\delta,k)$ be large enough and let $n \geq n_0(c,\eps,\delta,k)$.

 By \cref{lem:conv_by_stein}, 
 \[
 \frac{\hL_k - \expec{\hL_k}}{\left(\variance{\hL_k}\right)^{\frac{1}{2}}} \quad\indist\quad Z \quad \text{ as } n \rightarrow \infty.
 \]
 Since by \Cref{lem:variancelimit} \cref{item:limit_in_n}, $\lim_{n \rightarrow \infty} \left(\frac{\variance{\hL_k}}{n}\right)^{\frac{1}{2}} = \sigma_k$, we have by Slutsky's theorem (\cref{thm:Slutsky}), that
 \[
 \widehat{X}_k \quad\indist\quad \sigma_k Z \quad \text{ as } n \rightarrow \infty.
\] 
Moreover, since by \Cref{lem:variancelimit} \cref{item:limit_in_k}, $\lim_{k \rightarrow \infty} \sigma_k = \sigma > 0$, we have, again by Slutsky's theorem (\cref{thm:Slutsky}), that 
\[
\sigma_k Z \quad\indist\quad \sigma Z \quad \text{ as } k \rightarrow \infty.
\]

Together with the continuity of the function $x \mapsto \prob{\sigma Z \leq x}$, and our choice of $\delta, k_0$ and $n_0$, it follows that
    \begin{align*}
        \prob{\widehat{X}_k \leq x + 3 \delta} \leq \prob{\sigma_k Z \leq x + 3 \delta} + \frac{\eps}{4} 
        \leq \prob{\sigma Z \leq x + 3 \delta} + \frac{\eps}{3}
        \leq \prob{\sigma Z \leq x} + \frac{\eps}{2}.
    \end{align*}

    Finally, note that if $X \leq x$, then either $\widehat{X}_k \leq x + 3 \delta$, or $\max \{|\widehat{X}_k-X_k|,|X_k - \tX|,|\tX - X|\} > \delta$. Thus, together with \Cref{cor:tXvsX,cor:tXkvstX,cor:wideXvstXk}, it follows that
    \begin{align*}
        \prob{X \leq x} &\leq \prob{\widehat{X}_k \leq x + 3 \delta} + \prob{|\widehat{X}_k-X_k| > \delta} + \prob{|X_k - \tX| > \delta} + \prob{|\tX - X| > \delta} \\
        &\leq \prob{\sigma Z \leq x} + \eps.
    \end{align*}
    This shows the upper bound in \cref{eq:goal}. The lower bound is shown in an analogous way, completing the proof of \cref{eq:goal}.

    In particular, by \Cref{lem:variancelimit} \cref{item:limit_in_k} and Slutsky's theorem (\cref{thm:Slutsky}) we have
\[
    \frac{L(G) - \expec{L(G)}}{\sqrt{\variance{L(G)}}} \quad\indist\quad \mathcal{N}(0,1), \quad \text{ as } n \rightarrow \infty.
\]

\section{Notation and Preliminaries}\label{s:prelim}
In this section we introduce basic notation and some preliminary lemmas.

For $j\in \mathbb{N}$, we let $[j] \coloneqq \{1, \dots, j\}$. For a set $B$ and $j\in \mathbb{N}$, we let $\binom{B}{j} \coloneqq \{A \subseteq B \colon |A| = j\}$. We often abuse notation and write $xy$ for the $2$-element set $\{x, y\}$.

Let $G$ be a graph. We denote by $V(G)$ and $E(G) \subseteq \binom{V(G)}{2}$ the set of vertices and the set of edges of $G$, respectively. For $U \subseteq V(G)$ and $F \subseteq E(G)$, we let $G - U \coloneqq \left(V(G)\setminus U, E(G) \cap \binom{V(G)\setminus U}{2}\right)$ and $G - F \coloneqq (V(G), E(G) \setminus F)$. For disjoint $A, B \subseteq V(G)$, we let $E_G(A) \coloneqq E(G) \cap \binom{A}{2}$, $e_G(A) \coloneqq |E_G(A)|$, $G[A] \coloneqq (A, E_G(A))$, $E_G(A,B) \coloneqq E(G) \cap \{xy \colon x \in A, y \in B\}$, $e_G(A,B) \coloneqq |E_G(A,B)|$, and $G[A,B] \coloneqq (A \cup B, E_G(A,B))$. 

For $A \subseteq V(G)$, we let $N_G(A) \coloneqq \{x \in V(G) \setminus A \colon xy \in E(G) \text{ for some } y \in A\}$, which is called the {\em external neighbourhood} of $A$ in $G$. For $x \in V(G)$, we also write $N_G(x)$ for $N_G(\{x\})$. For $x, y \in V(G)$, we denote by $\dist_G(x,y)$ the \emph{graph-distance in $G$ from $x$ to $y$}, that is, the length of a shortest path from $x$ to $y$ (or $\infty$ if no such path exists). For $x \in V(G)$ and $A \subseteq V(G)$, we let $\dist_G(x,A) \coloneqq \min_{y \in A} \dist_G(x,y)$ (and for convenience, we let $\dist_G(x,\varnothing) \coloneqq \infty$).
For $x \in V(G)$ and $k \in \mathbb{N}$, we let $\partial_G(x,k) \coloneqq \{y \in V(G) \colon \dist_G(x,y) = k\}$ and $B_{G}(x,k) \coloneqq \{y \in V(G) \colon \dist_G(x,y) \leq k\}$. 

Given $e \in \binom{V(G)}{2}$, we define 
\[
    G+e \coloneqq (V(G),E(G)\cup \{e\}) \quad \text{ and } \quad G-e 
 \coloneqq (V(G),E(G)\setminus \{e\}).
\]

A \emph{rooted graph} is a pair $(H,r)$ where $H$ is a graph and $r \in V(H)$. Two rooted graphs $(H_1,r_1)$ and $(H_2,r_2)$ are \emph{isomorphic} (written $(H_1,r_1) \cong (H_2,r_2)$) if there exists a bijection $f \colon V(H_1) \rightarrow V(H_2)$ such that $f(r_1) = r_2$ and $f(u)f(v) \in E(H_2)$ if and only if $uv \in E(H_1)$ for all $u,v \in V(H_1)$.

We say that a sequence of events $\cE_n$ holds \emph{with high probability} (or whp for short) if $\prob{\cE_n} \rightarrow 1$ as $n \rightarrow \infty$.

We use standard asymptotic notations $o(\cdot),O(\cdot),\Theta(\cdot),\omega(\cdot),\Omega(\cdot)$ all of which are taken as $n \rightarrow \infty$ (unless otherwise specified) and the implied constants may depend on any parameter other than $n$, and in particular may depend on $c$ and $k$. We will also at times use probabilistic asymptotic notations, in particular, given a sequence of random variables $X(n)$ and a function $f : \mathbb{N} \to \mathbb{R}$, we write $X = o_{\bm{p}}(f)$ to mean that for all $\eps, \delta >0$ there exists $n_0 \in \mathbb{N}$ such that for all $n \geq n_0$
\[
\prob{\left|\frac{X(n)}{f(n)} \right| \geq \delta} \leq \eps
\]
and $X = O_{\bm{p}}(f)$ to mean that for any $\eps >0$ there exist $M>0$ and $n_0\in \mathbb{N}$ such that for all $n \geq n_0$
\[
\prob{\left|\frac{X(n)}{f(n)} \right| \geq M} \leq \eps.
\]

\subsection{Probabilistic Preliminaries}

In this subsection we collect some probabilistic tools that we will need.

We will use a probabilistic form of the Cauchy--Schwarz inequality, see for example \cite{M00}.
\begin{theorem}[Probabilistic Cauchy--Schwarz inequality] \label{thm:prob_Cauchy_Schwarz}
    For jointly distributed random variables $Y$ and $Z$ with $\expec{Y^2},\expec{Z^2},|\expec{Y Z}| < \infty$,
    \[
        |\mathrm{cov}(Y,Z)| \leq (\variance{Y}\variance{Z})^{\frac{1}{2}}.
    \]
    In particular,
    \begin{align}
    \left| \variance{Y} - \variance{Z}\right| \leq  \variance{Y-Z} + 2\left(\variance{Y-Z}\variance{Z}\right)^{\frac{1}{2}}.
    \end{align}
\end{theorem}

We will also need Slutsky's theorem, sometimes also known as Cram\'{e}r's theorem, which gives conditions under which the product of a pair of real random variables converges in distribution.

\begin{theorem}[Slutsky's theorem \cite{S26}] \label{thm:Slutsky}
    Let $(X_n)_{n \in \mathbb{N}}$ and $(Y_n)_{n \in \mathbb{N}}$ be sequences of real random variables. Let $X$ be a random variable and let $b \in \mathbb{R}$ be a constant. If $X_n\indist X$ and $Y_n \inprob b$, then $$X_n\cdot Y_n\quad\indist\quad bX,$$
   that is, if $X_n$ converges in distribution to $X$ and $Y_n$ converges in probability to $b$, then $X_nY_n$ converges in distribution to $bX$. 
\end{theorem}

We will also need some concentration results. The first is a bound for the `extreme' upper tail of a binomial distribution

\begin{theorem}[Extreme tail Chernoff bound {\cite[Corollary 2.4]{Janson2000}}] \label{thm:ex_tail_Chernoff}
    Let $Z \sim \mathrm{Bin}(n,p)$. Then for any $t \geq 7np$,
    \begin{align*}
        \prob{Z \geq t} \leq \exp(-t).
    \end{align*}
\end{theorem}

The second is a version of McDiarmid's inequality for biased distributions. 
\begin{lemma}[{\cite[Corollary 1.4]{Warnke2016}}] \label{Biased_McDiarmid_ineq} 
    Let $p \in (0,1)$, $n \in \mathbb{N}$, $M > 0$. 
    Let $X_1, \dots, X_n$ be independent $\mathrm{Bernoulli}(p)$ random variables. Let $f \colon \{0,1\}^n \rightarrow \mathbb{R}$ be a function such that for all $i \in [n]$ and $x_1, \dots, x_n \in \{0,1\}^n$, 
    \begin{align*}
        \max_{x_i' \in \{0,1\}} \left|f(x_1, \dots,x_{i-1}, x_i, x_{i+1}, \dots, x_n) - f(x_1 \dots, x_{i-1}, x_i', x_{i+1}, \dots, x_n) \right|  \leq M.
    \end{align*}
    Then, for all $\eps >0$,
    \begin{align*}
        \prob{\left|f(X_1, \dots, X_n) - \expec{f(X_1, \dots, X_n)}\right|>\eps} \leq 2\exp\left(-\frac{\eps^2}{2np(2-p)M^2+\frac{2M\eps}{3}}\right).
    \end{align*}
\end{lemma}

Finally, we will make use of the Efron--Stein inequality to bound several variances.
\begin{theorem}[Efron--Stein inequality \cite{ES81,S86}]\label{thm:EfronStein} 
    Suppose that $X_1,\dots,X_m,X_1',\dots,X_m'$ are independent random variables where $X_i,X_i'$ have the same distribution for all $i \in [m]$. Let $\textbf{X}=(X_1,\dots,X_{i-1},X_i,X_{i+1},\allowbreak\dots,X_m)$ 
    and $\textbf{X}^{(i)}=(X_{1},\dots ,X_{i-1},X_{i}',X_{i+1},\dots ,X_{m})$. Then for any measurable function $g \colon \mathbb{R}^m \rightarrow \mathbb{R}$,
    $$ \variance{g(\textbf{X})} \leq \frac{1}{2} \sum_{i=1}^{m} \mathbb E[(g(\textbf{X})-g(\textbf{X}^{(i)}))^2].$$
\end{theorem}

In particular, we will apply the Efron--Stein inequality to graph theoretical functions which can be expressed as a sum of `local' quantities, where the expected effect of `resampling' an edge can be controlled. This is expressed in the next lemma which will be essential for the proofs of \cref{lem:var_X_X_k_diff_bound,lem:Var_X_k_upper_bound,lem:var_L_k_hL_k_diff_bound}, the last of which provides an upper bound on $\variance{\tX_k}$.

\begin{lemma}\label{lem:ESapp}
Let $G \sim G(n,p)$ be a random graph, let $(\psi_G(v) \colon v \in V(G))$ be a collection of random variables and let $Z \coloneqq \sum_{v \in V(G)} \psi_G(v)$. Let $f \in \binom{V(G)}{2}$ be arbitrary and let $G^+ = G+f$ and $G^-=G-f$. Suppose that $K>0$ is such that $|\psi_G(v)| \leq K$ for each $v \in V(G)$ and let 
\[
D \coloneqq \{ v \in V(G) \colon \psi_{G^+}(v) \neq \psi_{G^-}(v)\}.
\]
Then
\[
\variance{Z} \leq 2K^2 p(1-p) n^2 \expec{|D|^2}.
\]
\end{lemma}
\begin{proof}
For an edge $e\in \binom{V(G)}{2}$, let $Y_e \coloneqq \ind{e \in E(G)}$. Furthermore, let $\textbf{Y}\coloneqq \left(Y_e\right)_{e\in\binom{V(G)}{2}}$ and $\textbf{Y}^{(e)}$ be the vector obtained by replacing $Y_e$ in $\textbf{Y}$ with a random variable $Y_e'$ which has the same distribution as $Y_e$ and is independent of $Y_{e'}$ for all $e' \in \binom{V(G)}{2}$. The random graph $G$ is fully determined by the random vector $\textbf{Y}$ and therefore, we can define $g(\textbf{Y})\coloneqq Z(G)$.

    Using the Efron--Stein inequality (\Cref{thm:EfronStein}) we get
\begin{align}\label{eq:Varsum}
\variance{g(\textbf{Y})} \leq
    \frac{1}{2}\sum_{e\in \binom{V(G)}{2}}\expec{\left(g(\textbf{Y})-g\left(\textbf{Y}^{(e)}\right)\right)^2}.
\end{align}
By symmetry, the value of $\expec{\left(g(\textbf{Y})-g\left(\textbf{Y}^{(e)}\right)\right)^2}$ is independent of the edge $e$. Let $f = \{u,v\}\in \binom{V(G)}{2}$ be a fixed edge as in the statement of the theorem, then, due to \eqref{eq:Varsum} and the law of total expectation, we have 
\begin{align}\label{eq:Varfixedfbound}
\variance{g(\textbf{Y})} 
&\leq
    \frac{n^2}{4} \expec{\left(g(\textbf{Y})-g\left(\textbf{Y}^{(f)}\right)\right)^2} \nonumber\\
    &= \frac{n^2}{4} \prob{Y_f=Y_f'}\condexpec{\left(g(\textbf{Y})-g\left(\textbf{Y}^{(f)}\right)\right)^2}{Y_f=Y_f'} \nonumber
   \\&\quad+\frac{n^2}{4}\prob{Y_f\neq Y_f'}\condexpec{\left(g(\textbf{Y})-g\left(\textbf{Y}^{(f)}\right)\right)^2}{Y_f\neq Y_f'} \nonumber
   \\
   &= \frac{n^2}{2} p(1-p) \condexpec{\left(g(\textbf{Y})-g\left(\textbf{Y}^{(f)}\right)\right)^2}{Y_f\neq Y_f'},
\end{align}
where we used in the last equality that if $Y_f=Y_f'$, then $\textbf{Y}=\textbf{Y}^{(f)}$ and therefore, 
\[
    \condexpec{\left(g(\textbf{Y})-g\left(\textbf{Y}^{(f)}\right)\right)^2}{Y_f=Y_f'}=0,
\]
and that $\prob{Y_f\neq Y_f'}=2p(1-p)$.

We note that, under the assumption $Y_f\neq Y_f'$, exactly one of the two vectors $\textbf{Y}$ and $\textbf{Y}^{(f)}$ \lq corresponds\rq\ to $G^+$ and the other to $G^-$. 
Hence, by the definition of $g$ we have 
\begin{align}
    \condexpec{\left(g(\textbf{Y})-g\left(\textbf{Y}^{(f)}\right)\right)^2}{Y_f\neq Y_f'}
    &=\expec{\left(Z(G^+)-Z(G^-)\right)^2} \nonumber
    \\
    &=\expec{\left(\sum_{w\in V(G)} \psi_{G^+}(w)-\psi_{G^-}(w) \right)^2}  \nonumber
    \\
    &\leq
    4 K^2 \cdot\expec{\left|\left\{w : \psi_{G^+}(w)\neq \psi_{G^-}(w\right\}\right|^2}  \nonumber
    \\
    &= 4 K^2 \expec{|D|^2}. \label{eq:Varfinalbound}
\end{align}
The result then follows from \eqref{eq:Varfixedfbound} and \eqref{eq:Varfinalbound}.
\end{proof}

\subsection{Stein's Method}\label{s:Stein} 

We will use the following version of Stein's method, developed by Barbour, Karo\'nski, and Ruci\'nski \cite{Barbour1989} (see \cref{thm:Stein}), to deduce \cref{lem:normality}, stated shortly. Barbour, Karo\'nski, and Ruci\'nski applied \cref{thm:Stein} in \cite{Barbour1989} to prove a cental limit theorem for the number of tree components of size $k\in \mathbb{N}$ in $G(n,p)$, in the regime where the expactation of this number tends to infinity. In comparison,  \cref{lem:normality} implies a central limit theorem for weighted sums of neighbourhood counts in the sparse regime. \cref{lem:normality} is used in the derivation of \cref{lem:conv_by_stein,lem:normality_truncatedneighbourhoods}.

Let $V \coloneqq [n]$ and $\cF \coloneqq \binom{V}{2}$. We think of $G(n,p)$ as a probability space $\mathscr{P}$ on $2^\cF$. For $U \subseteq V$, let $\cF(U) \coloneqq \{e \in \cF \colon e \cap U \neq \varnothing\}$. We say that for $U \subseteq V$, a random variable $Y\colon 2^\cF \rightarrow \mathbb{R}$ on $\mathscr{P}$ is \emph{semi-induced with respect to $U$} if for any $G_1, G_2 \in 2^\cF$ with $G_1 \cap \cF(U) = G_2 \cap \cF(U)$ we have $Y(G_1) = Y(G_2)$. 
Let $\cU \subseteq 2^V$ and for each $U \in \cU$ let $Y_U$ be a random variable on $\mathscr{P}$ that is semi-induced with respect to $U$. For $U \in \cU$ and $U' \subseteq V$ with $U' \cap U = \varnothing$, we let $Y_U^{U'}$ be the random variable on $\mathscr{P}$ defined by $Y_U^{U'}(G) \coloneqq Y_U(G \setminus \cF(U'))$ for all $G \in 2^\cF$. Let $Y \coloneqq \sum_{U \in \cU} Y_U$, $\sigma_Y \coloneqq (\variance{Y})^{\frac{1}{2}}$, and for every $U \in \cU$, $\mu_U \coloneqq \expec{Y_U}$ and $X_U \coloneqq \sigma_Y^{-1}(Y_U - \mu_U)$. For $U_1 \in \cU$, let $Z_{U_1} \coloneqq \sum_{U_2 \in \cU}Z_{U_1U_2}$, where 
\[
    Z_{U_1U_2} ~\coloneqq~ \begin{cases}
        \sigma_Y^{-1} Y_{U_2}, \text{ if } U_1 \cap U_2 \neq \varnothing \\
        \sigma_Y^{-1} (Y_{U_2} - Y_{U_2}^{U_1}), \text{ if } U_1 \cap U_2 = \varnothing.
    \end{cases}
\]
For $U_1, U_2 \in \cU$, let 
\[
    V_{U_1U_2} ~\coloneqq~ \sigma_Y^{-1}\sum_{\substack{U_3 \in \cU \\ U_3 \cap U_2 \neq \varnothing \\ U_3\cap U_1 = \varnothing}} Y_{U_3}^{U_1} + \sigma_Y^{-1} \sum_{\substack{U_3 \in \cU \\ U_3 \cap U_2 = \varnothing \\ U_3\cap U_1 = \varnothing}} (Y_{U_3}^{U_1} - Y_{U_3}^{U_1 \cup U_2}).
\]

Let 
\begin{align} \label{eq:delta}
    \delta \coloneqq \sum_{U \in \cU} \expec{|X_U| Z_U^2} + \sum_{U_1, U_2 \in \cU} \left(\expec{|X_{U_1} Z_{U_1U_2}V_{U_1U_2}|} + \expec{|X_{U_1}Z_{U_1U_2}|} \expec{|Z_{U_1} + V_{U_1U_2}|}\right).
\end{align}

\begin{theorem}[Barbour, Karo\'nski, Ruci\'nski \cite{Barbour1989}] \label{thm:Stein}
    In the above setting, if $\delta \rightarrow 0$ as $n \rightarrow \infty$, then 
    \[
        \frac{Y - \expec{Y}}{\sigma_Y} \quad\indist\quad \mathcal{N}(0,1) \quad \text{ as } n \rightarrow \infty,
    \]
    that is, $\sigma_Y^{-1}(Y - \expec{Y})$ converges in distribution to a standard normal random variable as $n \rightarrow \infty$. 
\end{theorem}

The following lemma gives a central limit for a broad class of random variables in $G(n,p)$ with $p = \Theta\left(\frac{1}{n}\right)$. Since the proof consists of relatively standard combinatorial estimates to bound the quantity $\delta$ as in~\cref{eq:delta}, we defer the details to \cref{a:normality}.

\begin{lemma} \label{lem:normality}
    Suppose that we are in the above setting and for $U \in \cU$, let $T_U \coloneqq K\mathds{1}_{\{G[U] \text{ is connected}\}}$ where $K>0$ is a constant. Let $t > 0$ be such that $|U| \leq t$ for all $U \in \cU$
    and suppose that for all $U \in \cU$, $G \in 2^{\cF}$, and $e \in \cF$ with $|e \cap U| = 1$, we have
    \begin{align} \label{eq:Y_U_ineq}
        0 \leq Y_U(G) \leq Y_U(G \setminus \{e\}) \leq T_U(G).
    \end{align}
Suppose $p = \Theta\left(\frac{1}{n}\right)$ and $\sigma_Y = \omega(n^{\frac{1}{3}})$. Then
\[
    \frac{Y - \expec{Y}}{\sigma_Y} \quad\indist\quad \mathcal{N}(0,1) \quad \text{ as } n \rightarrow \infty,
\]    
that is, $\sigma_Y^{-1}(Y - \expec{Y})$ converges in distribution to a standard normal random variable as $n \rightarrow \infty$.
\end{lemma}

In \cref{lem:normality}, the inequality in \cref{eq:Y_U_ineq} is equivalent to requiring $Y_{U}(G) = 0$ whenever $G[U]$ is not connected, and at most $K$ otherwise and moreover that deleting edges $e \in \cF$ with $|e \cap U| =1$ can only increase the value of $Y_U$.

\section{The effect of resampling edges on the global and local colourings}\label{s:strong4coreproperties}
In this section, with a view to using the Efron--Stein inequality (\Cref{thm:EfronStein}) to control the variance of $\tX - \tX_k$  (\cref{lem:var_X_X_k_diff_bound}), we will consider how resampling an edge of $G \sim G\left(n,\frac{c}{n}\right)$ effects the global and local colourings given by \Cref{alg:strong_4_core_alg,alg:strong_4_core_alg_local}. That is, given $e \in \binom{V(G)}{2}$, we are interested in how the global and local colourings in $G-e$ and $G+e$ can differ. A large difficulty in the analysis here comes from the fact that the strong $4$-core $S(G)$ is not a \emph{monotonic} function of the graph $G$ --- adding edges can either increase or decrease the set $S(G)$.

However, we will be able to show that any changes to the colouring (global or local) when adding an edge $e$ are in some sense \emph{localised}. In particular, we will show that only vertices that are in a red-purple component of $G-e$ or $G+e$ containing an endpoint of $e$ contribute to $\tX - \tX_k$ (see \cref{lem:D_cont_in_W}). We will show that these components are unlikely to be large in a strong sense (see \cref{lem:expec_W^*_bounds}) and also that if these components contain fewer than $k$ vertices, then $\tX = \tX_k$ (see \cref{lem:if_W_small_D_nothing}). These two facts allow us to bound the variance of $\tX - \tX_k$ by using the Efron--Stein inequality in the form of \cref{lem:ESapp}.

Before considering, in \cref{sec:resamp_typical_changes}, how resampling an edge affects the global and local colourings in the random graph $G\left(n, \frac{c}{n}\right)$ (in order to prove \cref{lem:expec_W^*_bounds}), we show, in \cref{sec:properties-colourings,sec:local_changes,sec:effect-resampling}, some properties about the global and local colourings for any graph and how adding or removing an edge can affect them. 

\subsection{Properties of the global and local colourings}\label{sec:properties-colourings}
In this section, let us fix a graph $G$, and let $V(G) = S \cup P \cup R$ be the partition given by \Cref{alg:strong_4_core_alg}.
We start with a few observations about the global colouring given by \Cref{alg:strong_4_core_alg} and its relation to the local colouring.

First, we note that in every red-purple component at least a $\frac{1}{4}$-fraction of the vertices are red.

\begin{observation} \label{obs:red_vxs_in_comp}
Let $P_t$ and $R_t$ be the sets of red and purple vertices after $t$ iterations of the while-loop in \cref{alg:strong_4_core_alg}, respectively.
Then for every component $C'$ of  $G[R_t \cup P_t]$, we have that $R_t\cap C'$ is incident to at most $3|R_t\cap C'|$ edges in $G$. Thus $|R_t \cap C'| \geq |C'|/4$. In particular, $R$ is incident to at most $3|R|$ edges in $G$ and if $C$ is a component of $G[R \cup P]$, then there are at least $|C|/4$ red vertices in $C$.
\end{observation}
\begin{proof}
   We first note that a vertex which is coloured red during the process will remain red. Furthermore, whenever a vertex $u$ changes its colour to purple, it is adjacent to a vertex $v$ which just changed its colour to red. We call $v$ a witness for $u$ and note that each vertex is a witness for at most $3$ vertices. 
   Note that for every edge $e$ incident to $R_t\cap C'$, the endpoint of $e$ that is coloured red first is a witness of the other endpoint. It follows that $R_t\cap C'$ is incident to at most $3|R_t \cap C'|$ edges in $G$. Thereafter as every vertex in $P_t\cap C'$ is adjacent to a vertex in $R_t\cap C'$ in $G$, we have $|P_t \cap C'| \leq 3|R_t \cap C'|$. Hence, $|R_t \cap C'| \geq |C'|/4$. 
\end{proof}

We will also need the following fact, which essentially says that the sets $S_k(G,w)$ approximate $S$ in a `monotone-decreasing' manner, in that, if a vertex is globally sapphire, then it is also locally sapphire from the point of view of every vertex that `sees' it.
\begin{observation} \label{obs:In_S_globally_implies_locally}
Let $w \in V(G)$, let $k \in \mathbb{N}$ and let $S_k(G,w)$ be the $k$-local strong $4$-core of $G$ at $w$. Then $S \cap B_G(w,k) \subseteq S_k(G, w)$. 
\end{observation}
\begin{proof}
    Let $A \coloneqq S \cap B_G(w,k-1)$ and let $B = \partial_G(w,k)$. Note that since $B \subseteq S_k(G,w)$ by definition, it suffices to show that $A \subseteq S_k(G,w)$. Hence, it will be sufficient to show that $A$ satisfies \eqref{e:localstrong4core}, that is, $|N_G(v) \cap (A \cup B)| \geq 4$ for all $v \in (A \cup N_G(A)) \setminus B$.

    By the definition of the strong $4$-core $S$ of $G$, we have $|N_G(v) \cap S| \geq 4$ for every $v \in A \cup N_G(A)$. Furthermore, if $v \in (A \cup N_G(A)) \setminus B$, then $N_G(v) \subseteq B_G(w,k)$ and so $N_G(v) \cap S \subseteq N_G(v) \cap (A \cup B)$. Hence, for every $v \in (A \cup N_G(A)) \setminus B$, we have
$|N_G(v) \cap (A \cup B)| \geq |N_G(v) \cap S| \geq 4$ as claimed.
\end{proof}

Recall that, for a vertex $v \in V(G)$, $C_v$ is defined to be the component of $G[P \cup R]$ which contains $v$ (and $C_v = \varnothing$ if $v \in S$). We next note that if $k \in \mathbb{N}$ is such that $C_v \subseteq B_G(v,k-1)$, then in the local colouring the component $C_v^k$ of $G[P_k(G,v) \cup R_k(G,v)]$ which contains $v$ is equal to $C_v$, and the global and local colourings agree on this component.

\begin{observation} \label{obs:C_small}
Let $v \in V(G)$ and $k \in \mathbb{N}$. Suppose $C_v \subseteq B_G(v,k-1)$. Then $C_v^k = C_v$, $P \cap C_v = P_k(G,v) \cap C_v$, and $R \cap C_v = R_k(G,v) \cap C_v$. In particular, $\phi_{G,k}(v) = \phi_G(v)$.
\end{observation}
\begin{proof}
    Note that $N_G(C_v) \subseteq S$, since $C_v$ is a component of $G - S$. Moreover, since $C_v \subseteq B_G(v,k-1)$, we have $N_G(C_v) \subseteq B_G(v,k)$. By \cref{obs:In_S_globally_implies_locally}, we have $N_G(C_v) \subseteq S_k(G,v)$ and so in particular, $C_v^k \subseteq C_v$. On the other hand, we may assume that  \Cref{alg:strong_4_core_alg} starts by recolouring the red vertices in $C_v$ in some order $v_1,v_2, \dots$. Since $N_G(C_v) \subseteq S(G,v)$ and $C_v \subseteq B_G(v,k-1)$, we could equally start the local colouring process by recolouring the same vertices in the same order. It follows that $C_v^k = C_v$ and furthermore the global and local colourings agree, implying that $\phi_{G,k}(v) = \phi_G(v)$.
\end{proof}

\subsection{Localising the changes} \label{sec:local_changes}
We will find that the effect of adding or deleting an edge $e \in \binom{V(G)}{2}$ on the global and local colourings can be in some sense \emph{localised} by finding a `separating' set $A$ of vertices  which are sapphire in a particular `resilient' manner, in that they are locally/globally sapphire in both $G-e$ and $G+e$. We will show that in this case, the local/global colouring can only change in the component of $G-A$ which contains the endpoints of $e$.

So, in slightly more generality, let us suppose we have two graphs $G_1$ and $G_2$ on the same vertex set which differ only on some set of vertices $A \cup B$, where $A$ is a set of vertices that are in the strong $4$-core for $G_2$ and $B$ is cut off from the rest of the both graphs by $A$ in the sense that $N_{G_1}(B), N_{G_2}(B) \subseteq A$. Since $G_1$ and $G_2$ agree outside of $A \cup B$, and the set $A$ is fully sapphire in $G_2$, it will follow that vertices outside $B$ which are sapphire in $G_1$ must also be sapphire in $G_2$. \cref{fig:sapphire_cut} depicts this configuration.

\begin{lemma}\label{lem:sapphire_cut_onesided}
    Let $G_1, G_2$ be two graphs on the same vertex set $V$. Let $A \subseteq S(G_2)$ and $B \subseteq V$ be such that $N_{G_1}(B) \cup N_{G_2}(B) \subseteq A$. Suppose $E(G_1) \triangle E(G_2) \subseteq \binom{A \cup B}{2}$.\footnote{Here $\triangle$ denotes the symmetric difference of sets.} Then $S(G_1) \setminus B \subseteq S(G_2) \setminus B$.
\end{lemma}
\begin{proof}
    We will show that $U \coloneqq (S(G_1) \setminus B) \cup S(G_2)$ has the strong $4$-core property \cref{Strong_4-core_property} for $G_2$, from which the result follows by the maximality of $S(G_2)$.

    By definition, every $v \in S(G_2) \cup N_{G_2}(S(G_2))$ satisfies 
    \[
    |N_{G_2}(v) \cap U| \geq |N_{G_2}(v) \cap S(G_2)| \geq 4 .
    \]
    In particular, since $A \subseteq S(G_2)$ this holds for every $v \in A \cup N_{G_2}(A)$.

    So, consider a vertex $v \in \big((S(G_1) \setminus B) \cup N_{G_2}( S(G_1) \setminus B)\big) \setminus (A \cup N_{G_2}(A))$. Note that, since $N_{G_2}(B) \subseteq A$, it follows that $v \not\in A \cup B$. Hence, $N_{G_1}(v) \cap B = \varnothing$ and since $E(G_1) \triangle E(G_2) \subseteq \binom{A \cup B}{2}$, we have $N_{G_2}(v) = N_{G_1}(v)$ and in particular
    $v \in S(G_1) \cup N_{G_1}(S(G_1))$. Thus, 
    \[
    |N_{G_2}(v) \cap U| \geq |N_{G_2}(v) \cap (S(G_1) \setminus B)| = |N_{G_1}(v) \cap (S(G_1) \setminus B)| = |N_{G_1}(v) \cap S(G_1)| \geq 4,
    \]
    where the last inequality follows from the definition of $S(G_1)$.
\end{proof}

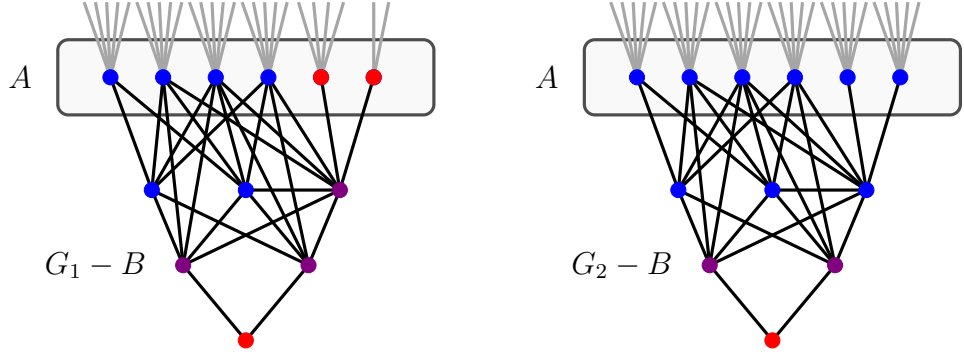
\begin{figure}
    \centering
\begin{tikzpicture}[scale=1]

\filldraw[color=black!70, fill=gray!5, rounded corners, very thick](0,0) rectangle (5,1);

\foreach \i in {1,2,3,4,5,6}{
    \foreach \j in {1,2,3,4,5}{
        \draw[very thick,black!35] (\i*0.7,0.5)--(\i*0.7-0.5 +0.15*\j,1.5);
    }
}

\draw[very thick] (2.5,-3)--(5/3,-2);
\draw[very thick] (2.5,-3)--(10/3,-2);
\foreach \i in {1,2}{
    \foreach \j in {1,2,3}{
        \draw[very thick] (\i*5/3,-2)--(\j*5/4,-1);
    }
}
\foreach \i in {1,2,3,4}{
    \foreach \j in {1,2}{
        \draw[very thick] (\i*0.7,0.5)--(\j*5/4,-1);
    }
}
\draw[very thick] (10/4,-1)--(15/4,-1);
\foreach \i in {2,3,4,5,6}{
    \draw[very thick] (\i*0.7,0.5)--(15/4,-1);
}
\draw[very thick] (1.4,0.5)--(5/3,-2);
\draw[very thick] (2.1,0.5)--(5/3,-2);
\draw[very thick] (2.1,0.5)--(10/3,-2);
\draw[very thick] (2.8,0.5)--(10/3,-2);
\foreach \i in {1,2,3,4,5,6}{
    \draw [fill=blue,draw=blue] (\i*0.7,0.5) circle [radius=.1];
}

\draw [fill=blue,draw=blue] (5/4,-1) circle [radius=.1];
\draw [fill=blue,draw=blue] (10/4,-1) circle [radius=.1];
\draw [fill=blue,draw=blue] (15/4,-1) circle [radius=.1];
\draw [fill=violet,draw=violet] (5/3,-2) circle [radius=.1];
\draw [fill=violet,draw=violet] (10/3,-2) circle [radius=.1];
\draw [fill=red,draw=red]  (2.5,-3) circle [radius=.1];

\newcommand\vara{-7}

\filldraw[color=black!70, fill=gray!5, rounded corners, very thick](0+\vara,0) rectangle (5+\vara,1);
\foreach \i in {1,2,3,4}{
    \foreach \j in {1,2,3,4,5}{
        \draw[very thick,black!35] (\i*0.7+\vara,0.5)--(\i*0.7-0.5 +0.15*\j +\vara,1.5);
    }
}
\foreach \j in {1,2,3,4}{
        \draw[very thick,black!35] (5*0.7+\vara,0.5)--(5*0.7-0.5 +0.2*\j +\vara,1.5);
    }
\foreach \j in {1,2}{
    \draw[very thick,black!35] (6*0.7+\vara,0.5)--(6*0.7-0.2 +0.2*\j +\vara,1.5);
}
\draw[very thick] (2.5+\vara,-3)--(5/3+\vara,-2);
\draw[very thick] (2.5+\vara,-3)--(10/3+\vara,-2);
\foreach \i in {1,2}{
    \foreach \j in {1,2,3}{
        \draw[very thick] (\i*5/3+\vara,-2)--(\j*5/4+\vara,-1);
    }
}
\foreach \i in {1,2,3,4}{
    \foreach \j in {1,2}{
        \draw[very thick] (\i*0.7+\vara,0.5)--(\j*5/4+\vara,-1);
    }
}
\draw[very thick] (10/4+\vara,-1)--(15/4+\vara,-1);
\foreach \i in {2,3,4,5,6}{
    \draw[very thick] (\i*0.7+\vara,0.5)--(15/4+\vara,-1);
}
\draw[very thick] (1.4+\vara,0.5)--(5/3+\vara,-2);
\draw[very thick] (2.1+\vara,0.5)--(5/3+\vara,-2);
\draw[very thick] (2.1+\vara,0.5)--(10/3+\vara,-2);
\draw[very thick] (2.8+\vara,0.5)--(10/3+\vara,-2);
\foreach \i in {1,2,3,4,5,6}{
    \draw [fill=blue,draw=blue] (\i*0.7+\vara,0.5) circle [radius=.1];
}
\draw [fill=red,draw=red] (5*0.7+\vara,0.5) circle [radius=.1];
\draw [fill=red,draw=red] (6*0.7+\vara,0.5) circle [radius=.1];
\draw [fill=blue,draw=blue] (5/4+\vara,-1) circle [radius=.1];
\draw [fill=blue,draw=blue] (10/4+\vara,-1) circle [radius=.1];
\draw [fill=violet,draw=violet] (15/4+\vara,-1) circle [radius=.1];
\draw [fill=violet,draw=violet] (5/3+\vara,-2) circle [radius=.1];
\draw [fill=violet,draw=violet] (10/3+\vara,-2) circle [radius=.1];
\draw [fill=red,draw=red]  (2.5+\vara,-3) circle [radius=.1];

\node at (-6.5,-2) {$G_1-B$};
\node at (-6.5+7,-2) {$G_2-B$};

\node at (-7.5,0.5) {$A$};
\node at (-0.5,0.5) {$A$};

\end{tikzpicture}
 \caption{Vertices in $V\setminus (A \cup B)$ have the same neighbourhoods in $G_1$ and in $G_2$ but in $G_2$ more of these neighbours will be sapphire.}
   \label{fig:sapphire_cut}
\end{figure}

It follows as a consequence of \Cref{lem:sapphire_cut_onesided} that if the set $A$ is sapphire in both $G_1$ and $G_2$, then the global colourings will agree outside of $B$.

\begin{lemma} \label{lem:sapphire_cut}
    Let $G_1, G_2$ be two graphs on the same vertex set $V$. Let $A \subseteq S(G_1) \cap S(G_2)$ and $B\subseteq V$ be such that $N_{G_1}(B) \cup N_{G_2}(B) \subseteq A$. Suppose $E(G_1) \triangle E(G_2) \subseteq \binom{A \cup B}{2}$. Then 
    \[
        S(G_1) \setminus B = S(G_2)\setminus B,\quad P(G_1) \setminus B = P(G_2) \setminus B,\quad \text{and} \quad R(G_1) \setminus B = R(G_2) \setminus B.
    \] 
    Moreover, for any vertex $w \in V \setminus B$, we have $C_w(G_1) = C_w(G_2) \subseteq V \setminus (A \cup B)$ and thus
    \[
        G_1[C_w(G_1)] = G_2[C_w(G_2)] \quad\text{and}\quad R(G_1) \cap C_w(G_1) = R(G_2) \cap C_w(G_2).
    \]
\end{lemma}
\begin{proof}
    By \Cref{lem:sapphire_cut_onesided}, it follows that $S(G_1) \setminus B = S(G_2) \setminus B$. Hence, using that $A \subseteq S(G_1) \cap S(G_2)$, $N_{G_1}(B) \cup N_{G_2}(B) \subseteq A$, and $E(G_1) \triangle E(G_2) \subseteq \binom{A \cup B}{2}$, we have
    \begin{align*}
        P(G_1) \setminus B &= N_{G_1}(S(G_1)) \setminus B =  N_{G_1}(S(G_1)) \setminus (A \cup B) = N_{G_2}(S(G_1)) \setminus (A\cup B)\\
        &= N_{G_2}(S(G_2)) \setminus (A\cup B) = N_{G_2}(S(G_2)) \setminus  B = P(G_2) \setminus B,
    \end{align*}
    and hence, also $R(G_1) \setminus B = R(G_2) \setminus B$.
    
    Finally, note that since $A \subseteq S(G_1) \cap S(G_2)$ and there is no edge from $B$ to $V \setminus (A \cup B)$, for any $w \in V \setminus B$, we have $C_w(G_1), C_w(G_2) \subseteq V \setminus 
(A \cup B)$ and thus $G_1[C_w(G_1)] = G_2[C_w(G_2)]$.
\end{proof}

We will also need a version of \cref{lem:sapphire_cut} which we can apply to the local colourings.
\begin{lemma} \label{lem:sapphire_cut_local}
    Let $G_1, G_2$ be two graphs on the same vertex set $V$. Let $k \geq 1$, $w\in V$ and $A \subseteq V$ be such that $A \subseteq S(G_1) \cap S(G_2)$ and $B \subseteq V \setminus \{w\}$ be such that $N_{G_1}(B)\cup N_{G_2}(B)  \subseteq A$. Suppose $E(G_1) \triangle E(G_2) \subseteq \binom{A \cup B}{2}$. 
    Then $C\coloneqq C_w^k(G_1) = C_w^k(G_2) \subseteq V \setminus (A \cup B)$, $G_1[C] = G_2[C]$, $P_k(G_1, w) \cap C = P_k(G_2,w) \cap C$, and $R_k(G_1, w) \cap C = R_k(G_2,w) \cap C$.
\end{lemma}

\begin{proof}
Since $A$ is a vertex-cut that separates $B$ and $V(G)\setminus (A\cup B)$ in both $G_i$, $w\notin A\cup B$, and $E(G_1) \triangle E(G_2) \subseteq \binom{A \cup B}{2}$
there exists a set $U\subseteq V\setminus A$ containing $w$ such that $G_1[U]=G_2[U]$ and the graph $F \coloneqq G_1[U]$ is connected and consists of a connected component of $G_i-A$ for each $i \in \{1,2\}$. By extension, with $A' \coloneqq (A \cap B_F(v,k)) \cup \partial_F(v,k)$, there exists  $U'\subseteq U$ containing $w$ such that $F'\coloneqq G_1[U']=G_2[U']$ is connected and consists of a connected component of $G_i-A'$ for each $i \in \{1,2\}$. By \Cref{obs:In_S_globally_implies_locally}, we have $A'\subseteq B_{G_i}(w,k) \cap S_k(G_i,w)$ for each $i \in \{1,2\}$. Thus $C_w^k(G_1), C_w^k(G_2) \subseteq U'$.

Now note that if the local colouring process in $G_1$ recolours the red vertices in $U'$ in some order $v_1,v_2,\dots$, then we could equally start the local colouring process in $G_2$ by recolouring the same vertices in the same order, and vice versa. 
It follows that $C\coloneqq C_w^k(G_1) = C_w^k(G_2) \subseteq U'$, $P_k(G_1, w) \cap C = P_k(G_2,w) \cap C$, and $R_k(G_1, w) \cap C = R_k(G_2,w) \cap C$. 
\end{proof}

\subsection{The effect of resampling an edge}\label{sec:effect-resampling}
Throughout this subsection, let us fix a graph $G$ and a potential edge $e = \{u,v\} \in \binom{V(G)}{2}$. We let $G^+ = G+e$ and $G^- = G- e$.  Moreover, for $x \in V(G)$ and $\circ \in \{+, -\}$, we let $C_x^\circ \coloneqq C_x(G^\circ)$, $W^\circ \coloneqq C_u^\circ \cup C_v^\circ$, and we abbreviate $C_x^{k,\circ} \coloneqq C_x^k(G^\circ)$, $\phi^\circ \coloneqq \phi_{G^\circ}$, $\phi_k^\circ \coloneqq \phi_{G^\circ,k}$, $S^\circ \coloneqq S(G^\circ)$, $P^\circ \coloneqq P(G^\circ)$, and $R^\circ \coloneqq R(G^\circ)$.
Further, we define 
\begin{align*} 
    I^* \coloneqq \{w \in V(G) \colon G^+[C_w^+] = G^-[C_w^-] \text{ and } R^+ \cap C_w^+ = R^- \cap C_w^-\}, 
\end{align*}
and
\begin{align*} 
    I^*_k \coloneqq \{&w \in V(G) \colon G^+[C_w^{k,+}] = G^-[C_w^{k,-}] \text{ and }\nonumber \\ &R_k(G^+,w) \cap C_w^{k,+} = R_k(G^-,w) \cap C_w^{k,-}\}. 
\end{align*}
Suppose that $w \in I^*$. By definition, we have $G^+[C_w^+] = G^-[C_w^-]$ which implies that $C_w^+ = C_w^-$. Since $R^+ \cap C_w^+ = R^- \cap C_w^-$, $C_w^+ \subseteq P^+ \cup R^+$, and $C_w^- \subseteq P^- \cup R^-$, we also have $P^- \cap C_w^- = P^+ \cap C_w^+$.
Thus $I^*$ is the set of vertices whose red-purple components in the global colouring are `indifferent' to the change in the edge $e = \{u,v\}$ and $I^*_k$ is the corresponding set of vertices for the local colouring. In particular, this implies that the value of $\phi$ does not change for vertices in $I^*$ when flipping the edge $e$.
\begin{observation} \label{obs:no_change_for_indif}
    For $w \in I^*$, we have $\phi^+(w) = \phi^-(w)$ and for $w \in I^*_k$, we have $\phi^+_k(w) = \phi^-_k(w)$.
\end{observation}
\begin{proof}
    For a graph $H$ and $w \in V(H)$, the value of $\phi_H(w)$ only depends on $H[C_w(H)]$ and the global colouring on $C_w(H)$. If $w \in I^*$, then these are identical for $G^+$ and $G^-$, thus $\phi^+(w) = \phi^-(w)$. The second statement follows analogously.
\end{proof}

Let us now consider how the global colourings in $G^+$ and $G^-$ can differ, which will depend on the colours that $u$ and $v$ receive in $G^+$ and $G^-$.
Firstly, we note that if $u$ and $v$ are both sapphire in $G^-$, then the global colourings in $G^+$ and $G^-$ agree.

\begin{observation} \label{obs:uv_in_S^-}
    Suppose that $\{u,v\} \subseteq S^-$. Then $S^- = S^+$, $P^- = P^+$, and $R^- = R^+$. In particular, $I^* = V(G)$.
\end{observation}
\begin{proof}
    We first show that $S^- \subseteq S^+$.
    Since $e = \{u,v\} \subseteq S^-$, we have $N_{G^+}(S^-) = N_{G^-}(S^-)$.
    By the definition of $S^-$, we have $|N_{G^-}(w) \cap S^-| \geq 4$ for every $w \in S^- \cup N_{G^-}(S^-)$. Therefore, $|N_{G^+}(w) \cap S^-| \geq |N_{G^-}(w) \cap S^-| \geq 4$ for every $w \in S^- \cup N_{G^-}(S^-) =  S^- \cup N_{G^+}(S^-)$. Hence, $S^-$ has the strong $4$-core property \cref{Strong_4-core_property} for $G^+$ and it follows that $S^- \subseteq S^+$. 

    We now show that $S^+ \subseteq S^-$.
    To that end, let $x \in S^+ \cup N_{G^-}(S^+)$. Since $N_{G^-}(S^+) \subseteq N_{G^+}(S^+)$, we have $|N_{G^+}(x) \cap S^+| \geq 4$, by the definition of $S^+$. If $x \notin \{u,v\}$, then $N_{G^+}(x) = N_{G^-}(x)$ and so $|N_{G^-}(x) \cap S^+| = |N_{G^+}(x) \cap S^+| \geq 4$. If $x \in \{u,v\}$, then, since $\{u, v\} \subseteq S^- \subseteq S^+$, we have $|N_{G^-}(x) \cap S^+| \geq |N_{G^-}(x) \cap S^-| \geq 4$, by the definition of $S^-$.
    It follows that $S^+$ has the strong $4$-core property \cref{Strong_4-core_property} for $G^-$. Hence, $S^+ \subseteq S^-$.

    Therefore, $S^- = S^+$. Finally note that since $\{u,v\} \subseteq S^-$, we have $P^- = N_{G^-}(S^-) = N_{G^+}(S^+) = P^+$ and thus $R^- = V(G) \setminus (S^- \cup P^-) = V(G) \setminus (S^+ \cup P^+) = R^+$.

    Now note that $G^-[R^- \cup P^-]$ and $G^+[R^+ \cup P^+]$ have the same components and for any such component $C$, we have $G^-[C] = G^+[C]$. Hence $I^* = V(G)$.
\end{proof}

Whilst the global colourings do not agree in every case, we will find that one of $S^+$ or $S^-$ will always be a subset of the other, and the same is true for $W^+$ and $W^-$.

\begin{lemma} \label{lem:S_and_W_containment}
    The following hold. 
    \begin{enumerate}[label = \upshape{(\roman*)}]
        \item If $\{u,v\} \cap S^+ = \varnothing$, then $N_{G^-}(W^+) \subseteq N_{G^+}(W^+) \subseteq S^+ \subseteq S^-$, $R^- \subseteq R^+$, and $W^- \subseteq W^+$.
        \item If $\{u,v\} \cap S^+ \neq \varnothing$, then $N_{G^-}(W^-) \subseteq N_{G^+}(W^-) \subseteq S^- \subseteq S^+$, $R^+ \subseteq R^-$, and $W^+ \subseteq W^-$.
    \end{enumerate}
\end{lemma}
\begin{proof} We deal with the two cases separately.
    \begin{enumerate}[label=\textbf{Case \arabic*:}, ref=\arabic*, wide, labelwidth=0pt, labelindent=0pt, parsep=5pt]
        \item $\{u,v\} \cap S^+ = \varnothing$.
        
            \noindent We first show that $S^+ \subseteq S^-$. 
            By the definition of $S^+$, we have $|N_{G^+}(x) \cap S^+| \geq 4$ for every $x \in S^+ \cup N_{G^+}(x)$. 
            Since $\{u,v\} \cap S^+ = \varnothing$, we have $N_{G^+}(S^+) = N_{G^-}(S^+)$ and for any vertex $x\in V(G)$, we have $N_{G^+}(x) \cap S^+ = N_{G^-}(x) \cap S^+$. 
            Thus, it follows that $|N_{G^-}(x) \cap S^+| \geq 4$ for every $x \in S^+ \cup N_{G^-}(S^+)$. That is, $S^+$ satisfies the strong $4$-core property \cref{Strong_4-core_property} for $G^-$. Hence, $S^+ \subseteq S^-$, and so
            \begin{align}
                N_{G^-}(W^+) \subseteq N_{G^+}(W^+) \subseteq S^+ \subseteq S^-, \label{eq:N_containment}
            \end{align}
            where we used that $G^- \subseteq G^+$ and the definition of $W^+$.

            Furthermore, since $N_{G^+}(S^+) = N_{G^-}(S^+)$ and $S^+ \subseteq S^-$, we have $S^+ \cup N_{G^+}(S^+) = S^+ \cup N_{G^-}(S^+) \subseteq S^- \cup N_{G^-}(S^-)$. Hence, 
            \[
            R^- = V(G) \setminus (S^- \cup N_{G^-}(S^-)) \subseteq V(G) \setminus (S^+ \cup N_{G^+}(S^+)) = R^+.
            \]

            Finally, by \cref{eq:N_containment}, any component of $G^-[P^- \cup R^-] = G^- - S^-$ is contained in $W^+$ or $V(G^-) \setminus (W^+ \cup N_{G^-}(W^+))$. Since $\{u,v\} \cap S^+ = \varnothing$, it follows from the definition of $W^+$ that $\{u,v\} \subseteq W^+$. Hence, the components $C_u^-$ and $C_v^-$ of $G^- - S^-$ are contained in $W^+$.
            
        \item $\{u,v\} \cap S^+ \neq \varnothing$. 
        
            \noindent Note that, in this case, since $\{u,v\} \in E(G^+)$, we have $\{u,v\} \subseteq S^+ \cup N_{G^+}(S^+)$. Thus, by the definition of $S^+$, we have 
            \begin{align} \label{uv_satisfy_property}
                |N_{G^+}(u) \cap S^+|, |N_{G^+}(v) \cap S^+| \geq 4.
            \end{align}
            We now show that $S^+ \cup S^-$ satisfies the strong $4$-core property \cref{Strong_4-core_property} for $G^+$. To that end, let $x \in S^+ \cup S^- \cup N_{G^+}(S^+ \cup S^-)$. If $x \in S^+ \cup N_{G^+}(S^+)$, then $|N_{G^+}(x) \cap (S^+ \cup S^-)| \geq |N_{G^+}(x) \cap S^+| \geq 4$ by the definition of $S^+$. Similarly, if $x \in S^- \cup N_{G^-}(S^-)$, then $|N_{G^+}(x) \cap (S^+ \cup S^-)| \geq |N_{G^-}(x) \cap S^-| \geq 4$ by the definition of $S^-$. Thus, we may assume that 
            \[
            x \in N_{G^+}(S^+ \cup S^-) \setminus (N_{G^+}(S^+) \cup N_{G^-}(S^-)) = N_{G^+}(S^-) \setminus N_{G^-}(S^-) \subseteq \{u,v\}
            \]
            and so $|N_{G^+}(x) \cap (S^+ \cup S^-)| \geq 4$ by \cref{uv_satisfy_property}. 
            Hence, $S^+ \cup S^-$ satisfies the strong $4$-core property~\cref{Strong_4-core_property} for $G^+$. It follows that $S^+ \cup S^- \subseteq S^+$, and so $S^- \subseteq S^+$. 

            Next we show that $N_{G^-}(W^-) \subseteq N_{G^+}(W^-) \subseteq S^-$, where the first inclusion is apparent. Furthremore, since $W^-$ is a union of components of $G^- - S^-$, we have $N_{G^-}(W^-) \subseteq S^-$. We split into three cases.
            If $\{u,v\} \cap S^- = \varnothing$, then $\{u,v\} \subseteq W^-$ by the definition of $W^-$ and thus $N_{G^+}(W^-) = N_{G^-}(W^-) \subseteq S^-$. If $\{u,v\} \subseteq S^-$, then by \cref{obs:uv_in_S^-}, $S^- = S^+$ and $W^- = W^+ = \varnothing$ and thus the conclusion trivially holds.
            So, we may assume without loss of generality that $u \in V(G^-) \setminus S^-$
            and $v \in S^-$. Then $u \in W^-$ and thus $N_{G^+}(W^-) \subseteq N_{G^-}(W^-) \cup \{v\} \subseteq S^-$. 

            Next we show that $R^+ \subseteq R^-$. Indeed, since $S^- \subseteq S^+$ and $G^- \subseteq G^+$, we have 
            \[
            R^+ = V(G) \setminus (S^+ \cup N_{G^+}(S^+)) \subseteq V(G) \setminus (S^- \cup N_{G^+}(S^-)) \subseteq V(G) \setminus (S^- \cup N_{G^-}(S^-)) = R^-. 
            \]
            
            Finally, we show that $C_x^+ \subseteq W^-$ for each $x \in \{u,v\}$, and hence $C_u^+ \cup C_v^+ = W^+ \subseteq W^-$. Note that, if $x \in S^- \subseteq S^+$, then $C_x^+ = \varnothing \subseteq W^-$ vacuously. Hence we may assume that $x \not\in S^-$, and so $x \in W^-$. In particular, $C_x^+$ meets $W^-$. On the other hand, recalling that $N_{G^+}(W^-) \subseteq S^- \subseteq S^+$, we see that every component of $G^+ - S^+$ is contained in $W^-$ or $V(G^+) \setminus (W^- \cup N_{G^+}(W^-))$. Hence, since $C_x^+$ meets $W^-$, it follows that $C_x^+ \subseteq W^-$ as claimed.
            \qedhere
\end{enumerate} 
\end{proof}

Using this information, we can get some control over the set of vertices $w$ whose global or local red-purple components are indifferent to flipping the edge $e$. 

We start by showing that, if $u$ and $v$ are both sapphire in $G^-$, then all local colourings for $G^-$ and $G^+$ agree.

\begin{lemma} \label{lem:uv_in_S^-_implies_equality_locally}
    Let $k \in \mathbb{N}$.
    If $\{u,v\} \subseteq S^-$, then $I^*_k = V(G).$
\end{lemma}
\begin{proof}
    By \cref{obs:uv_in_S^-}, $S^+ = S^- \eqqcolon A$. 
    Since $\{u,v\} \subseteq A$, we have $E(G^-) \triangle E(G^+) \subseteq \binom{A}{2}$. 
    By \cref{lem:sapphire_cut_local}, with $G^-$, $G^+$, $A$, and $\varnothing$ playing the roles of $G_1$, $G_2$, $A$, and $B$, respectively, we have $I^*_k = V(G)$.
\end{proof}

Next we show that vertices not in $W^+$ or $W^-$ are `indifferent' to flipping the edge $e$.

\begin{lemma} \label{lem:not_in_W_then_in_I}
    The following hold for any $k \in \mathbb{N}$.
    \begin{enumerate}[label= \upshape{(\roman*)}]
        \item If $\{u,v\} \cap S^+ = \varnothing$, then $V(G) \setminus W^+ \subseteq I^* \cap I^*_k$.
        \item If $\{u,v\} \cap S^+
        \neq \varnothing$, then $V(G) \setminus W^- \subseteq I^* \cap I^*_k$.
    \end{enumerate}
\end{lemma}
\begin{proof}
    We deal with the two cases separately.
    \begin{enumerate}[label=\textbf{Case \arabic*:}, ref=\arabic*, wide, labelwidth=0pt, labelindent=0pt, parsep=5pt]
        \item $\{u,v\} \cap S^+ = \varnothing$. 
            
            \noindent Let $w \in V(G) \setminus W^+$. Note that by \cref{lem:S_and_W_containment}, we have $N_{G^-}(W^+) \subseteq N_{G^+}(W^+) \subseteq S^+ = S^+ \cap S^-$. Since $\{u,v\} \cap S^+ = \varnothing$, we have $\{u,v\} \subseteq W^+$. Hence, $E(G^+) \triangle E(G^-) \subseteq \binom{W^+}{2}$. By \cref{lem:sapphire_cut} and \cref{lem:sapphire_cut_local}  with $G^+$, $G^-$, $N_{G^+}(W^+)$, and $W^+$ playing the roles of $G_1$, $G_2$, $A$, and $B$, respectively, we have $w \in I^* \cap I^*_k$.
        \item $\{u,v\} \cap S^+ \neq \varnothing$.
        
            \noindent Let $w \in V(G) \setminus W^-$. By \cref{lem:S_and_W_containment}, we have $N_{G^-}(W^-) \subseteq N_{G^+}(W^-) \subseteq S^- = S^- \cap S^+$.
            
            First suppose that $\{u,v\} \subseteq S^-$. Then by \cref{lem:uv_in_S^-_implies_equality_locally}, we have $I^*_k = V(G)$ and by \cref{obs:uv_in_S^-}, we have $I^* = V(G)$. Hence, $w \in I^* \cap I^*_k$.
            
            Thus, we may assume without loss of generality that $u \in P^- \cup R^-$. Hence, $u \in W^-$ and thus $\{u,v\} \subseteq W^- \cup N_{G^+}(W^-)$. Let $A \coloneqq N_{G^+}(W^-)$. It follows that $E(G^+) \triangle E(G^-) \subseteq \binom{A \cup W^-}{2}$.
            Hence, by \cref{lem:sapphire_cut} with $G^+$, $G^-$, $A$, and $W^-$ playing the roles of $G_1$, $G_2$, $A$, and $B$, respectively, we have $w \in I^*$.

            Since $A \subseteq S^- \cap S^+$, we have, by \cref{obs:In_S_globally_implies_locally}, that $A \cap B_{G^-}(w,k) \subseteq S_k(G^-, w)$ and $A \cap B_{G^+}(w,k) \subseteq S_k(G^+, w)$. Hence, by \cref{lem:sapphire_cut_local} with $G^+$, $G^-$, $A$, and $W^+$ playing the roles of $G_1$, $G_2$, $A$, and $B$, respectively, we have $w \in I^*_k$.

            It follows that $w \in I^* \cap I^*_k$. \qedhere
    \end{enumerate}
\end{proof}

To bound the variances of $\tX_k$ and $\tX -\tX_k$ using the Efron--Stein inequality (in the form of \cref{lem:ESapp}), we will be interested in the set
\begin{align} \label{eq:def_tD}
    \tD \coloneqq \{ w \in V(G) \colon \phi_k^+(w) \neq \phi_k^-(w)\}
\end{align}
of vertices whose local contribution changes and the set
\begin{align} \label{eq:def_D}
    D \coloneqq \{w \in V(G) \colon \phi^+(w)-\phi_k^+(w) \neq \phi^-(w) - \phi_k^-(w)\}
\end{align}
of vertices for which the difference between the local and the global contribution changes. We can now show that the vertices in $D \cup \tD$ are in some sense local to the edge $e$, in that they are contained in $W^+$ or $W^-$.

\begin{lemma} \label{lem:D_cont_in_W}
    The following hold.
    \begin{enumerate}[label= \upshape{(\roman*)}]
        \item If $\{u,v\} \cap S^+ = \varnothing$, then $D \cup \tD \subseteq W^+$.
        \item If $\{u,v\} \cap S^+
        \neq \varnothing$, then $D \cup \tD \subseteq W^-$.
    \end{enumerate}
\end{lemma}
\begin{proof}
    Suppose that $\{u,v\} \cap S^+ = \varnothing$ (the proof in the other case is analgous). Let $w \in V(G) \setminus W^+$. By \cref{lem:not_in_W_then_in_I}, we have $w \in I^* \cap I^*_k$. By \cref{obs:no_change_for_indif}, we have that $\phi^+(w) = \phi^-(w)$ and $\phi^+_k(w) = \phi^-_k(w)$. Hence $w \notin D \cup \tD$. Thus we have shown that $D \cup \tD \subseteq W^+$.
\end{proof}

The following lemma shows that $D =  \varnothing$ if both $W^+$ and $W^-$ have size smaller than $k$.
\begin{lemma} \label{lem:if_W_small_D_nothing}
    If $\max \{|W^+|, |W^-|\} < k$, then $D = \varnothing$.
\end{lemma}
\begin{proof}
    We make the usual case distinction according to \Cref{lem:S_and_W_containment}.
    \begin{enumerate}[label=\textbf{Case \arabic*:}, ref=\arabic*, wide, labelwidth=0pt, labelindent=0pt, parsep=5pt]
        \item $\{u,v\} \cap S^+ = \varnothing$. \label{Case_1_W_small}
        
            \noindent By \cref{lem:D_cont_in_W}, we have $D \subseteq W^+$. Let $w \in W^+$ and note that, by definition, $C_w(G^+) \subseteq W^+$. Since $|W^+| < k$, this implies that $C_w(G^+) \subseteq B_{G^+}(w,k-1)$. Hence, by \cref{obs:C_small}, we have $\phi_k^+(w) = \phi^+(w)$. 
            
            On the other hand, by \cref{lem:S_and_W_containment}, we have $N_{G^-}(W^+) \subseteq S^-$. 
            Since $w \in W^+$ and $C_w(G^-)$ is the component of $G^- - S^-$ that contains $w$, this implies that $C_w(G^-) \subseteq W^+$. Since $|W^+| < k$, this implies that $C_w(G^-) \subseteq B_{G^-}(w,k-1)$. Hence, by \cref{obs:C_small}, we have $\phi_k^-(w) = \phi^-(w)$.

            Therefore, $w \notin D$ and so $D = \varnothing$.
        \item $\{u,v\} \cap S^+ \neq \varnothing$.
        
            \noindent This case is proved in an analogous way to Case~\ref{Case_1_W_small}, noting that \cref{lem:D_cont_in_W} implies that $D \subseteq W^-$. \qedhere
    \end{enumerate}
\end{proof}

Finally, we will prove some structural properties about $W^- \cup W^+ \cup \{u,v\}$ which will later use in the proof of \cref{lem:expec_W^*_bounds} to show that this set is unlikely to be large in $G\left(n,\frac{c}{n}\right)$. 

\begin{lemma} \label{lem:W^*_properties}
    Let $W^* \coloneqq W^- \cup W^+ \cup \{u,v\}$ and $W^*_R \coloneqq (R^- \cap W^*) \cup (R^+ \cap W^*)$. Then the following hold.
    \begin{enumerate}[label = \upshape{(\roman*)}]
        \item $G^+[W^*]$ is connected. \label{item:G^+[W^*]_is_connected}
        \item $N_G(x) \subseteq W^*$ for all $x \in W^*_R$. \label{item:W^*R_property}
        \item $|W^*_R| \geq \frac{|W^*|-2}{4}$. \label{item:W^*_R_size}
    \end{enumerate}
\end{lemma}
\begin{proof} 
    Recall that $W^- \coloneqq C_u^- \cup C_v^-$ where for $x \in \{u,v\}$, $C_x^-$ is the component of $G^- - S^-$ that contains $x$ (or $C_x^- = \varnothing$ if $x \in S^-$). Since $G^+ \coloneqq G + uv$, this implies that $G^+[W^-]$ is connected. Similarly, $G^+[W^+]$ is connected. Hence, since the union of two non-disjoint connected sets is connected, \cref{item:G^+[W^*]_is_connected} follows. 

    Let $x \in R^- \cap W^*$. By the definition of $R^-$, we have $N_{G^-}(x) \subseteq R^- \cup P^-$. 
    By the definition of $W^-$, we have $x \in W^- \cup W^+$. Note that by \cref{lem:S_and_W_containment}, we have $N_{G^+}(W^- \cup W^+) \subseteq S^-$. Thus $N_{G}(x) \subseteq N_{G^+}(x) \subseteq W^- \cup W^+ \subseteq W^*$.
    Arguing along similar lines, we can conclude that for each $x \in W^*_R$, we have $N_G(x) \subseteq W^*$. Hence we have proved \cref{item:W^*R_property}.

    To prove \cref{item:W^*_R_size}, suppose first that $\{u,v\} \cap S^+ = \varnothing$. Then, by \cref{lem:S_and_W_containment}, we have $R^- \subseteq R^+$ and $W^- \subseteq W^+$. Hence, $R^- \cap W^- \subseteq R^+ \cap W^+$ and so $W^* = W^+ \cup \{u,v\}$ and $W^*_R \supseteq R^+ \cap W^+ = (R^+ \cap C_u^+) \cup (R^+ \cap C_v^+)$. Since $C_u^+$ and $C_v^+$ are components of $G^+ - S^+$, we have $C_u^+ \cap C_v^+ = \varnothing$ or $C_u^+ = C_v^+$. Thus, by \cref{obs:red_vxs_in_comp}, we have $|W^*_R| \geq |R^+ \cap W^+| \geq \frac{1}{4}|W^+| \geq \frac{1}{4}|W^*\setminus \{u,v\}| = \frac{|W^*| -2}{4}$. In the case where $\{u,v\} \cap S^+ \neq \varnothing$, $|W^*_R| \geq \frac{|W^*| -2}{4}$ follows analogously.
\end{proof}

\subsection{Typical changes in the colourings after resampling} \label{sec:resamp_typical_changes}
So far we have only considered the deterministic effects of adding or removing an edge to a fixed $G$ on the global and local colourings. Let us now show what these deterministic results imply for the `typical' changes when resampling an edge in a {\em random} graph.

In what follows we let $G \sim G\left(n,\frac{c}{n}\right)$, $e = \{u,v\} \in \binom{V(G)}{2}$ be arbitrary and let $G^{\circ},S^{\circ},P^{\circ},R^{\circ}$ and $W^{\circ}$ for $\circ \in \{+, -\}$ be as in the previous section.

The proof of the following lemma is an adaption of the proof of \cite[Lemma 7 (b)]{A23}. It states that whp there are no large red-purple components in $G^+$ or $G^-$. For completeness, the proof is given in \cref{a:comp_lem}.

\begin{lemma} \label{lem:no_large_red_purple_comps}
    Let $\cE_\mathrm{comp}$ be the event that every component in $G^+ - S^+$ or in $G^- - S^-$ has size at most $\log^4 n$. Then $\prob{\cE_\mathrm{comp}^c} = o(n^{-2})$. 
\end{lemma}

The following lemma states that $$W^*\coloneqq W^- \cup W^+ \cup \{u,v\}$$ is unlikely to be large in a strong way. Together with \cref{lem:D_cont_in_W}, we will use this to prove \cref{lem:var_X_X_k_diff_bound} via \cref{lem:ESapp}.
\begin{lemma} \label{lem:expec_W^*_bounds}
If $k \geq 100$ is fixed, then
    \begin{align*}
        \expec{|W^*|^2 \ind{|W^*| \geq k}} \leq 32 k^2 \left(\frac{3}{4}\right)^{k-2} + o(1) \quad\text{and}\quad \expec{|W^*|^2} \leq \tB + o(1),
    \end{align*}
    as $n \rightarrow \infty$, where $\tB > 0$ is an absolute constant.
\end{lemma}
\begin{proof}
    Recall that by \Cref{lem:no_large_red_purple_comps}, $\prob{\cE_\mathrm{comp}^c} = o(n^{-2})$. Hence, by linearity of expectation
    \begin{align}
    \expec{|W^*|^2 \ind{|W^*| \geq k}} &= \expec{|W^*|^2 \ind{|W^*| \geq k}  \indev{\cE_\mathrm{comp}}} + \expec{|W^*|^2 \ind{|W^*| \geq k}  \indev{\cE_\mathrm{comp}^c}} \nonumber\\
    &\leq \expec{|W^*|^2 \ind{|W^*| \geq k}  \indev{\cE_\mathrm{comp}}} + \prob{\cE_\mathrm{comp}^c} n^2 \nonumber\\
    &\leq \expec{|W^*|^2  \ind{|W^*| \geq k} \indev{\cE_\mathrm{comp}}} + o(1). \label{e:W*conditionalexpec}
    \end{align}
    
    Since, by \cref{lem:W^*_properties}, $W^*$ is connected in $G^+$ and contains at least $\frac{|W^*|-2}{4}$ vertices with no neighbour outside of $W^*$, it follows that for any $k \geq 100$,
    \begin{align*}
    \expec{|W^*|^2  \ind{|W^*| \geq k} \indev{\cE_\mathrm{comp}}} &\leq  \sum_{j=k}^{4 \log^4 n} j^2 \binom{n-2}{j-2} j^{j-2} p^{j-2} \binom{j}{\lceil\frac{j-2}{4}\rceil} (1-p)^{\frac{(j-2)(n-j)}{4}} \\
    &\leq  \sum_{j=k}^{\infty} 8j^2 \left(\frac{2ejnp}{j-2} e^{-\frac{c}{4}} \right)^{j-2}\\
    &\leq  \sum_{j=k}^{\infty} 8j^2 \left(\frac{k}{k-2}\cdot 2c e^{1-\frac{c}{4}} \right)^{j-2}\\
    &\leq 8k^2\left(\frac{k}{k-2}\cdot 2c e^{1-\frac{c}{4}} \right)^{k-2}\sum_{j=k}^{\infty} \left(\frac{j}{k}\right)^2\left(\frac{k}{k-2}\cdot 2c e^{1-\frac{c}{4}} \right)^{j-k}.
    \end{align*}
    Now note that $\left(\frac{j}{k}\right)^2 = \left(1+\frac{j-k}{k}\right)^2 \leq e^{\frac{2(j-k)}{k}}$ and that, since $k \geq 100$ and $c\geq 20$, we have
    $e^{\frac{2}{k}} \cdot \frac{k}{k-2} \cdot 2c e^{1-\frac{c}{4}} \leq e^{\frac{2}{100}} \cdot \frac{100}{98} \cdot 40 e^{-4} \leq \frac{3}{4}$. 
    Hence, we obtain
    \begin{align}
    \expec{|W^*|^2  \ind{|W^*| \geq k} \indev{\cE_\mathrm{comp}}} &\leq 8k^2 \left(\frac{3}{4}\right)^{k-2} \sum_{j=k}^{\infty} \left(e^{\frac{2}{k}}\frac{k}{k-2}\cdot 2c e^{1-\frac{c}{4}}\right)^{j-k}  \nonumber\\
    &\leq 8k^2 \left(\frac{3}{4}\right)^{k-2} \sum_{j=0}^{\infty} \left(\frac{3}{4}\right)^j \nonumber\\
    &= 32 k^2 \left(\frac{3}{4}\right)^{k-2}. \label{e:W*conditionalexpecbound}
    \end{align}
    The first part of the statement now follows from \eqref{e:W*conditionalexpec} and \eqref{e:W*conditionalexpecbound}.

    Furthermore, it follows from \eqref{e:W*conditionalexpecbound} that
    \begin{align*}
        \expec{|W^*|^2} &\leq \expec{|W^*|^2 \ind{|W^*| < 100} \indev{\cE_\mathrm{comp}}} + \expec{|W^*|^2 \ind{|W^*| \geq 100} \indev{\cE_\mathrm{comp}}}
+\expec{|W^*|^2\indev{\cE_\mathrm{comp}^c}} \\
        &\leq \sum_{j=2}^{99} j^2  + 32\cdot 100^2 \cdot \left(\frac{3}{4}\right)^{98} + o(1).
    \end{align*}
    Hence, the second part follows by setting $\tB \coloneqq \sum_{j=2}^{99} j^2  + 32\cdot 100^2 \cdot \left(\frac{3}{4}\right)^{98}$.
\end{proof}

\section{Variance bounds via Efron--Stein inequality}\label{s:ESvariance}
In this section we will use the results in \cref{s:strong4coreproperties} together with the Efron--Stein inequality (in the form of \cref{lem:ESapp}) to bound the variance of various random variables.
We prove upper bounds on the variance of $\tX - \tX_k$ (\cref{lem:var_X_X_k_diff_bound}), the variance of $\tX_k - \hX_k$ (\cref{lem:var_L_k_hL_k_diff_bound}), and the variance of $\tX_k$ (\cref{lem:Var_X_k_upper_bound}). 
The upper bounds on $\variance{\tX - \tX_k}$ and $\variance{\tX_k - \hX_k}$ will together imply that $\hX_k$ is a good approximation for $\tX$. The upper bound on $\variance{\tX_k}$ will be used to show that the bounds on $\variance{\tX - \tX_k}$ and $\variance{\tX_k - \hX_k}$ imply bounds on $\left|\variance{\tX} - \variance{\tX_k}\right|$ and $\left|\variance{\tX_k} - \variance{\hX_k}\right|$ via \cref{thm:prob_Cauchy_Schwarz}.

Throughout this section let $G \sim G\left(n,\frac{c}{n}\right)$, let $f = \left\{u, v\right\} \in \binom{V(G)}{2}$ be a fixed edge, and let $G^+ \coloneqq G+f$ and $G^- \coloneqq G-f$. Let $L,\tL,\tL_k,\hL_k$ and $X,\tX,\tX_k,\widehat{X}_k$ be defined as in \cref{s:Proof}. 

We start by giving a proof of \Cref{lem:var_X_X_k_diff_bound}, which bounds the variance of $\tX - \tX_k$. 
\begin{proof}[Proof of \Cref{lem:var_X_X_k_diff_bound}]
Recall that, by definition
\[
 \variance{\tX-\tX_k}=\frac{1}{n}\variance{\tL-\tL_k},
\]
and $\tL-\tL_k = \sum_{w \in V(G)} \left( \phi_{G}(w)-\phi_{G,k}(w) \right)$. Note further that, since $\phi_{G}(w),\phi_{G,k}(w) \in [0,1]$, it follows that $|\phi_{G}(w)-\phi_{G,k}(w)| \leq 1$ for all $w \in V(G)$. Hence, if we define
\[
D \coloneqq \{w \in V(G) \colon \phi^+(w)-\phi_k^+(w) \neq \phi^-(w) - \phi_k^-(w)\}
\]
as in \eqref{eq:def_D}, then by \Cref{lem:ESapp}
\begin{align}
\variance{\tX-\tX_k}&=\frac{1}{n}\variance{\tL-\tL_k} \nonumber\\
&\leq 2 p(1-p) n \expec{|D|^2}\nonumber \\
&\leq 2c  \expec{|D|^2}. \label{eq:VarboundexpecD}
\end{align}

By \cref{lem:D_cont_in_W,lem:if_W_small_D_nothing}, we have $|D| \leq |W^*| \ind{|W^*| \geq k}$, where $W^* \coloneqq W^- \cup W^+ \cup \{u,v\}$.  Together with \cref{lem:expec_W^*_bounds}, we have
\begin{align*}
    \expec{|D|^2} &\leq \expec{|W^*|^2 \ind{|W^*| \geq k}} \leq 32 k^2 \left(\frac{3}{4}\right)^{k-2} + o(1).
\end{align*}

Hence, by \cref{eq:VarboundexpecD}, we obtain
\begin{align*}
    \variance{\tX-\tX_k} &\leq 64ck^2 \left(\frac{3}{4}\right)^{k-2} + o(1),
\end{align*}
as desired.
\end{proof}

We also need the following bound on the variance of $\tX_k$ which is derived in a similar way.
\begin{lemma} \label{lem:Var_X_k_upper_bound}
    There exists an absolute constant $B^* > 0$ such that for all fixed $k \geq 100$ and $c \geq 20$, we have $$\variance{\tX_k} \leq cB^* + o(1),$$ as $n \rightarrow \infty$.
\end{lemma}
\begin{proof}
    Fix $k \geq 100$.
    Recall that $\tD \coloneqq \{ w \in V(G) \colon \phi_k^+(w) \neq \phi_k^-(w)\}$ as in \cref{eq:def_tD}, where we note that $\phi_{G,k}(w) \in [0,1]$ for all $w \in V(G)$. By applying \Cref{lem:ESapp} to $\tL$, as in the proof of  \cref{lem:var_X_X_k_diff_bound}, we get that 
    \begin{align} \label{eq:Var_X_k_bound}
        \variance{\tX_k} \leq 2c \expec{|\tD|^2}.
    \end{align}
    By \cref{lem:D_cont_in_W}, we have $|\tD| \leq |W^*|$, where $W^* \coloneqq W^- \cup W^+ \cup \{u,v\}$. Together with \cref{lem:expec_W^*_bounds}, we have
    \begin{align*}
        \expec{|\tD|^2} \leq \expec{|W^*|^2} \leq \tB +o(1),
    \end{align*}
    where $\tB > 0$ is an absolute constant. 
    Together with \cref{eq:Var_X_k_bound}, it follows that $\variance{\tX_k} \leq 2c \tB + o(1)$. Hence, we are done by setting $B^* \coloneqq 2\tB$.
\end{proof}

Finally, as mentioned in the introduction, our bounds on the variance of $\tX_k - \hX_k$ will hold for a much broader class of random variables, coming from the difference between (rescalings of) a weighted neighbourhood sum and an appropriate truncation thereof. Let us define a \emph{weighted ($k$th)-neighbourhood sum} to be a random variable of the form
\[
\tM_ k= \sum_{v \in V(G)} \sum_{(H,r) \in \mathcal{H}} \beta_{(H,r)} \mathds{1}_{\{(G[B_G(v,k)],v) \cong (H,r)\}},
\]
where $\mathcal{H}$ is the set of rooted graphs $(H,r)$ of radius at most $k$ and $\beta_{(H,r)} \in [0,1]$ is a constant depending only on $(H,r)$. The \emph{truncation} of $\tM_k$ is the random variable
\begin{equation}\label{e:truncation}
    \widehat{M}_k ~\coloneqq~ \sum_{v \in V(G)}\sum_{(T,r) \in \mathcal{T}} \beta_{(T,r)} \mathds{1}_{\{(G[B_G(v,k)],v) \cong (T,r)\}},
\end{equation}
where $\mathcal{T}$ is the set of rooted trees $(T,r) \in \mathcal{H}$ with $1 \leq v(T) \leq t_k = \tbound$. We will show the following lemma, which clearly implies \Cref{lem:var_L_k_hL_k_diff_bound}.

\begin{lemma} \label{lem:var_diff_with_trunc}
Let $\tM_ k$ be a weighted neighbourhood sum and let $\widehat{M}_k$ be its truncation. Then for fixed $k \geq 100$ and $c > 1$
    \begin{align*}
        \variance{\tM_k - \widehat{M}_k} \leq n\exp\left(-c^2k^2\right) + O\left(1\right),
    \end{align*}
    as $n \rightarrow \infty$.
\end{lemma}

In order to prove \Cref{lem:var_diff_with_trunc} we will require the following technical lemma about the typical neighbourhood structure in sparse random graphs, which will imply that the effect of moving to a truncation is neglible, whose proof is given in \cref{a:neighbourhood_approx}.
\begin{lemma}\label{lem:NUNUCexpectation}
For fixed $k \geq 100$, $c > 1$ and $t_k \coloneqq \tbound$, let 
\begin{align*}
U_N &~\coloneqq~ \{w \in V(G) \colon |B_{G^+}(w,k)| > t_k\}, \\ 
U_C &~\coloneqq~ \{w \in V(G) \colon G^+[B_{G^+}(w,k)] \text{ contains a cycle}\}, \text{ and} \\
N^* &~\coloneqq~ B_{G}(u,k) \cup B_{G}(v,k).
\end{align*}
Then 
\begin{enumerate}[label = \upshape{(\roman*)}]
\item\label{i:N*large} $\expec{|N^*|^2 \ind{|N^*| > t_k}} \leq 2k\exp\left(-2c^2k^2\right)$;
\item\label{i:N*capUNUC} $\expec{|N^* \cap (U_N \cup U_C)|} \leq 10k \exp\left(-2c^2k^2\right) +O\left(\frac{1}{n}\right).$
\end{enumerate}
\end{lemma}

We are now ready to prove \Cref{lem:var_diff_with_trunc}.

\begin{proof}[Proof of \Cref{lem:var_diff_with_trunc}]
Let us suppose that 
\[
\tM_ k= \tM_ k(G) = \sum_{v \in V(G)} \sum_{(H,r) \in \mathcal{H}} \beta_{(H,r)} \mathds{1}_{\{(G[B_G(v,k)],v) \cong (H,r)\}},
\]
where $\beta_{(H,r)} \in [0,1]$ and let us define for $w \in V(G)$
\[
\psi_{G,k}(w) = \sum_{(H,r) \in \mathcal{H}} \beta_{(H,r)} \mathds{1}_{\{(G[B_G(w,k)],v) \cong (H,r)\}},
\]
so that $\tM_ k=\sum_{w \in V(G)} \psi_{G,k}(w)$. Note that, since $\beta_{(H,r)} \in [0,1]$, it follows that $\psi_{G,k}(w) \in [0,1]$. Then
\[
\widehat{M}_k =\widehat{M}_k(G) =  \sum_{w \in V(G)}\psi_{G,k}(w)  \ind{|B_G(w,k)| \leq t_k} \ind{G[B_G(w,k)] \text{ is a tree}}:= \sum_{w \in V(G)} \widehat{\psi}_{G,k}(w),
\]
where similarly $\widehat{\psi}_{G,k} \in [0,1]$.

Hence, by applying \Cref{lem:ESapp} to $\tM_k - \widehat{M}_k$, it follows that 
    \begin{align} \label{eq:var_L_k_hL_k_diff_bound}
        \variance{\tM_k - \widehat{M}_k} \leq 2cn \expec{|D^*|^2}, 
    \end{align}
    where 
    \begin{align*}
        D^* \coloneqq \{w \in V(G) \colon \psi_{G^+,k}(w) - \widehat{\psi}_{G^+,k}(w) \neq \psi_{G^-,k}(w) - \widehat{\psi}_{G^-,k}(w)\}.
    \end{align*}
    
        It remains to bound $\expec{|D^*|^2}$. To this end let us define, as in \cref{lem:NUNUCexpectation}, 
    \begin{align*}
        U_N &~\coloneqq~ \{w \in V(G) \colon |B_{G^+}(w,k)| > t_k\},\\ 
        U_C &~\coloneqq~ \{w \in V(G) \colon G^+[B_{G^+}(w,k)] \text{ contains a cycle}\}, \text{ and }\\
        N^* &~\coloneqq~ B_G(u,k) \cup B_G(v,k).
    \end{align*}
We claim that 
\begin{align}\label{eq:D*subset}
        D^* \subseteq N^* \cap (U_N \cup U_C).
    \end{align}

    Indeed, suppose that $w \in V(G) \setminus (U_N \cup U_C)$. Then $|B_{G^-}(w,k)| \leq |B_{G^+}(w,k)| \leq t_k$ and $G^-[B_{G^-}(w,k)]$ and $G^+[B_{G^+}(w,k)]$ are trees. Hence, $\psi_{G^+,k}(w) = \widehat{\psi}_{G^+,k}(w)$ and $\psi_{G^-,k}(w) =  \widehat{\psi}_{G^-,k}(w)$. Thus, $w \notin D^*$. It follows that $D^* \subseteq U_N \cup U_C$. 

    On the other hand, suppose that $w \in V(G) \setminus N^*$. 
    Then $G^+[B_{G^+}(w,k)] = G^-[B_{G^-}(w,k)]$ and thus $\psi_{G^+,k}(w) = \psi_{G^-,k}(w)$ and $\widehat{\psi}_{G^+,k}(w) = \widehat{\psi}_{G^-,k}(w)$. Hence, $w \notin D^*$. It follows that $D^* \subseteq N^*$. Hence, \eqref{eq:D*subset} follows.

    From \eqref{eq:D*subset} we can conclude that
    \begin{align} \label{eq:D^*_expec_bound}
            \expec{|D^*|^2} &\leq \expec{|N^* \cap (U_N \cup U_C)|^2 \ind{|N^*| > t_k}} + \expec{|N^* \cap (U_N \cup U_C)|^2 \ind{|N^*| \leq t_k}} \nonumber\\
            &\leq \expec{|N^*|^2 \ind{|N^*| > t_k}} + t_k\expec{|N^* \cap (U_N \cup U_C)|}.
    \end{align}

    Hence, by \eqref{eq:var_L_k_hL_k_diff_bound}, \eqref{eq:D^*_expec_bound}, and \Cref{lem:NUNUCexpectation}, 
    \begin{align*}
         \variance{\tM_k - \widehat{M}_k} &\leq 2cn \expec{|D^*|^2}\\
        &\leq n\left(4ck\exp\left(-2c^2k^2\right) + 20ck \tbound \exp\left(-2c^2k^2\right)\right) +O\left(1\right) \\
        &\leq n\exp\left(-c^2k^2\right) +O\left(1\right),
    \end{align*}
    since $t_k \coloneqq \tbound$, $k\geq 100$ and $c>1$, as desired.
\end{proof}

\section{\texorpdfstring{A lower bound on $\variance{\tL_k}$}{A lower bound on the variance of Lk}}\label{s:varlowerbound}

In this subsection we will prove that the variance of $\tL_k$ is linear in $n$, where we recall that $\tL_k = \tL_k(G) \coloneqq n-\sum_{v\in V(G)} \phi_{G,k}(v)$. 

The heuristic idea is relatively simple, but surprisingly hard to implement rigorously due to the non-monotonicity of the strong $4$-core. Let us briefly sketch our approach.
Suppose we expose $G \sim G(n,p)$ via two rounds with probabilities $p_1$ and $p_2$ where $p_1,p_2 = \Theta\left(\frac{1}{n} \right)$ but $p_2$ is much smaller than $p_1$. Let $G_1 \sim G(n,p_1)$ and $G_2 \sim G(n,p_2)$, so that $G = G_1 \cup G_2$.

Since $p_1 \approx \frac{c}{n}$, for large $c$ we expect the strong $4$-core $S_1 = S(G_1)$ to cover a large proportion of the vertices, but equally we expect there to be a linearly sized set $A'$ of isolated vertices in $G_1$. Furthermore, since $p_2 = \Theta\left(\frac{1}{n} \right)$ it is relatively easy to verify that we expect there to be some linear-sized subset $A \subseteq A'$  with $N_{G_2}(A) : =B$ which satisfies the following properties:
\begin{itemize}
\item $A$ is independent in $G_1$;
\item $B \subseteq S(G_1)$;
\item $m \coloneqq e_{G_2}(A,B) = \Omega(n)$;
\item $|A| = \Omega(n)$;
\item $\frac{m}{|A|} < \frac{1}{100}$.
\item Each $b \in B$ is incident to at most one edge in $E_{G_2}(A,B)$.
\end{itemize}

If we condition on this likely event, and the set of edges of $G_1 \cup G_2$ outside of $E_G(A,B)$,
then by symmetry the number of edges incident to each vertex $a \in A$ can be modelled by placing $m$ balls into $|A|$ bins. 

Let us assume for simplicity's sake that $B$ is so \emph{robustly} sapphire in $G_1$ that exposing the edges $E_{G_2}(A,B)$ will not change the colouring outside of $A \cup B$. In this case, the contribution to $\Phi(G)$ from vertices outside of $A \cup B$ will not be changed by the edges $E_{G_2}(A,B)$. Furthermore, it can be seen (see \cref{prop:H^*_vs_H}) that the contribution to $\Phi(G)$ from the red-purple component containing a vertex $a \in A_2$ depends only on the number of edges of $E_{G_2}(A,B)$ which are incident to $a$. If this number is at most $1$ the contribution is $1$, and if this number is at least $2$ then the contribution is $0$.

However, standard calculations (see \cref{lem:balls_in_bins}) will show that the variance of the random variable which counts the number of bins containing at least two balls when $m$ balls are placed in $|A|$ bins at random is at least $\frac{m^2}{3|A|} = \Omega(n)$. Due to the likely existence of sets $A,B$ as above, this will be enough to bound the variance of $\tL$ from below by a linear function of $n$.

Of course, there are many simplifications in the strategy above. Firstly, although this is not a big difference, in the above sketch we are working with $\tL$ rather than $\tL_k$. Secondly, and more crucially, it is not clear that we can guarantee the existence of sets $A$ and $B$ with the desired properties to make this sketch work, or even what properties we should require in terms of the robustness of the colouring of $B$.

Let us make the above sketch precise. Throughout this section we will assume that $c \geq 20$, $k \geq 100$, and that $n$ is large enough. We will prove the following lower bound on the variance of $\tL_k$.

\begin{lemma} \label{lem:L_k_var_lower_bound}
    For every $c \geq 20$, $k \geq 100$, and $n$ large enough, $$\variance{\tL_k} \geq 2A_c n,$$ where $A_c \coloneqq \frac{1}{12} e^{-20c} > 0$.
\end{lemma}
The proof of \cref{lem:L_k_var_lower_bound}, which we give in \cref{sec:proof_of_var_lower_bound}, will rely on the following notion of a \emph{robustly} sapphire set. Given a graph $H$ and a subset $B \subseteq V(H)$, we say that $B$ has the \emph{robust sapphire property for $H$} if the following hold. 
\begin{enumerate}[label = \upshape{(RS\arabic*)}, leftmargin= \widthof{RS100000}]
    \item For every $x \in B \cup N_H(B)$, we have $\{x\} \cup N_H(x) \subseteq S(H)$ and $d_H(x) \geq 5$. \label{item:RS1}
    \item For every $b_1, b_2 \in B$ with $b_1 \neq b_2$, we have $\dist_H(b_1, b_2) \geq 5$. \label{item:RS2}
\end{enumerate}

Given a graph $H$ and disjoint subsets $A,B\subseteq V(H)$, we say that $(H,A,B)$ has the property~$\mathcal{P}$ (written $(H,A,B) \in \mathcal{P}$) if the following hold, where $H^* \coloneqq H - E_H(A,B)$.
\begin{enumerate}[label = \upshape{(P\arabic*)}, leftmargin= \widthof{P100000}]
    \item $E_H(A) = E_H(B) = \varnothing$. \label{item:P1}
    \item $H[A \cup B]$ is a union of stars and $|N_H(b) \cap A| \leq 1$ for each $b \in B$. \label{item:P2}
    \item $A$ consists of isolated vertices in $H^*$. \label{item:P3}
    \item $B$ has the robust sapphire property for $H^*$. \label{item:P4}
\end{enumerate}

The next lemma states that if $(H,A,B) \in \cP$, then $L_k(H) - L_k(H^*)$ is precisely the number of stars in $H[A,B]$ with at least $2$ edges.
This will allow us to show that if we condition on $E(H^*)$, then the remaining variance will be the same as the variance of the number of bins with at least $2$ balls in a balls in bins model. 

\begin{lemma}\label{lem:adjacment}
    Let $k\geq 2$. Let $H$ be a graph and let $A, B \subseteq V(H)$ be disjoint such that $(H,A,B) \in \cP$. Let $Y$ be the number of stars in $H[A,B]$ with at least $2$ edges. Let $H^*\coloneqq H - E_H(A,B)$. Then,
  \begin{equation}\label{eq:phichwerries}
      \tL_k(H)= \tL_k(H^*) +Y.
  \end{equation}
\end{lemma}
We prove \cref{lem:adjacment} in \cref{sec:robust_sapphire}.
In order to use this to prove \cref{lem:L_k_var_lower_bound}, we will need the following bound on the variance of the number of bins with at least $2$ balls in a balls in bins model. The proof of this is given in \cref{a:balls_in_bins}.

\begin{lemma} \label{lem:balls_in_bins}
    Let $N$ and $m$ be positive integers such that $m = \Theta(N)$ as $N \rightarrow \infty$ and $\frac{m}{N} < \frac{1}{100}$. Consider a balls and bins configuration with $N$ bins and $m$ balls, that is, each of the $m$ balls is assigned independently and uniformly at random to one of the $N$ bins.
    Let $Z$ be the number of bins with at least $2$ balls. Then, for $N$ large enough,
    \[
        \variance{Z} = h\left(\frac{m}{N}\right)N +O(1) \geq \frac{1}{3}\left(\frac{m}{N}\right)^2 N,
    \]
    as $N \rightarrow \infty$, where $h(x) \coloneqq (1+x)e^{-x} -((1+x)^2+x^3) e^{-2x}$.
\end{lemma}

As mentioned in the sketch proof above, we will expose the edges of $G$ in two rounds.
Let $p_2 \coloneqq \frac{e^{-5c}}{n}$ and let $p_1 \in (0,1)$ be such that $1-p=(1-p_1)(1-p_2)$. Let $G_1\sim G(n,p_1)$ and $G_2\sim G(n,p_2)$ be independent random graphs. Then, we can couple $G$ and $(G_1,G_2)$ such that $G=G_1\cup G_2$.
To bound the variance of $\tL_k$ we will implement an \emph{edge revealing process}, that is, a process that reveals partial information about the random graphs $G_1$ and $G_2$ to find two vertex disjoint sets, which will play the roles of $A$ and $B$ from the sketch proof above. The next lemma makes precise the properties we want these sets to have. We will prove this in \cref{subsec:edge_reveal}.

\begin{lemma} \label{lem:edge_reveal_alg_exists}
    There exists an edge revealing process that outputs two vertex disjoint sets $A_4, B_4 \subseteq V(G)$ such that the following hold, where $G^* \coloneqq G - E_{G}(A_4,B_4))$.
    \begin{enumerate}[label = \upshape{(R\arabic*)}, leftmargin= \widthof{R100000}]
    \item $(G,A_4,B_4)\in \mathcal{P}$. \label{item:Reveal1}
    \item $E(G_1[A_4,B_4])=\varnothing$. \label{item:Reveal2}
    \item The process reveals precisely the edges of $G^*$, that $E(G_1[A_4,B_4])=\varnothing$, that $|N_{G_2}(b) \cap A_4| \leq 1$ for each $b \in B_4$, and the number of edges in $G_2[A_4,B_4]$, which we denote by $m$. \label{item:Reveal3}
    \item $\prob{|A_4| \geq \frac{e^{-c}n}{2} \text{ and } e^{-10c}n \leq m \leq e^{-5c}n} \geq \frac{1}{2}$. \label{item:Reveal4}
\end{enumerate}
\end{lemma}
Note that \cref{item:Reveal2} implies that $G^* = G_1 \cup (G_2 - E_{G_2}(A_4,B_4))$, that is, $G^*$ consists of all edges of $G$ except the ones from $A_4$ to $B_4$ that belong to $G_2$.

\subsection{Proof of Lemma \ref{lem:L_k_var_lower_bound}} \label{sec:proof_of_var_lower_bound}
By \cref{lem:edge_reveal_alg_exists}, there exists an edge revealing process such that \cref{item:Reveal1,item:Reveal2,item:Reveal3,item:Reveal4} hold. Let $A_4$ and $B_4$ be the two vertex disjoint subsets of $V(G)$ that the process outputs.
For disjoint $U, W \subseteq V(G)$ let $\mathcal{F}(U,W)$ be the set of subsets of edges of $\binom{V(G)}{2} \setminus \{\{x,y\} \colon x \in U, y \in W\}$. Furthermore for disjoint $U,W\subseteq V(G)$, $F\in \mathcal{F}(U,W)$, and $0\leq m\leq |U||W|$ we let 
$\mathcal{E}(U,W,F,m)$ be the event that 
\begin{align*}
    A_4 = U,\quad B_4 = W, \quad E(G)\setminus E(G[U,W]) = F, \quad\text{and}\quad e(G_2[U,W]) = m.
\end{align*}
In addition, we let $Y(U,W,m)$ be the number of stars in $G[U, W]$ with at least $2$ edges, conditioned on the event $\mathcal{E}(U,W,F,m)$. Note that if $\mathcal{E}(U,W,F,m)$ holds, then $F = E(G^*)$ and thus $\tL_k(G^*)$ is determined by $F$.
Hence, \cref{lem:adjacment} implies that there exists a deterministic $x(F)$ such that in the event $\mathcal{E}(U,W,F,m)$ one has 
\begin{align} \label{eq:x_F}
    \tL_k(G)=x(F)+Y(U,W,m).
\end{align}
Let $\mathcal{Q}$ be the set of $4$-tuples $(U, W, F, m)$, where $U, W \subseteq V(G)$ are disjoint, $F \in \cF(U,W)$, and $0 \leq m \leq |U||W|$ are such that $\prob{\cE(U,W,F,m)} >0$, $|U| \geq \frac{e^{-c}n}{2}$, and $e^{-10c}n \leq m \leq e^{-5c}n$. By the law of total expectation, it follows that
\begin{align*}
&\variance{\tL_k}=\expec{\left(\tL_k-\expec{\tL_k}\right)^2} \\ 
&\quad \overset{\mathclap{\cref{eq:x_F}}}{=} \sum_{\substack{(U,W,F,m) \\ \prob{\cE(U,W,F,m)}>0}} \prob{\mathcal{E}(U,W,F,m)} 
\expec{\left(x(F)+Y(U,W,m)-\expec{\tL_k}\right)^2\bigg|\mathcal{E}(U,W,F,m)} \\
&\quad\geq \sum_{(U,W,F,m) \in \cQ} \prob{\mathcal{E}(U,W,F,m)} 
\expec{\left(x(F)+Y(U,W,m)-\expec{\tL_k}\right)^2\bigg|\mathcal{E}(U,W,F,m)}
\\& \quad\geq \sum_{(U,W,F,m) \in \cQ} \prob{\mathcal{E}(U,W,F,m)} \variance{Y(U,W,m)\bigg|\mathcal{E}(U,W,F,m)},
\end{align*}
where the second inequality follows from the fact that, conditioned on $\mathcal{E}(U,W,F,m)$, $\expec{\tL_k} - x(F)$ is a constant and that the constant $a$ that minimises 
\[
    \expec{(Y(U,W,m) - a)^2 \bigg|\mathcal{E}(U,W,F,m)}
\]
is $a = \expec{Y(U,W,m)\bigg|\mathcal{E}(U,W,F,m)}$.
By \cref{item:Reveal4}, we have
\begin{equation}
    \sum_{(U,W,F,m) \in \cQ} \prob{\mathcal{E}(U,W,F,m)}   \geq \frac{1}{2}.
\end{equation}  
Thus, it now suffices to prove that for all $(U,W,F,m) \in \cQ$,
\begin{equation}\label{eq:goal_exp}
\variance{Y(U,W,m)\bigg|\mathcal{E}(U,W,F,m)} \geq 4 A_c n,
\end{equation} 
where $A_c \coloneqq \frac{1}{12} e^{-20c} > 0$.
Observe that $\variance{Y(U,W,m)\bigg|\mathcal{E}(U,W,F,m)}$ equals the variance of the number of bins with at least $2$ balls in a balls in bins configuration where each of the $m$ balls is placed in one of the $|U|$ many bins uniformly at random.
Using the bounds on $m$ given by $(U,W,F,m) \in \cQ$, we have by \cref{lem:balls_in_bins} that the variance of this quantity can be lower bounded by $\frac{m^2}{3|U|} \geq 4A_c n$, where $A_c \coloneqq \frac{1}{12} e^{-20c} > 0$. 
This completes the proof of \cref{lem:L_k_var_lower_bound}.

\subsection{The robust sapphire property} \label{sec:robust_sapphire}
In this subsection we will prove \cref{lem:adjacment}. In order to do so we need to understand how adding the edges in $E(H) \setminus E(H^*)$ to $H^*$ changes its strong $4$-core. We first prove the following proposition which tells us how a set of vertices with the robust sapphire property can limit changes to the strong $4$-core when adding more edges.

\begin{proposition} \label{prop:RS_retained}
    Let $H^*$ be a graph and $A, B \subseteq V(H^*)$ be disjoint such that the vertices in $A$ are isolated in $H^*$ and $B$ has the robust sapphire property for $H^*$. Let $F \subseteq \{ab \colon a \in A, b \in B\}$ and $H \coloneqq (V(H^*), E(H^*) \cup F)$. 
    Let $B_{=1} \coloneqq \{b \in B \colon |N_H(b) \cap A| =1\}$, $A' \coloneqq \{a \in A \colon d_H(a) \geq 4, N_H(a) \subseteq B_{=1}\}$, and $B' \coloneqq N_H(A \setminus A')$.
    Then $(S(H^*) \setminus B') \cup A' \subseteq S(H)$. In addition, $B \setminus N_H(A)$ has the robust sapphire property for $H$.
\end{proposition}
\begin{proof}
    Let $\tS \coloneqq (S(H^*) \setminus B') \cup A'$. We first show that $\tS \subseteq S(H)$. It suffices to show that $\tS$ has the strong $4$-core property \cref{Strong_4-core_property} for $H$. Observe that $B \subseteq S(H^*)$, since $B$ has the robust sapphire property for $H^*$. Note that, since $B' \cup N_H(A') \subseteq B \subseteq S(H^*)$,
    \begin{align*}
        N_H(\tS) \subseteq N_H(S(H^*) \setminus B') \cup N_H(A') \subseteq N_{H^*}(S(H^*)) \cup A' \cup B 
        \subseteq  N_{H^*}(S(H^*)) \cup  S(H^*) \cup A'.
    \end{align*} 
    Hence,
    \begin{align} \label{eq:tS_cup_NtS_inclusion}
        \tS \cup N_H(\tS) \subseteq S(H^*) \cup N_{H^*}(S(H^*)) \cup A'.
    \end{align}
    Let $x \in \tS \cup N_H(\tS)$. We show that $|N_H(x) \cap \tS| \geq 4$. First suppose that $x \in A'$. By the definition of $A'$, we have $N_H(A')\subseteq B_{=1}$. Thus, by the definition of $B_{=1}$, we have $N_H(A')\subseteq B$ is disjoint from $N_H(A\setminus A') \eqqcolon B'\subseteq B\setminus B_{=1}$. By the definition of $A'$, it now follows that $x$ has at least $4$ neighbours in $B_{=1}\setminus B' \subseteq B\setminus B'\subseteq S(H^*) \setminus B' \subseteq \tS$, and so $|N_H(x) \cap \tS| \geq 4$.
    
    Now assume that $x \notin A'$ and hence, $x \in S(H^*) \cup N_{H^*}(S(H^*))$ by \cref{eq:tS_cup_NtS_inclusion}.
    By the definition of $S(H^*)$, we have
    \begin{align}
        |N_{H^*}(x) \cap S(H^*)| \geq 4.
    \end{align}
    If $x \notin B \cup N_{H^*}(B)$, then $N_{H^*}(x) = N_{H^*}(x) \setminus B$ and so 
    \begin{align*}
        |N_H(x) \cap \tS| \geq |N_H(x) \cap (S(H^*) \setminus B)| \geq |(N_{H^*}(x) \setminus B) \cap S(H^*)| = |N_{H^*}(x) \cap S(H^*)| \geq 4.
    \end{align*}
    So, we may assume that $x \in B \cup N_{H^*}(B)$. Note that, since $B$ has the robust sapphire property for $H^*$, we have $|N_{H^*}(x) \cap S(H^*)| \geq 5$ and $|N_{H^*}(x) \cap B| \leq 1$.
    Hence,
    \begin{align*}
        |N_H(x) \cap \tS| &\geq |N_H(x) \cap (S(H^*) \setminus B)| \geq |N_{H^*}(x) \cap (S(H^*) \setminus B)| \\ 
        &\geq |N_{H^*}(x) \cap S(H^*)| - |N_{H^*}(x) \cap B| \geq 5-1 =4.
    \end{align*}
    We have shown that $\tS$ has the strong $4$-core property \cref{Strong_4-core_property} for $H$ and hence, $\tS \subseteq S(H)$. 

    It remains to show that $\tB \coloneqq B \setminus N_H(A)$ has the robust sapphire property for $H$. We first note that $\dist_{H}(b,b') \geq 5$ for all $b \in B$ and $b' \in \tB$. Indeed, since $\tB\subseteq B$, by \cref{item:RS2} $\dist_{H^*}(b,b') \geq 5$. So, let $P$ be a path from $b$ to $b'$ in $H$ that uses an edge of $F$. Note that, $P$ contains a path in $H^*$ from $b$ to a vertex in $B \setminus \tB$. Since $B$ has the robust sapphire property for $H^*$, this implies that $P$ has length at least $5$. Hence, $\dist_{H}(b,b') \geq 5$. In particular, since $\tB\subseteq B$, \cref{item:RS2} is satisfied.

   Furthermore, it follows that $(\tB \cup N_H(\tB) \cup N_H(N_H(\tB))) \cap B' = \varnothing$ and, since $B$ has the robust sapphire property for $H^*$, $\tB \cup N_H(\tB) \cup N_H(N_H(\tB)) \subseteq B \cup N_{H^*}(B) \cup N_{H^*}(N_{H^*}(B))$. Hence, $\tB \cup N_H(\tB) \cup N_H(N_H(\tB)) \subseteq  S(H^*) \setminus B' \subseteq S(H)$. 
    
    Finally, since $B$ has the robust sapphire property for $H^*$, we have $d_H(x) = d_{H^*}(x) \geq 5$ for all $x \in \tB \cup N_H(\tB) \subseteq (B \cup N_{H^*}(B))\setminus B'$. Hence, \cref{item:RS1} is satisfied.
\end{proof}

Next, we show precisely how the strong $4$-core changes when going from $H^*$ to $H$ (in the setting of \cref{lem:adjacment}).

\begin{proposition} \label{prop:H^*_vs_H}
    Let $H$ be a graph and let $A, B \subseteq V(H)$ be disjoint such that $(H,A,B) \in \cP$. Let $H^*\coloneqq H - E_H(A,B)$, $A' \coloneqq \{x \in A \colon d_{H}(x) \geq 4\}$ and $B' \coloneqq N_H(A \setminus A') \subseteq B$. Then the following hold.
    \begin{enumerate}[label = \upshape{(\roman*)}]
        \item $S(H) = (S(H^*) \setminus B') \cup A'$. \label{item:S(H)}
        \item $P(H) = P(H^*) \cup B'$. \label{item:P(H)}
        \item $R(H) = R(H^*) \setminus A'$. \label{item:R(H)}
    \end{enumerate}
\end{proposition}
\begin{proof}
    Let $\tS \coloneqq (S(H^*) \setminus B') \cup A'$.
    Note that, since $(H,A,B) \in \cP$, by \cref{item:P2}, we have $B_{=1} \coloneqq \{b \in B \colon |N_H(b) \cap A| = 1\} =B$. Thus,
    \[
    A' \coloneqq \{a \in A \colon d_H(a) \geq 4\} = \{a \in A \colon d_H(a) \geq 4, N_H(a) \subseteq B_{=1}\}.
  \]
  Hence, by \cref{prop:RS_retained}, we have $\tS \subseteq S(H)$.

    Next, we show that $S(H) \subseteq \tS$. Let $W \coloneqq A \cup B'$ and $U \coloneqq S(H^*) \setminus B'$. Note that $B \subseteq S(H^*)$ and so $B \subseteq (S(H^*)\setminus B') \cup B' \subseteq W \cup U$. Since also $A \subseteq W \cup U$, we have $E(H) \triangle E(H^*) \subseteq \binom{A \cup B}{2} \subseteq \binom{W \cup U}{2}$. Note that $U \subseteq S(H) \cap S(H^*)$, since $U \coloneqq S(H^*) \setminus B' \subseteq \tS \subseteq S(H)$. Observe that 
    \begin{align*}
        N_{H^*}(W) \subseteq N_{H^*}(B) \subseteq S(H^*) \setminus B \subseteq U
    \end{align*}
    and
    \begin{align*}
        N_H(W) \subseteq (B \setminus B') \cup (N_H(B) \setminus A) \subseteq (S(H^*)\setminus B') \cup N_{H^*}(B) \subseteq S(H^*) \setminus B' = U.
    \end{align*}
    Hence, by \cref{lem:sapphire_cut} with $U$ and $W$ playing the roles of $A$ and $B$, respectively, we have $S(H) \setminus W = S(H^*) \setminus W = S(H^*) \setminus B'$. Thus, $S(H) \subseteq (S(H^*)\setminus B') \cup W$. It follows by the definition of $A'$, that $A \setminus A' \subseteq R(H)$ and thus $B' \subseteq P(H)$. 
    Hence, $W \setminus (R(H) \cup P(H)) \subseteq A'$. 
    It follows that
    \begin{align*}
        S(H) \subseteq (S(H^*) \setminus B') \cup A' \eqqcolon \tS.
    \end{align*}
    Together with the reverse inclusion established above, we have proved \cref{item:S(H)}.

    To prove \cref{item:P(H)}, note that 
    \begin{align*}
        P(H) \coloneqq N_H(S(H)) = N_H((S(H^*) \setminus B') \cup A') = N_{H^*}(S(H^*)) \cup B' = P(H^*) \cup B'.
    \end{align*}

    To prove \cref{item:R(H)}, note that 
    \begin{align*}
        R(H) &\coloneqq V(H) \setminus (S(H) \cup P(H)) = (V(H^*) \setminus (S(H^*) \cup P(H^*))) \setminus (A \setminus A') \\
        &= R(H^*) \setminus A'. 
    \end{align*}
    This completes the proof.
\end{proof}

Finally, we are ready to prove \cref{lem:adjacment}.

\begin{proof}[Proof of \cref{lem:adjacment}]
    By \cref{prop:H^*_vs_H}, we have
    \begin{align}\label{eq:global_colouring}
        S(H) = (S(H^*) \setminus B') \cup A', \quad P(H) = P(H^*) \cup B', \quad \text{and} \quad R(H) = R(H^*) \setminus A', 
    \end{align}
    where $A' \coloneqq \{x \in A \colon d_H(x) \geq 4\}$ and $B' \coloneqq N_H(A\setminus A')$. 
    Note that, using \cref{item:P4}, \[
    U \coloneqq N_H(A \cup B) = N_{H^*}(A \cup B) = N_{H^*}(B) \subseteq S(H^*) \setminus B \subseteq S(H).
    \]
    and hence,
    \begin{align} \label{U_in_both_S}
        U \subseteq S(H) \cap S(H^*).
    \end{align}
    Let $w \in V(H) \setminus (A \cup B)$.     Since $E(H) \triangle E(H^*) \subseteq \binom{A \cup B \cup U}{2}$, by \cref{lem:sapphire_cut_local} with $U$ and $A \cup B$, playing the roles of $A$ and $B$, respectively, we have 
    \begin{align} \label{eq:phi_outside_A_B}
        \phi_{H,k}(w) = \phi_{H^*,k}(w), \text{ for } w \in V(H) \setminus (A \cup B).
    \end{align}

    Note that, since by \cref{item:P3} $A$ is a set of isolated vertices in $H^*$, we have $A \subseteq R(H^*)$. Moreover, by \cref{item:P4}, it follows that $B \subseteq S(H^*)$.
    For $x \in A \cup B$, we denote by $C_x$ and $C_x^*$ the components in $H[P(H) \cup R(H)]$ and $H^*[P(H^*) \cup R(H^*)]$ that contain $x$, respectively. 
    Note that for $x \in A$, $C_x^* = \{x\}$ and $x \in R(H^*)$ and thus, by \cref{obs:C_small}, $\phi_{H^*,k}(x) = 1$. Moreover, for $x \in B$, $x \in S(H^*)$ and by \cref{obs:In_S_globally_implies_locally}, $x \in S_k(H^*,x)$ implying $\phi_{H^*,k}(x) = 0$. It follows that 
    \begin{align} \label{eq:phi_A_B_H^*}
        \sum_{x \in A \cup B} \phi_{H^*,k}(x) = |A|.
    \end{align}
    
    By \cref{item:P1,item:P2} and by \cref{eq:global_colouring,U_in_both_S}, for each $x \in A \cup B$ we either have $x \in S(H)$ or $H[C_x]$ is a star with centre at a vertex in $A$ and all leaves in $B$. Moreover, since $k \geq 2$, we have by \cref{obs:C_small}, that $\phi_{H,k}(x) = \phi_H(x)$ for each $x \in A \cup B$. Let $A'_0 \coloneqq \{x \in A' \colon d_H(x) = 0\}$, $A'_1 \coloneqq \{x \in A' \colon d_H(x) = 1\}$, and $B'_1 \coloneqq N_H(A'_1)$. 
    Now, note that for $x \in A_0'$, we have $C_x = \{x\}$ and $x \in R(H)$ and thus $\phi_{H,k}(x) = \phi_H(x) = 1$. 
    For $x \in A'_1 \cup B'_1$, we see that $H[C_x]$ is a single edge with one endpoint in each of $R(H)$ and $P(H)$ and thus $\phi_{H,k}(x) = \phi_H(x) = \frac{1}{2}$. Let $x \in (A' \cup B') \setminus (A'_0 \cup A'_1 \cup B'_1)$. Then $H[C_x]$ is a $K_{1,2}$ or a $K_{1,3}$ whose central vertex is in $R(H)$ and whose leaves are in $P(H)$. Hence, $\phi_{H,k}(x) = \phi_H(x) = 0$ as the vertex in $R(H)$ can be covered by a path of length $2$ with endpoints in $P(H)$. Finally, for $x \in (A \cup B) \setminus (A' \cup B')$, we have $x \in S(H)$ and thus $\phi_{H,k}(x) = \phi_H(x) = 0$. 
    It follows that 
    \begin{align*}
        \sum_{x \in A \cup B} \phi_{H,k}(x) = |A'_0| + \frac{|A'_1|}{2} + \frac{|B'_1|}{2} = |A'_0| + |A'_1|. 
    \end{align*}
    Note that $|A| - |A'_0| -|A'_1| = Y$, and thus 
    \begin{align} \label{eq:phi_A_B_H}
        \sum_{x \in A \cup B} \phi_{H,k}(x) = |A| - Y. 
    \end{align}
    Hence, we have
    \begin{align*}
        \tL_k(H) &\coloneqq n - \sum_{x \in V(H)} \phi_{H,k}(x) = n - \sum_{x \in V(H) \setminus (A \cup B)} \phi_{H,k}(x) - \sum_{x \in A \cup B} \phi_{H,k}(x) \\
        \overset{\cref{eq:phi_outside_A_B}, \cref{eq:phi_A_B_H}}&{=} n -  \sum_{x \in V(H) \setminus (A \cup B)} \phi_{H^*,k}(x) -|A| + Y \\
        \overset{\cref{eq:phi_A_B_H^*}}&{=} n - \sum_{x \in V(H^*)} \phi_{H^*,k}(x) + Y = \tL_k(H^*) + Y,
    \end{align*}
        where the penultimate equality also uses that $V(H^*) = V(H)$.
    This completes the proof.
\end{proof}

\subsection{Proof of Lemma \ref{lem:edge_reveal_alg_exists}} \label{subsec:edge_reveal}
We first describe the edge revealing process and then show that it satisfies \cref{item:Reveal1,item:Reveal2,item:Reveal3,item:Reveal4}. For ease of notation, given a graph $G$ and $v\in V(G)$, we let $N_G[v]:=\{v\} \cup N_{G}(v)$.

\begin{breakablealgorithm}
    \caption{Edge revealing process} \label{alg:edge_revealing_process}
\begin{algorithmic}[1] 
    \MyInput{$c \geq 20$, $k \geq 100$, and $n\in \mathbb{N}$.}
    \MyOutput{Random sets $A_4$ and $B_4$ such that $(G,A_4,B_4) \in \cP$.}
    \State Reveal $E(G_1)$ and set $A_0 \coloneqq \{v \in V(G) \colon d_{G_1}(v) = 0\}$.

    \State Reveal all edges of $G_2[A_0]$ and set $A_1 \coloneqq \{v \in V(G) \colon N_G(v) \cap A = \varnothing\}$.

    \State Reveal all edges of $G_2[V(G) \setminus A_1]$.

    \State Set $G_0 \coloneqq (V(G), E(G_1) \cup E(G_2[V(G)\setminus A_1])\cup E(G_2[A_1]))$. \Comment{$G_0$ is the graph of all edges revealed so far.}

    \State Use \cref{alg:strong_4_core_alg} to determine $S(G_0)$.

    \State Set $B_0 \coloneqq \{b \in V(G) \colon \text{Each } x \in N_{G_0}[b] \text{ satisfies } N_{G_0}[x] \subseteq S(G_0) \text{ and } d_{G_0}(x) \geq 5 \}$. 
    \LComment{Note that $B_0$ is the maximal set of vertices in $G_0$ such that \cref{item:RS1} holds with $G_0$ and $B_0$ playing the roles of $H$ and $B$, respectively.}

    \State Set $B_1 \coloneqq \varnothing$.

    \While{There exists $x \in B_0 \setminus B_1$ such that $\dist_{G_0}(x,B_1) \geq 5$.}
        \State Add $x$ to $B_1$.
    \EndWhile

    \LComment{Note that $B_1$ is a maximal subset of $B_0$, with the property that no two vertices in $B_1$ are within distance at most $4$ in $G_0$. Observe that $B_1$ has the robust sapphire property for $G_0$.}

    \State Set $A_2 \coloneqq A_1$, $B_2 \coloneqq B_1$, and $G^* \coloneqq G_0$.

    \While{True.}
        \State Reveal $E(G_2[A_2, V(G) \setminus B_2])$. 

        \State $G^* \coloneqq G^* \cup G_2[A_2, V(G) \setminus B_2]$. \Comment{$G^*$ is the graph of all edges revealed so far.}

        \State $A^* \coloneqq \{a \in A_2 \colon d_{G^*}(a) = 0\}$.

        \State Use \cref{alg:strong_4_core_alg} to determine $S(G^*)$.

        \State $B_0^* \coloneqq \{b \in B_2 \colon \text{Each } x \in N_{G^*}[b]  \text{ satisfies } N_{G^*}[x] \subseteq S(G^*) \text{ and } d_{G^*}(x) \geq 5\}$.

        \State $B^* \coloneqq \varnothing$.

        \While{There exists $x \in B_0^* \setminus B^*$ such that $\dist_{G^*}(x,B^*) \geq 5$.}
            \State Add $x$ to $B^*$.
        \EndWhile

        \LComment{Note that $B^*$ is a maximal subset of $B_2$ that has the robust sapphire property for the graph of all edges revealed so far.}
        
        \If{$A^* = A_2$ and $B^* = B_2$}
            \State Break. \Comment{Ends the while-loop once $A_2$ and $B_2$ have `stabilised'.}
        \EndIf

        \State Set $A_2 \coloneqq A^*$ and $B_2 \coloneqq B^*$.
    \EndWhile

    \State Set $A_3 \coloneqq A_2$, $B_3 \coloneqq B_2$, and $B_3^- \coloneqq \varnothing$.

    \For{$b \in B_3$}
        \State Query if $|N_G(b) \cap A_3| \geq 2$. 
        
        \LComment{We only reveal if $|N_G(b) \cap A_3| \geq 2$, not the corresponding edges.}
        
        \If{$|N_G(b) \cap A_3| \geq 2$}
            \State Add $b$ to $B_3^-$.
        \EndIf
    \EndFor

    \State Reveal $E(G[A_3, B_3^-])$. 

    \State Set $G^* \coloneqq G^* \cup G[A_3, B_3^-]$. \Comment{$G^*$ is the graph of all edges revealed so far.}

    \State Set $A_4 \coloneqq A_3 \setminus N_{G^*}(B_3^-)$.

    \LComment{Note that $A_4 = A_3 \setminus N_G(B_3^-)$ since all edges incident to $B_3^-$ have been revealed.}

    \State Reveal $E(G[A_3\setminus A_4, B_3])$.
    
    \State Set $G^* \coloneqq G^* \cup G[A_3\setminus A_4, B_3]$. \Comment{$G^*$ is the graph of all edges revealed so far.} 
    \State Set $B_4 \coloneqq B_3 \setminus N_{G^*}(A_3 \setminus A_4)$.

    \LComment{Note that $B_4 = B_3 \setminus N_G(A_3 \setminus A_4)$, since all edges incident to $A_3 \setminus A_4$ have been revealed.}
    \State Reveal the number of edges $m$ in $G_2[A_4,B_4]$. \Comment{We do not reveal the corresponding edges.}
    
    \State \Return $A_4$, $B_4$, and $m$.
\end{algorithmic}
\end{breakablealgorithm}  
We first show that \cref{item:Reveal1} holds, that is, $(G, A_4, B_4) \in \cP$. Note that \cref{item:P1,item:P2,item:P3} clearly hold by construction. We now show that \cref{item:P4} holds, that is, that $B_4$ has the robust sapphire property for $G^* \coloneqq G- E_G(A_4,B_4)$. 
  
By construction, $B_3$ has the robust sapphire property for $G - E_G(A_3,B_3)$ and $A_3$ is isolated in $G - E_G(A_3,B_3)$. Hence, by \cref{prop:RS_retained} with $G - E_G(A_3, B_3)$, $G^*$, 
$A_3$, and $B_3$ playing the roles of $H^*$, $H$, $A$, and $B$, respectively, it follows that $B_3 \setminus N_{G^*}(A_3)$ has the robust sapphire property for $G^*$. We claim that $B_3 \setminus N_{G^*}(A_3) = B_4$, and hence, $B_4$ has the robust sapphire property for $G^*$ as claimed.
Indeed, since $B_4 = B_3 \setminus N_G(A_3 \setminus A_4)$, we have $E_G(A_3 \setminus A_4, B_4) = \varnothing$. We now show that $E_G(A_4, B_3 \setminus B_4) = \varnothing$. Suppose for a contradiction, that there is an edge $ab \in E(G)$ with $a \in A_4$ and $b \in B_3 \setminus B_4$. Since $b \in B_3 \setminus B_4$, we have $b \in N_G(A_3 \setminus A_4)$. Let $a' \in A_3 \setminus A_4$ such that $a'b \in E(G)$. Hence, $|N_G(b) \cap A_3| \geq 2$, so $b \in B_3^-$. This implies that $a \in N_G(B_3^-)$, a contradiction to our assumption that $a \in A_4 = A_3 \setminus N_G(B_3^-)$.
Since $E_G(A_3 \setminus A_4, B_4) = \varnothing$ and $E_G(A_4, B_3 \setminus B_4) = \varnothing$, we have $G^* \coloneqq G - E_G(A_4,B_4)$ contains all the edges of $G$ incident to $A_3 \setminus A_4$ but none of the edges of $G$ incident to $A_4$. It follows that $B_4 = B_3 \setminus N_G(A_3 \setminus A_4) = B_3 \setminus N_{G^*}(A_3)$.

Now, note that \cref{item:Reveal2} is satisfied, since $A_4$ is contained in $A_0$, the set of isolated vertices of $G_1$.
Moreover, note that by construction \cref{alg:edge_revealing_process} satisfies \cref{item:Reveal3}.

It remains to prove that \cref{item:Reveal4} holds.
Every isolated vertex of $G$ belongs to $A_4$. Thus, by a simple application of McDiarmid's inequality~(\cref{Biased_McDiarmid_ineq}), whp $|A_4| \geq \frac{e^{-c}n}{2}$.
Moreover, $m \leq e(G_2) \leq 2p_2\frac{n^2}{2} \leq e^{-5c}n$ holds with high probability by a Chernoff bound.
To show that \cref{item:Reveal4} holds, it now suffices to prove that $m\geq e^{-10c}n$ with probability at least $\frac{7}{10}$. To that end, we first prove that whp $|B_1|\geq \frac{n}{c^7}$, and then that $|B_3|\geq \frac{n}{2c^7}$ with probability at least $\frac{4}{5}$. We start with the following (standard) bound on the size of the neighbourhood of a set of vertices.
\begin{claim}\label{lem:neighbourhoodsize}
Whp for every set $W\subseteq V(G)$, we have $|N_G(W)|\leq \max\{e^{-2c}n,(2c \log c)|W|\}$.
\end{claim}
\begin{proofclaim}
    The probability that there exists $W\subseteq V(G)$ of size at least $e^{-3c}n$ with $|N_G(W)| > (2c\log c)|W|$ is at most
\begin{align*}
\sum_{j\geq e^{-3c}n}    \binom{n}{j} \binom{n}{(2c\log c)j}j^{(2c\log c)j} p^{(2c\log c)j}
&\leq \sum_{j\geq e^{-3c}n } \bigg( \frac{en}{j} \bfrac{enp}{2c\log c}^{2c\log c} \bigg)^j\\ &\leq \sum_{j\geq e^{-3c}n} \bigg( e^{3c+1} \bfrac{e}{2\log c}^{2c\log c} \bigg)^j =o(1).
\end{align*} 
Thus, whp $|N_G(W)| \leq (2c\log c)|W|$ for every $W \subseteq V(G)$ with $|W| \geq e^{-3c}n$. Now, suppose this event holds and let $W \subseteq V(G)$ with $|W| < e^{-3c}n$. Let $W' \subseteq V(G)$ be such that $W \subseteq W'$ and $|W'| = \lceil e^{-3c}n\rceil$. Then, since $N_G(W) \subseteq W \cup N_G(W) \subseteq W' \cup N_G(W')$,
\begin{align*}
    |N_G(W)| &\leq |W| + |N_G(W)| \leq |W'| + |N_G(W')| \leq (2c \log c +1) |W'| \\
    &= (2c \log c +1) \lceil e^{-3c}n\rceil \leq e^{-2c}n.
\end{align*}
This completes the proof of the claim. 
\end{proofclaim}

We also need a bound on the number of vertices of a given degree in $G_0$.
\begin{claim} \label{lem:deg_in_G_0}
    Let $0 \leq i \leq 6$. Then whp the number of vertices of degree $i$ in $G_0$ is at most \[\left(1+\frac{1}{10}\right)\frac{c^ie^{-c}}{i!}n.\]
\end{claim}
\begin{proofclaim}
    Fix $0 \leq i \leq 6$ and set $Y_i$ and $\tY_i$ to be the number of vertices of degree $i$ in $G_0$ and $G$, respectively. Let $Z$ be the number of vertices that are adjacent to at least one edge of $G_2$. Note that $Y_i \leq \tY_i + Z$. Observe that $\expec{\tY_i} = n\binom{n-1}{i}p^{i}(1-p)^{n-1-i} = (1+o(1)) \frac{c^i e^{-c}}{i!} n$ and $\expec{Z} \leq n (n-1)p_2 \leq e^{-5c}n$. Since adding or removing an edge changes the value of $\tY_i$ and $Z$ by at most $2$, by McDiarmid's inequality~(\cref{Biased_McDiarmid_ineq}), whp $\tY_i \leq \left(1+\frac{1}{20}\right)\frac{c^ie^{-c}}{i!}n$ for all $i \in [6]$ and $Z \leq \left(1+\frac{1}{20}\right)e^{-5c}n$. Hence, whp $Y_i \leq \left(1+\frac{1}{10}\right)\frac{c^ie^{-c}}{i!}n$ for all $i \in [6]$.
\end{proofclaim}

Next, we prove a lower bound on $|B_1|$. We aim to show that 
\begin{align}\label{eq:B_1_bound}
    \textit{whp} \qquad  |B_1|\geq \frac{n}{c^7}.
\end{align}
Let $V_4$ be the set of vertices of degree $4$ in $G_0$. The key observation is that  
\begin{equation}\label{eq:Bcontainment}
 V(G)\setminus \left(R(G_0)\cup P(G_0)\cup N_{G_0}(P(G_0))\cup N_{G_0}(N_{G_0}(P(G_0))) \cup V_4\cup N_{G_0}(V_4)\right) \subseteq B_0.     
\end{equation}
Indeed, suppose that 
\[
    x \in  V(G)\setminus \left(R(G_0)\cup P(G_0)\cup N_{G_0}(P(G_0))\cup N_{G_0}(N_{G_0}(P(G_0))) \cup V_4\cup N_{G_0}(V_4)\right).
\]
Since $x \notin R(G_0) \cup P(G_0)\cup N_{G_0}(P(G_0))\cup N_{G_0}(N_{G_0}(P(G_0)))$, we have $N_{G_0}[y] \subseteq S(G_0)$ for every $y \in N_{G_0}[x]$.
Since $x \notin V_4 \cup N_{G_0}(V_4)$ and since all vertices of degree $3$ or fewer in $G_0$ are in $R(G_0)$, we have $d_{G_0}(y) \geq 5$ for all $y \in N_{G_0}[x]$.
It follows that $x \in B_0$. Thus, we have proved \cref{eq:Bcontainment}.

\begin{claim} \label{claim:RP_etc_upper} Whp
    \[
        |R(G_0)| + |P(G_0)| + |N_{G_0}(P(G_0))| + |N_{G_0}(N_{G_0}(P(G_0)))| + |V_4| + |N_{G_0}(V_4)| \leq \frac{3}{5}n.
    \]
\end{claim}
\begin{proofclaim}
    Note that if $R(G_0) = \varnothing$, then also $P(G_0) = \varnothing$. So we may assume without loss of generality that $R(G_0) \neq \varnothing$. Let $r \coloneqq |P(G_0)|/|R(G_0)|$. Note that by \cref{obs:red_vxs_in_comp}, $|P(G_0)| \leq 3 |R(G_0)|$ and so $0 < r \leq 3$.
    Note further that $P(G_0) = N_{G_0}(R(G_0))$ and thus $r|R(G_0)| = |P(G_0)| = |N_{G_0}(R(G_0))| \leq e_{G_0}(R(G_0), V(G) \setminus R(G_0))$.
    Moreover, note that $R(G_0)$ is incident with at most $3|R(G_0)|$ edges in $G_0$ by \cref{obs:red_vxs_in_comp}. Hence,
    \begin{align*}
        e_{G_0}(R(G_0)) + r|R(G_0)| \leq e_{G_0}(R(G_0)) + e_{G_0}(R(G_0), V(G) \setminus R(G_0)) \leq 3|R(G_0)|.
    \end{align*}
    It follows that $e_{G_0}(R(G_0)) \leq (3-r)|R(G_0)|$. We deduce that 
    \begin{align} \label{eq:total_deg_upper}
        \sum_{x \in R(G_0)} d_{G_0}(x) &= 2 e_{G_0}(R(G_0)) + e_{G_0}(R(G_0), V(G) \setminus R(G_0)) \nonumber\\
        &\leq (3-r)|R(G_0)| + 3|R(G_0)| = (6-r)|R(G_0)|.
    \end{align} 
    For $0 \leq i \leq 7$, we let $d_i$ and $d_{\geq i}$ be the number of vertices in $R(G_0)$ with degree equal to $i$ and at least $i$ in $G_0$, respectively. 
    Note that 
    \begin{align} \label{eq:R(G_0)_upper}
        |R(G_0)| = \sum_{i = 0}^5 d_i + d_{\geq 6} \quad \text{and} \quad |R(G_0)| = \sum_{i = 0}^6 d_i + d_{\geq 7}
    \end{align}
    and
    \begin{align} \label{eq:total_deg_lower}
        \sum_{x \in R(G_0)} d_{G_0}(x) \geq \sum_{i=0}^5 i d_i + 6 d_{\geq 6} \quad \text{and} \quad \sum_{x \in R(G_0)} d_{G_0}(x) \geq \sum_{i=0}^6 i d_i + 7 d_{\geq 7}.
    \end{align}
    Combining \cref{eq:total_deg_lower,eq:total_deg_upper}, we obtain 
    \begin{align} \label{eq:total_deg_comb}
        6d_{\geq 6} \leq (6-r)|R(G_0)| - \sum_{i = 0}^5 i d_i \quad \text{and} \quad 7d_{\geq 7} \leq (6-r)|R(G_0)| - \sum_{i = 0}^6 i d_i.
    \end{align}
    Combining \cref{eq:total_deg_comb,eq:R(G_0)_upper} gives
    \begin{align*}
        6|R(G_0)| = \sum_{i=0}^5 6d_i + 6d_{\geq 6} \leq \sum_{i=0}^5 (6-i)d_i + (6-r)|R(G_0)|
    \end{align*}
    and
    \begin{align*}
        7|R(G_0)| = \sum_{i=0}^6 7d_i + 7 d_{\geq 7} \leq \sum_{i=0}^6 (7-i) d_i +(6-r)|R(G_0)| \leq \sum_{i=0}^6 (7-i) d_i + 6|R(G_0)|.
    \end{align*}
    From these inequalities, we deduce
    \begin{align*}
        |R(G_0)| \leq \frac{1}{r}\sum_{i=0}^5 (6-i)d_i \quad\text{and}\quad |R(G_0)| \leq \sum_{i=0}^6 (7-i) d_i.
    \end{align*}
    Since $|P(G_0)| = r |R(G_0)|$, we obtain
    \begin{align*}
        |P(G_0)| \leq \sum_{i=0}^5 (6-i)d_i \quad\text{and}\quad |R(G_0)| \leq \sum_{i=0}^6 (7-i) d_i.
    \end{align*}
    Together with \cref{lem:deg_in_G_0}, we have whp
    \begin{align} \label{eq:P(G_0)_upper_whp}
        |P(G_0)| \leq \sum_{i=0}^5 (6-i)\left(1+\frac{1}{10}\right)\frac{c^ie^{-c}}{i!}n \leq \frac{2c^5e^{-c}}{5!}n
    \end{align}
    and
    \begin{align} \label{eq:R(G_0)_upper_whp}
        |R(G_0)| \leq \sum_{i=0}^6 (7-i)\left(1+\frac{1}{10}\right)\frac{c^ie^{-c}}{i!}n \leq \frac{2c^6e^{-c}}{6!}n.
    \end{align}
    Combining \cref{lem:neighbourhoodsize} with \cref{eq:P(G_0)_upper_whp}, we obtain that whp
    \begin{align} \label{eq:NP_upper_whp}
        |N_{G_0}(P(G_0))|\leq |N_{G}(P(G_0))| \leq \frac{2 c^6(\log c)e^{-c}}{5!}n
    \end{align}
    and
    \begin{align} \label{eq:NNP_upper_whp}
        |N_{G_0}(N_{G_0}(P(G_0)))|\leq |N_{G}(N_{G}(P(G_0)))| \leq \frac{2 c^7(\log c)^2 e^{-c}}{5!}n.
    \end{align}
    Moreover, by \cref{lem:deg_in_G_0}, whp
    \begin{align} \label{eq:VNV_upper_whp}
        |V_4| + |N_{G_0}(V_4)| \leq 5 |V_4| \leq  5\left(1+\frac{1}{10}\right)\frac{c^4e^{-c}}{4!}n.
    \end{align}
    It follows from \cref{eq:R(G_0)_upper_whp,eq:P(G_0)_upper_whp,eq:NP_upper_whp,eq:NNP_upper_whp,eq:VNV_upper_whp} that whp
    \begin{align*}
        &|R(G_0)| + |P(G_0)| + |N_{G_0}(P(G_0))| + |N_{G_0}(N_{G_0}(P(G_0)))| + |V_4| + |N_{G_0}(V_4)| \\
        &\quad \leq \frac{2c^6e^{-c}}{6!}n + \frac{2c^5e^{-c}}{5!}n + \frac{2 c^6(\log c)e^{-c}}{5!}n + \frac{2 c^7(\log c)^2 e^{-c}}{5!}n + 5\left(1+\frac{1}{10}\right)\frac{c^4e^{-c}}{4!}n \\
        &\quad \leq \frac{3 c^7(\log c)^2e^{-c}}{5!}n \leq \frac{3}{5}n,
    \end{align*}
    as desired.
\end{proofclaim} 
By \cref{claim:RP_etc_upper} and \cref{eq:Bcontainment}, we have whp
\begin{align*}
    |B_0|&\geq n- \left(|R(G_0)| + |P(G_0)| + |N_{G_0}(P(G_0))| + |N_{G_0}(N_{G_0}(P(G_0)))| + |V_4| + |N_{G_0}(V_4)|\right)\\ &\geq \frac{2}{5}n.
\end{align*}
Since $\dist_G(b,B_1) \leq 4$ for all $b \in B_0$, whp the set $B_1$ has the property that 
\[
    |B_1|+|N_{G}(B_1)|+ |N_G(N_{G}(B_1))|+|N_G(N_G(N_{G}(B_1)))|+|N_G(N_G(N_G(N_{G}(B_1))))|\geq \frac{2}{5}n.
\]
As $\frac{n}{c^7}\geq e^{-3c}n$ and 
$\frac{1}{c^7}\sum_{i=0}^4(2c\log c)^i<\frac{2}{5}$ for $c\geq 20$, \cref{lem:neighbourhoodsize} implies that $|B_1|\geq \frac{n}{c^7}$, proving \cref{eq:B_1_bound}.  

We are now ready to prove the bound on $|B_3|$. We will show that
\begin{align}\label{eq:B_3_bound}
     |B_3|\geq \frac{n}{2c^7} \qquad \text{ with probability at least } \frac{4}{5}.
\end{align}
Let $U$ be the set of vertices in red/purple components of size at least 2 of $\tG \coloneqq G-E_G(A_3,B_3)$ that contain a vertex in $A_0$ (and hence, also an edge of $G_2$). 
\begin{claim} 
    Let $\mathcal{B}_{\mathrm{comp}}$ be the event that every component in $\tG[R(\tG) \cup P(\tG)]$ has size at most $\log^4 n$.
    Then $\prob{\mathcal{B}_{\mathrm{comp}}^c} = o(n^{-2})$.
\end{claim}
The proof of this claim is very similar to the proof of \cref{lem:no_large_red_purple_comps} (except that we upper bound the probability that an edge is not present by $1-p_1$ instead of $1-p$), so we omit it. 

Note that each component $C$ of $\tG[R(\tG) \cup P(\tG)]$ with $C \cap A_0 \neq \varnothing$ of size $k \geq 2$ contains at least $\frac{k}{4}$ vertices of $R(\tG)$ by \cref{obs:red_vxs_in_comp}. Those vertices have no neighbours in $\tG$ outside of $V(C)$. Moreover, since $C \cap A_0 \neq \varnothing$, we have that $C$ contains an edge of $G_2$.  
It follows that, 
\begin{align*}
    \expec{|U|} &= \expec{|U|\indev{\mathcal{B}_{\mathrm{comp}}}} + \expec{|U|\indev{\mathcal{B}_{\mathrm{comp}}^c}} \\
    &\leq \sum_{k = 2}^{\log^4 n} k \binom{n}{k}k^{k-2}p^{k-2} (kp_2) \binom{k}{\lceil \frac{k}{4} \rceil  }(1-p_1)^{\lceil \frac{k}{4} \rceil (n-k)}+ n \prob{\mathcal{B}_{\mathrm{comp}}^c}\\ 
    &\leq \sum_{k = 2}^{\log^4 n} (1+o(1)) \bfrac{en}{k}^k k^{k}p^{k}\bfrac{p_2}{p^2} 2^ke^{-p_1\lceil \frac{k}{4} \rceil n}+o(1)\\
    &\leq \sum_{k = 2}^{\log^4 n} (1+o(1)) \bfrac{p_2}{p^2} \left(2enp e^{-\frac{n}{4}(p-p_2)} \right)^k+o(1) \\ &\leq \sum_{k\geq 2}(1+o(1)) c^{-2}e^{-5c} \left(\frac{4}{5}\right)^kn+o(1)\leq \frac{1}{10}e^{-5c}n. 
\end{align*}
Markov's inequality implies that $|U|\leq e^{-5c}n$ with probability at least $\frac{9}{10}$.
Let $U_{\geq 1}^*$ be the set of vertices adjacent to at least one edge of $G_2$.
Note that $\expec{|U_{\geq 1}^*|} \leq n^2 p_2 = e^{-5c}n$. A simple application of McDiarmid's inequality~(\cref{Biased_McDiarmid_ineq}) gives that whp $|U_{\geq 1}^*| \leq 2e^{-5c}n$.
\begin{claim}
    Let $x \in B_1$ with $\dist_{G}(x, U \cup U_{\geq 1}^*) \geq 4$. Then $x \in B_3$.
\end{claim}
\begin{proofclaim}
    Let $y \in B_1$ with $y \neq x$. By the definition of $B_1$, we have $\dist_{G_0}(x,y) \geq 5$. Since $\dist_{G}(x, U_{\geq 1}^*) \geq 4$, any path from $x$ to $y$ that is in $G$ but not in $G_0$ (and thus contains an edge of $G_2$) is of length at least $5$. Hence, $\dist_{\tG}(x,y) \geq 5$.
    Note that since $\dist_{G}(x, U_{\geq 1}^*) \geq 4$, we have 
    \[
        \tG\left[  N_{\tG}[x] \cup N_{\tG}(N_{\tG}(x))\right] = G_0\left[N_{G_0}[x]  \cup N_{G_0}(N_{G_0}(x))\right].
    \]
    Thus, in particular, since $B_1$ has the robust sapphire property for $G_0$, we have 
    \begin{align} \label{eq:y_in_S}
      N_{\tG}[y] \subseteq S(G_0) \quad\text{and}\quad d_{\tG}(y) \geq 5, \quad\text{ for every } y \in   N_{\tG}[x].
    \end{align}
    Hence, it suffices to show that $N_{\tG}[x] \cup N_{\tG}(N_{\tG}(x)) \subseteq S(\tG)$. To that end, let $z \in N_{\tG}[x] \cup N_{\tG}(N_{\tG}(x))$. Note that by \cref{eq:y_in_S}, we have $z \in S(G_0)$. Moreover, since $\dist_G(x,z) \leq 2$ and $\dist_G(x,U) \geq 4$, we have $z \notin U$. By the definition of $U$ and by $z \notin A_0$, $z$ is contained in a component $C$ of $\tG[R(\tG) \cup P(\tG)]$ that does not contain a vertex of $A_0$. Let $W \coloneqq V(G) \setminus (C \cup N_{\tG}(C))$. Since all edges in $E(G_0) \triangle E(\tG)$ are incident to $A_0$, we have $N_{\tG}(W) = N_{G_0}(W) = N_{G_0}(C) = N_{\tG}(C)$ and $E(G_0) \triangle E(\tG) \subseteq \binom{W \cup N_{\tG}(C)}{2}$. Furthermore, since $C$ is a component of $\tG[R(\tG) \cup P(\tG)]$, we have $N_{\tG}(C) \subseteq S(\tG)$. Hence, by \cref{lem:sapphire_cut_onesided} with $N_{\tG}(C)$ and $W$ playing the roles of $A$ and $B$, respectively, we have $S(G_0) \cap C \subseteq S(\tG) \cap C$. Finally, since $z \in S(G_0) \cap C$, $z \in S(\tG)$.
\end{proofclaim}
 Thus, \cref{lem:neighbourhoodsize,eq:B_1_bound} imply that with probability at least $\frac{4}{5}$, we have
\begin{align*}
    |B_3| &\geq |B_1| - (|U\cup U_{\geq 1}^*| + |N_G(U\cup U_{\geq 1}^*)| + |N_G(N_G(U\cup U_{\geq 1}^*))| \\
    &\phantom{{}\geq |B_1| - (|U\cup U_{\geq 1}^*|}+ |N_G(N_G(N_G(U\cup U_{\geq 1}^*)))|)\\
    &\geq \frac{n}{c^7} - e^{-2c}n\sum_{i=0}^3 (2c\log c)^i  \geq  \frac{n}{2c^7}.
\end{align*}
This proves \cref{eq:B_3_bound}.

Finally, we will show that
\begin{align} \label{eq:m_bound}
    m\geq e^{-10c}n \qquad \text{ with probability at least } \frac{7}{10}.
\end{align}

Note that $m \coloneqq e_{G_2}(A_4,B_4)$ is bounded below by the number $m'$ of edge-components\footnote{By `edge-components' we mean components isomorphic to $K_2$.} in $G_2[A_3,B_3]$. Indeed, any edge that forms such a component in $G_2[A_3,B_3]$ is also in $G[A_4,B_4]$ by the definition of $A_4$ and $B_4$.
Observe that, by the time we define $A_3$ and $B_3$, we have not revealed anything about the edges in $G[A_3,B_3]$. Hence, for each $x \in A_3$ and $y \in B_3$, the edge $xy$ is in $G_2[A_3,B_3]$ independently with probability $p_2$.
Note further, that $|A_3|$ is bounded from below by the number of isolated vertices of $G$, thus whp $|A_3|\geq \frac{e^{-c}n}{2}$. Thereafter,
\begin{align*}
\condexpec{m'}{|A_3|\geq \frac{e^{-c}n}{2}, |B_3|\geq \frac{n}{2c^7}}&\geq \frac{e^{-c}n}{2} \frac{n}{2c^7} p_2 (1-p)^{2n}  \geq  \frac{e^{-c}n}{2} \frac{n}{2c^7} \frac{e^{-5c}}{n} \frac{e^{-2c}}{2}\geq \frac{e^{-9c}n}{100}.
\end{align*}
Observe that, adding or removing an edge from $G_2[A_3,B_3]$ changes $m'$ by at most $2$. Thus, a simple application of McDiarmid's inequality~(\cref{Biased_McDiarmid_ineq}) gives that whp $m \geq m' \geq e^{-10c}n$ conditioned on the events $|A_3|\geq \frac{e^{-c}n}{2}$ and $|B_3|\geq \frac{n}{2c^7}$.
Thus, by \cref{eq:B_3_bound}, $m\geq e^{-10c}n$ with probability at least $\frac{7}{10}$.   

This proves \cref{eq:m_bound} and thus \cref{lem:edge_reveal_alg_exists}.

\section{\texorpdfstring{Properties of $\hL_k$}{Properties of Lk}}\label{s:Lkprop}

In this section, we will prove that $\left(\frac{\variance{\hL_k}}{n}\right)^{\frac{1}{2}}$ converges to a constant $\sigma_k >0$ as $n \rightarrow \infty$ for each $k \in \mathbb{N}$ and the the sequence $\sigma_k$ converges as $k \rightarrow \infty$. This will allow us to prove \cref{lem:variancelimit}. Moreover, we will show using Stein's method that $\hL_k$ satisfies a central limit theorem (\Cref{lem:conv_by_stein}). As we showed in the proof of \cref{thm:Law_of_longest_cycle}, this then carries over into a central limit theorem for $L\left(G\left(n, \frac{c}{n}\right)\right)$.

\subsection{\texorpdfstring{The form of $\variance{\hL_k}$ and the convergence of $\frac{\variance{\hL_k}}{n}$ in $n$}{The form of the variance of Lk and the convergence of Lk/n}}

It will be convenient in what follows to view $\hL_k$ not as a weighted sum of neighbourhood counts, but rather by considering the contribution from each set $U \subseteq V(G)$ coming from vertices whose neighbourhood is $U$.

More concretely, we can rewrite $\hL_k$ as follows, where we recall that $\cT$ is the set of rooted trees with $1 \leq v(T) \leq t_k \coloneqq \tbound$ and radius at most $k$.
\begin{align*}
     \hL_k &= \sum_{v \in V(G)}\sum_{(T,r) \in \mathcal{T}} \alpha_{(T,r)} \mathds{1}_{\{(G[B_G(v,k)],v) \cong (T,r)\}}\\
     &=\sum_{\substack{U \subseteq V(G) \\ 1 \leq |U| \leq t_k}} \sum_{v \in U} \sum_{\substack{(T,r) \in \mathcal{T} \\ v(T) = |U|}} \alpha_{(T,r)} \mathds{1}_{\{B_G(v,k) = U\}} \mathds{1}_{\{(G[U],v) \cong (T,r)\}}\\
     &\coloneqq \sum_{\substack{U \subseteq V(G) \\ 1 \leq |U| \leq t_k}} \sum_{v \in U} \sum_{\substack{(T,r) \in \mathcal{T} \\ v(T) = |U|}} \alpha_{(T,r)} I(U,v,T,r),
\end{align*}
where we have defined $I(U,v,T,r)\coloneqq \mathds{1}_{\{B_G(v,k) = U\}} \mathds{1}_{\{(G[U],v) \cong (T,r)\}}$ as this term will occur frequently.
 Letting 
 \begin{align} \label{eq:Y_U_Def}
     Y_U \coloneqq \sum_{v \in U} \sum_{\substack{(T,r) \in \mathcal{T} \\ v(T) = |U|}} \alpha_{(T,r)} I(U,v,T,r),
 \end{align}
 and $\cU_k \coloneqq \{U \subseteq V(G) \colon 1 \leq |U| \leq t_k\}$,
 we have $\hL_k = \sum_{U \in \cU_k} Y_U$.

 We start by showing that $\variance{\hL_k} = O(n)$ and that the limit $\lim_{n \rightarrow \infty}\frac{\variance{\hL_k}}{n}$ exists. 
\begin{lemma} \label{lem:Var_is_Kn}
    There exists a constant $K(k,c) \in \mathbb{R}$ such that $\variance{\hL_k} = K(k,c)n + O(1)$.
\end{lemma}
In fact, we show that this hold more generally for any truncated weighted neighbourhood sum.
\begin{lemma} \label{lem:Var_is_Kn_general}
    Let $\widehat{M}_k \coloneqq \sum_{v \in V(G)}\sum_{(T,r) \in \mathcal{T}} \beta_{(T,r)} \mathds{1}_{\{(G[B_G(v,k)],v) \cong (T,r)\}}$ be a truncated weighted neighbourhood sum as in \cref{e:truncation} with weights $\bm{\beta} \coloneqq \left(\beta_{(T,r)}\right)_{(T,r) \in \cT}$\,.
    Then there exists a constant $K(k,c,\bm{\beta}) \in \mathbb{R}$ such that $\variance{\widehat{M}_k} = K(k,c,\bm{\beta})n + O(1)$.
\end{lemma}
Since the proof consists of standard, albeit reasonably involved, combinatorial estimates, we defer the details to \cref{a:VarhLk}.

However, despite the fact that we can give an explicit formula for the constant $K(k,c)$ in \cref{lem:Var_is_Kn}, it is not obvious that $K(k,c)\neq 0$. Since it will be important for the application of Stein's method that $\variance{\hL_k}$ is not too small, we will also want to establish a lower bound on $\variance{\hL_k}$, which we will do by relating it to $\variance{\tL_k}$, which we have already bounded from below in \cref{lem:L_k_var_lower_bound}.

\begin{lemma} \label{lem:var_asymptotics}
    For fixed $c \geq 20$, there exists $k_0(c) \geq 100$ such that for all fixed $k \geq k_0(c)$ the following hold, where $A_c$ and $B^*$ are as in \cref{lem:L_k_var_lower_bound} and \cref{lem:Var_X_k_upper_bound}, respectively.
    \begin{enumerate}[label = \upshape{(\roman*)}]
        \item $\variance{\hL_k} = \Theta(n)$. \label{item:hL_k_var_order_n}
        \item The limit $\sigma_k \coloneqq \lim_{n\rightarrow \infty}\left(\frac{\variance{\hL_k}}{n}\right)^{\frac{1}{2}}$ exists.
        \item  $0 < \left(A_c\right)^{\frac{1}{2}} \leq \sigma_k \leq \left(2cB^*\right)^{\frac{1}{2}}$.
    \end{enumerate}
\end{lemma}
\begin{proof}
    By \cref{lem:Var_is_Kn}, we have $\variance{\hL_k} = K(k,c)n+O(1) = O(n)$ and 
    $$\sigma_k \coloneqq \lim_{n\rightarrow \infty}\left(\frac{\variance{\hL_k}}{n}\right)^{\frac{1}{2}} = \sqrt{K(k,c)}.$$
    Note however, that \cref{lem:Var_is_Kn} does not rule out the possibility that $K(k,c) = 0$. So, let us give quantitative bounds on $\variance{\hL_k}$.

    By \cref{lem:L_k_var_lower_bound,lem:Var_X_k_upper_bound}, 
    \begin{align} \label{eq:var_L_k_bounds}
         2 A_cn \leq \variance{\tL_k} = n \variance{X_k} \leq cB^*n + o(n),
    \end{align}
    and by \cref{lem:var_L_k_hL_k_diff_bound}, we have
    \begin{align} \label{eq:var_L_k_hL_k_diff_bound_2}
        \variance{\tL_k - \hL_k} \leq \exp\left(-c^2k^2\right)n + O(1).
    \end{align}
    Hence, by \eqref{eq:var_L_k_bounds}, \eqref{eq:var_L_k_hL_k_diff_bound_2} and \cref{thm:prob_Cauchy_Schwarz}, we obtain
    \begin{align}\label{eq:Var_hL_L_diff}
    \left| \variance{\hL_k} - \variance{\tL_k} \right| &\leq \variance{\hL_k - \tL_k} + 2\left(\variance{\hL_k - \tL_k}\variance{\tL_k}\right)^{\frac{1}{2}} \nonumber \\
    &\leq \exp\left(-c^2k^2\right)n + O(1) + 2\left(cB^*\exp\left(-c^2k^2\right)+o(1)\right)^{\frac{1}{2}}n \nonumber\\
    &\leq \left( \min\left\{ A_c , cB^* \right\} +o(1) \right)n,
 \end{align}
 for all $k \geq k_0(c)$ if $k_0(c)$ is sufficiently large. In particular, by \eqref{eq:var_L_k_bounds} and \eqref{eq:Var_hL_L_diff}, it follows that
\[
\variance{\hL_k} = \Theta(n) \text{ and } \left(A_c\right)^{\frac{1}{2}} \leq \sigma_k \leq \left(2cB^*\right)^{\frac{1}{2}},
\]
 for all $k \geq k_0(c)$ if $k_0(c)$ is sufficiently large.
\end{proof}

\subsection{The convergence of \texorpdfstring{$\sigma_k$}{σk} in \texorpdfstring{$k$}{k}}

We first prove the following lemma which roughly speaking states that $\variance{\hL_k}$ and $\variance{\tL}$ are close to each other. This will allow us to show that the sequence of real numbers $(\sigma_k)_{k \geq 1}$ given by \cref{lem:var_asymptotics} has a limit.
\begin{lemma} \label{lem:var_diff_limit}
    For every fixed $c \geq 20$, there exists $\tk_0(c) \geq 100$ such that for all $k \geq \tk_0(c)$,
    \begin{align*}
        \lim_{n \rightarrow \infty} \frac{1}{n} \left|\variance{\hL_k} - \variance{\tL} \right| \leq 2\left(\frac{4}{5}\right)^{\frac{k}{2}}.
    \end{align*}
\end{lemma}
\begin{proof}
    Fix $c \geq 20$, let $\tk_0(c)$ be large enough, and let $k \geq \tk_0(c)$.
    Note that 
    \begin{align} \label{eq:lim_split}
        \lim_{n \rightarrow \infty} \frac{1}{n}\left|\variance{\hL_k} - \variance{\tL} \right| \leq \lim_{n \rightarrow \infty} \frac{1}{n}\left|\variance{\hL_k} - \variance{\tL_k} \right| + \lim_{n \rightarrow \infty} \frac{1}{n}\left|\variance{\tL_k} - \variance{\tL} \right|.
    \end{align}
    We bound these two terms as follows. For the second term, we observe that, by
    \cref{thm:prob_Cauchy_Schwarz}, \cref{lem:var_X_X_k_diff_bound}, and \cref{lem:Var_X_k_upper_bound},
    \begin{align} \label{eq:var_lim_2}
        \lim_{n \rightarrow \infty} \frac{1}{n}\left| \variance{\tL} -\variance{\tL_k} \right| 
        &\leq \lim_{n \rightarrow \infty} \left(\variance{\tX -\tX_k} + 2\left(\variance{\tX -\tX_k} \variance{\tX_k}\right)^{\frac{1}{2}} \right) \nonumber\\
        &\leq 64ck^2 \left(\frac{3}{4}\right)^{k-2} + 2\left(64ck^2 \left(\frac{3}{4}\right)^{k-2} cB^*\right)^{\frac{1}{2}} \nonumber\\
        &\leq \left(\frac{4}{5}\right)^{\frac{k}{2}},
    \end{align}
    if $\tk_0(c)$ is large enough compared to $B^*$.
    
    We bound the first term in \cref{eq:lim_split} similarly. By \cref{thm:prob_Cauchy_Schwarz}, \cref{lem:var_L_k_hL_k_diff_bound}, and \cref{lem:Var_X_k_upper_bound},
    \begin{align} \label{eq:var_lim_1}
        \lim_{n \rightarrow \infty} \frac{1}{n}\left|\variance{\hL_k} - \variance{\tL_k} \right|
        &\leq \lim_{n \rightarrow \infty}\left(\variance{\hX_k - \tX_k} +2\left(\variance{\hX_k - \tX_k} \variance{\tX_k}\right)^{\frac{1}{2}}\right) \nonumber\\
        &\leq \left(\frac{4}{5}\right)^{\frac{k}{2}}.
    \end{align}
    By \cref{eq:lim_split,eq:var_lim_1,eq:var_lim_2} we obtain the desired conclusion.
\end{proof}

We now show that $(\sigma_k)_{k \geq 1}$ has a limit.
\begin{lemma}\label{lem:limitsigmak}
    For every fixed $c \geq 20$, the sequence of real numbers $(\sigma_k)_{k \geq 1}$ given by \cref{lem:var_asymptotics} has a limit $\sigma$ as $k \rightarrow \infty$ with $0 < \left(A_c\right)^{\frac{1}{2}} \leq \sigma \leq \left(2cB^*\right)^{\frac{1}{2}}$, where $A_c$ and $B^*$ are as in \cref{lem:L_k_var_lower_bound} and \cref{lem:Var_X_k_upper_bound}, respectively.
\end{lemma}
\begin{proof}
    We show that $(\sigma_k)_{k \geq 1}$ converges by showing that $(\sigma_k^2)_{k \geq 1}$ is a Cauchy sequence. Let $k^* \geq \tk_0(c)$, where $\tk_0(c)$ is as in \cref{lem:var_diff_limit}.
    Let $k_1, k_2 \geq k^*$. Note that by \cref{lem:var_diff_limit},
    \begin{align*}
        |\sigma_{k_1}^2 - \sigma_{k_2}^2| &= \lim_{n \rightarrow \infty} \frac{1}{n}\left| \variance{\hL_{k_1}} - \variance{\hL_{k_2}} \right| = \lim_{n \rightarrow \infty} \frac{1}{n}\left| \variance{\hL_{k_1}} -\variance{\tL} +\variance{\tL}- \variance{\hL_{k_2}} \right| \\
        &\leq \lim_{n \rightarrow \infty} \frac{1}{n}\left| \variance{\hL_{k_1}} - \variance{\tL} \right| +\lim_{n \rightarrow \infty} \frac{1}{n}\left| \variance{\hL_{k_2}} - \variance{\tL} \right| \\
        &\leq 2\left(\frac{4}{5}\right)^{\frac{k_1}{2}} + 2\left(\frac{4}{5}\right)^{\frac{k_2}{2}} 
        \leq 4\left(\frac{4}{5}\right)^{\frac{k^*}{2}}.
    \end{align*}
    It follows that $(\sigma_k^2)_{k \geq 1}$ is a Cauchy sequence.
    The bounds on $\sigma$ follow from the bounds on $\sigma_k$ in \cref{lem:var_asymptotics}.
\end{proof}

\subsection{A central limit theorem for \texorpdfstring{$\hL_k$}{Lk}}\label{sec:normalapprox}

We first note that \cref{lem:normality} easily implies the following central limit theorem for truncations of weighted neighbourhood sums whose variance is not too small, which in particular, by \cref{lem:var_asymptotics} \cref{item:hL_k_var_order_n}, implies \cref{lem:conv_by_stein}.

\begin{lemma}\label{lem:normality_truncatedneighbourhoods}
Let $c > 1$, let $p = \frac{c}{n}$ and let $G \sim G(n,p)$. Let $\widehat{M}_k$ be the truncation of a weighted neighbourhood sum. If $\variance{\widehat{M}_k} = \omega\left( n^{\frac{2}{3} }\right)$, then 
\[
 \frac{\widehat{M}_k - \expec{\widehat{M}_k}}{\left(\variance{\widehat{M}_k}\right)^{\frac{1}{2}}} \quad\indist\quad \mathcal{N}(0,1) \quad \text{ as } n \to \infty.
 \]
\end{lemma}
\begin{proof}
Suppose $\widehat{M}_k$ is the truncation (as in \eqref{e:truncation}) of a neighbourhood sum
\[
\tM_ k= \sum_{v \in V(G)} \sum_{(H,r) \in \mathcal{H}} \beta_{(H,r)} \mathds{1}_{\{(G[B_G(v,k)],v) \cong (H,r)\}}.
\]
As in \cref{eq:Y_U_Def}, we can write $\widehat{M}_k = \sum_{U \in \cU_k} Y_U$, where  $\cU \coloneqq \{U \subseteq V(G) \colon 1 \leq |U| \leq t_k\}$,
\[
Y_U ~\coloneqq~ \sum_{v \in U} \sum_{\substack{(T,r) \in \mathcal{T} \\ v(T) = |U|}} \beta_{(T,r)} I(U,v,T,r),
\]
and $I(U,v,T,r)\coloneqq \mathds{1}_{\{B_G(v,k) = U\}} \mathds{1}_{\{(G[U],v) \cong (T,r)\}}$. Observe that for each $U \in \cU$, we have that $Y_U$ is semi-induced with respect to $U$.

Thus, it is easy to see that the conditions of \cref{lem:normality} are satisfied for $Y =\widehat{M}_k$ with $t_k$ playing the roles of both $t$ and $K$. Hence, by \cref{lem:normality}, we have 
    \[
        \frac{\widehat{M}_k - \expec{\widehat{M}_k}}{\left(\variance{\widehat{M}_k}\right)^{\frac{1}{2}}} \quad\indist\quad \mathcal{N}(0,1) \quad \text{ as } n \rightarrow \infty,
    \]
    as desired.
\end{proof}

\subsection{Proof of Lemma \ref{lem:variancelimit}}\label{sec:variancelimit}

We first note that \cref{item:limit_in_n} follows immediately from \cref{lem:var_asymptotics}. Furthermore, \cref{item:limit_bounds} and the first part of \cref{item:limit_in_k} follow from \cref{lem:limitsigmak}.

For the second part of \cref{item:limit_in_k} we argue as follows. By \cref{thm:L_tL_equal}, \cref{lem:var_X_X_k_diff_bound}, and \cref{lem:var_L_k_hL_k_diff_bound},
\begin{align*}
&\variance{L-\tL} \leq \expec{(L-\tL)^2} = O(1);\\
&\variance{\tL-\tL_k} \leq \left(64ck^2 \left(\frac{3}{4}\right)^{k-2} + o(1) \right)n; \text{ and }\\
&\variance{\tL_k - \hL_k} \leq \exp\left(-c^2k^2\right)n + O(1).
\end{align*}

Furthermore, since by \cref{thm:prob_Cauchy_Schwarz},
\begin{align*}
\variance{X} \leq \variance{X-Y} + \variance{Y} + 2\left(\variance{X-Y}\variance{Y}\right)^{\frac{1}{2}},
\end{align*}
we can deduce from the above, and the fact that 
$\variance{\hL_k} = O(n)$ (see~\cref{lem:Var_is_Kn}), that also $\variance{L},\variance{\tL}$, and $\variance{\tL_k}$ are $O(n)$.

It then follows by repeated applications of \cref{thm:prob_Cauchy_Schwarz} that
\[
    \lim_{n \rightarrow \infty} \frac{\variance{L(G)}}{n} = \lim_{n \rightarrow \infty} \frac{\variance{\tL}}{n} = \lim_{k \rightarrow \infty} \lim_{n \rightarrow \infty} \frac{\variance{\hL_k}}{n} = \lim_{k \rightarrow \infty} \sigma_k^2 = \sigma^2,
\]
establishing the second part of \cref{item:limit_in_k}.

\section{Discussion}\label{s:discussion}
In this paper, we prove a central limit theorem for the circumference of a sparse random graph. 
Our proof 
has essentially two main parts:  we first establish a `local' approximation to the circumference, for which a central limit theorem can be proven using standard techniques; and we then use the Efron--Stein inequality to control the accuracy of our approximation and in this way deduce a central limit theorem from the local version.

We complement our main result  by proving that the circumference of a random graph exhibits the same behaviour up to the Hamiltonicity threshold. In addition, we show that at the Hamiltonicity threshold the difference between the number of vertices and the circumference of a random graph follows the Poisson paradigm.

Unlike many other graph parameters for which a central limit theorem has been established, for example, subgraph counts \cite{GK16, K84, KR83, R88}, the order of the giant component \cite{BR12,PW05,PW08} or the order of the $k$-core \cite{JL08}, there is not an obvious `local' approximation to the circumference, however the main novelty in this regard comes from earlier work of the first author \cite{A23} and the first author and Frieze \cite{AF21}, see for example \cref{thm:L_tL_equal}, where they established an explicit structural parameter which has a `local' approximation and which is close in distribution to the circumference.

Our main contribution is then the application of the Efron--Stein inequality in this context to `lift' properties of these approximations by analysing local changes under the resampling of edges, together with the non-trivial combinatorial analysis of the local changes of this approximation under resampling.

It seems likely that these methods, both the local approximations coming from \cite{A23,AF21} and the use of the Efron--Stein inequality to transfer a central limit theorem from an approximation, can be applied to prove central limit theorems in other contexts.

As a simple example, since the $k$-core can be defined via a similar peeling process as in \cref{alg:strong_4_core_alg}, there is a natural `local' approximation to the order of the $k$-core in the vein of \cref{alg:strong_4_core_alg_local}, for which a central limit theorem can be established via \cref{lem:normality}. It is reasonable to believe that the changes in the difference between the $k$-core and this approximation under resampling can be shown to be small and `local' as in \cref{s:strong4coreproperties}, and so these methods should lead to an alternative proof of a central limit theorem for the order of the $k$-core for a suitable range of $p$, which we intend to address in future work \cite{AEKP25}.

As a more complex example, it seems likely that similar methods should give an alternative proof that the matching number of $G\left(n,\frac{c}{n}\right)$ satisfies a central limit theorem for $c$ sufficiently large, which was recently proved for $c>e$ by Kreacic \cite{K17} and for $c \leq e$ by Glasgow, Kwan, Sah and Sawhney \cite{GKSS24}.

Indeed, similar arguments as in \cite{A23} should show that for $c$ sufficiently large, if we take the colouring $V(G) = S \cup P \cup R$ given by \Cref{alg:strong_4_core_alg}, then $G[S \cup P]$ is robust enough that any matching on $G[P \cup R]$ can be extended to a matching which covers all but at most one vertex of $S \cup P$. In particular, if a vertex is contained in a component of $G[P\cup R]$ of size at most $k$, then the `contribution' of $v$ towards the matching number can already be `recognised' in $B_G(v,k)$. This would lead to a similar `local' approximation to the matching number, for which most of the analysis in this paper would follow verbatim.
Of course, there are many technical details left to check to verify that the approximations in these cases are appropriately `robust' under the resampling of edges, and that the variances involved are large enough to apply Stein's method (c.f.~\cref{lem:normality}).

Whilst \cref{thm:Law_of_longest_cycle} goes slightly beyond a central limit theorem, giving some quantitative control on the variance $\variance{L}$ of the circumference $L \coloneqq L\left(G\left(n, \frac{c}{n}\right)\right)$, our upper and lower bounds on $\sigma = \lim_{n\to \infty}\left(\frac{\variance{L}}{n}\right)^{\frac{1}{2}}$ are quite far apart as functions of $c$. It would be interesting to know what the actual asymptotic dependence of $\sigma$ on $c$ is, where we would expect the answer to be closer to the upper bound than the lower bound. Furthermore, given \cref{thm:Law_of_longest_cycle}, we find ourselves in the situation that, for $c \geq 20$, we know that $L$ is concentrated in a reasonably small interval about its mean $\mu:=\expec{L}$, in that $L = \mu + O_{\bm{p}}(\sqrt{n})$, but the best known estimate for $\mu$, which follows from \cref{t:AnastosFriezeScaling}, only gives $\mu = f(c) n + o_{\bm{p}}(n)$.  Hence, it would be interesting to obtain a more precise estimate for $\mu$. More ambitiously, it would be interesting to determine if $L$ satisfies a \emph{local limit theorem}, in the sense that, for any compact interval $I \subseteq \mathbb{R}$ and $\delta >0$, there exists $n_0 \in \mathbb{N}$ such that for all $n \geq n_0$, we have that if $k = k(n) \in \mathbb{N}$ is such that $\frac{k - \mu}{\sqrt{n}} \in I$, then
\[
(1-\delta) \frac{1}{\sqrt{2\pi n} \sigma} \exp\left(- \frac{(k-\mu)^2}{2\sigma^2 n}\right) \leq \prob{L = k} \leq (1+\delta)  \frac{1}{\sqrt{2\pi n} \sigma} \exp\left(- \frac{(k-\mu)^2}{2\sigma^2 n}\right).
\]

Finally, as mentioned in the introduction, whilst the distribution of the circumference of a sparse random graph $G\left(n,\frac{c}{n}\right)$ is well understood now when $c$ is large, there is still a large gap in our knowledge for smaller $c$, where even a scaling limit has not yet been established. It would be interesting to know if the lower bound $c \geq 20$ can be improved upon significantly, where we note that for small enough values of $c$, for example below the threshold for the existence of a $3$-core, there does not seem to be a way to apply the methods used in \cite{A23} and in this paper. In this paper, we specifically use the lower bound on $c$ in the proofs of \cref{lem:expec_W^*_bounds,lem:L_k_var_lower_bound}, and it is also needed to apply \cref{thm:L_tL_equal}. 
Roughly, a large $c$ is needed to ensure the existence of the strong $4$-core and to show that the \emph{mantle}, what remains when we delete the strong $4$-core, is roughly distributed like a sub-critical random graph.
Heuristically, there should be some threshold $c^*$ at which the strong $4$-core emerges, and above which the mantle should have such a sub-critical distribution. It would be interesting to determine the value of $c^*$, which computational experiments suggest is $\approx 9.2$, and more generally the threshold for the existence of strong a $k$-core for general $k \in \mathbb{N}$. We note that, in the case $k=4$, the conjectured threshold is much larger than the corresponding threshold for the $4$-core, which is known to be $\approx 5.15$, and instead lies between the threshold for the $6$-core ($\approx 8.37$) and the threshold for the $7$-core ($\approx 9.88$).

\section*{Acknowledgements}
We thank Michael Missethan for discussions at the beginning of the project.  This research was funded in part by the Austrian Science Fund (FWF) [10.55776/P36131, 10.55776/I6502], and by the European Union’s Horizon 2020 research and innovation programme under the Marie Sk\l{}odowska-Curie grant agreement No.\ 101034413 \includegraphics[width=4.5mm, height=3mm]{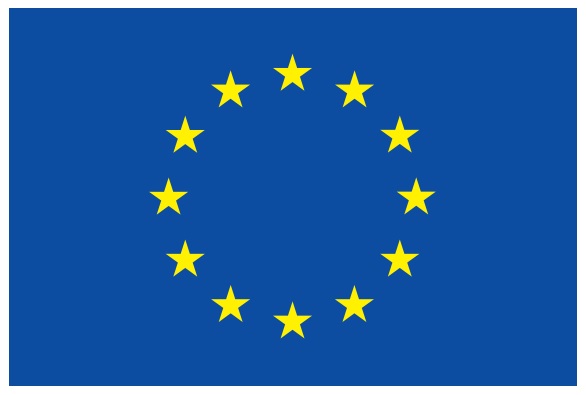}. For open access purposes, the authors have applied a CC BY public copyright license to any author accepted manuscript version arising from this submission.

\bibliographystyle{abbrv}
\bibliography{bib}
 
\appendix

\section{Proof of Proposition \ref{p:generalp}}\label{a:generalp}
This proof is most easily understood when read after reading the rest of the paper, as it uses some results and ideas from the proof of Theorem \ref{thm:Law_of_longest_cycle}, and we give less detail where arguments are similar to earlier parts of the paper.

We first note that for $\lambda(n) \to \infty$, a classic result of Bollob\'{a}s~\cite{B83} and Koml{\'o}s and Szemer{\'e}di \cite{KS83} implies that whp $G(n,p)$ is Hamiltonian, and so $n=L(G)$, dealing with \ref{i:large}. Let us suppose then, in what follows, that $\lambda(n)$ is bounded above by a constant and thus $c = \log n + \log \log n + \lambda(n) \leq 2 \log n$. 

Let $V(G) = S \cup P \cup R$ be the partition given by \Cref{alg:strong_4_core_alg} and let $\tL$ be defined as in \cref{thm:L_tL_equal}, so that $\prob{\tL(G) \neq L(G)} = O\left(n^{-2}\right)$, and so it suffices to show that $\tL(G)$ satisfies \cref{p:generalp} in place of $L(G)$. Moreover, note that by \cref{obs:red_vxs_in_comp}, a component containing $i$ red vertices has at most $3i$ purple vertices and a similar argument as in \cref{lem:no_large_red_purple_comps} shows that, for $c=O(\log n)$, whp all components of $G[P \cup R]$ have order at most $\log^4 n$.

Let us first consider the case where $\lambda(n) \rightarrow \lambda \in \mathbb{R}$. In this case, we note that a simple first moment calculation, as in \eqref{e:W*conditionalexpecbound}, implies that whp there are no components of $G[P \cup R]$ containing more than one red vertex.
The probability that there is some component of $G[P \cup R]$ containing more than one red vertex is at most
\begin{align*}
 \sum_{i\geq 2}^{4 \log^4 n} \sum_{j=0}^{3i} \binom{n}{j+i}\binom{j+i}{i} (j+i)^{j+i-2} p^{j+i-1} (1-p)^{i(n-j-i)} 
=o(1).
\end{align*}

However, if a component $K$ contains a single red vertex $v$, then $\uc(G[K],P \cap K)$ is equal to $0$ if $d(v) \geq 2$, and $1$ if $d(v) \leq 1$. On the other hand, every vertex of degree $\leq 1$ lies in $R$ and hence whp
\[
\tL(G) = n - n_0(G) - n_1(G),
\]
where $n_i(G)$ is the number of vertices of degree $i$ in $G$. It follows that $\tL(G) - (n - n_0(G) - n_1(G)) \inprob 0$, as $n \rightarrow \infty$. Hence, standard results on the distribution of $n_0(G) + n_1(G)$, see for example~\cite[Section 3]{FK16}, 
imply that \ref{i:medium} holds.

The case where $\lambda(n) \rightarrow -\infty$ and $c \rightarrow \infty$ is slightly more delicate, but follows a similar idea. We write
\begin{align*}
\tL(G) &= (\tL(G)  - n_0(G) - n_1(G)) + n_0(G) + n_1(G)\\
&\eqqcolon Y + n_0(G) + n_1(G).
\end{align*}
Since, by standard results, $n_1(G)$ satisfies a central limit theorem in this regime of $p$, it will be sufficient to show that 
\begin{align} \label{eq:var_dominates}
    \max\{ \variance{n_0(G)}, \variance{Y} \} = o \left( \variance{n_1(G)}\right).
\end{align}
Indeed, note that if we assume \cref{eq:var_dominates}, then by Chebyshev's inequality, we have as $n \rightarrow \infty$, 
\begin{align*}
    \frac{\tL(G)-\expec{\tL(G)}}{\sqrt{\variance{n_1(G)}}} - \frac{n_1(G)-\expec{n_1(G)}}{\sqrt{\variance{n_1(G)}}} = \frac{Y - \expec{Y}}{\sqrt{\variance{n_1(G)}}} + \frac{n_0(G) - \expec{n_0(G)}}{\sqrt{\variance{n_1(G)}}} \inprob 0.
\end{align*}
This implies that as $n \rightarrow \infty$,
\begin{align*}
    \frac{\tL(G)-\expec{\tL(G)}}{\sqrt{\variance{n_1(G)}}} \quad\indist\quad \mathcal{N}(0,1).
\end{align*}
Since \cref{eq:var_dominates} and \cref{thm:prob_Cauchy_Schwarz} imply that $\variance{\tL(G)} = (1+o(1))\variance{n_1(G)}$, we have by Slutsky's theorem (\cref{thm:Slutsky}), that as $n \rightarrow \infty$,
\begin{align*}
    \frac{\tL(G)-\expec{\tL(G)}}{\sqrt{\variance{\tL(G)}}} \quad\indist\quad \mathcal{N}(0,1).
\end{align*}
Thus it suffices to show \cref{eq:var_dominates}.

Again by standard results, the variances of $n_0(G)$ and $n_1(G)$ are known to satisfy
\[
\variance{n_1(G)}= (1+o(1))ce^{-c}n \quad \text{ and } \quad \variance{n_0(G)}= o(\variance{n_1(G)}).
\]

Hence, it remains to prove that $\variance{Y}= o(ce^{-c}n) = o(\variance{n_1(G)})$. Now, by considering $Y$ as a sum of weighted components we have
\[Y= \sum_{C} (\Phi(C)-n_0(C) - n_1(C)),
\]
where the sum is taken over components $C$ of $G[P\cup R]$.
However, we note that the contribution to this sum from components $C$ containing at most one red vertex is $0$. Indeed, by \Cref{obs:red_vxs_in_comp}, all components contain at least one red vertex, and for a component $C$ containing a unique red vertex $w$ it easy to verify that $\Phi(C)=1$ if $d(w) \leq 1$ and $\Phi(C) = 0$ otherwise. Hence, since each vertex of degree at most $1$ is always red, in each case it follows that $\Phi(C)-n_0(C) - n_1(C) = 0$.
Thus we can rewrite $Y$ as 
\begin{align*}
    Y&= \sum_{\substack{C\\ |C\cap R| \geq 2 }} \Phi(C)-n_0(C) - n_1(C)
\\&= \sum_{\substack{C\\ |C\cap R| \geq 2 }} \sum_{v\in C}\frac{ \Phi(C)-n_0(C) - n_1(C)}{|C|}.
\end{align*}
For every component $C$ of $G[P\cup R]$ and $w \in C$, let $\psi'_G(w)=\frac{ \Phi(C)-n_0(C) - n_1(C)}{|C|}$ and let $\psi'_G(w)=0$ for $w\in S(G)$. 
If we let $D' \coloneqq \{w \in V(G) \colon \psi'_{G^+}(w) \neq \psi'_{G^-}(w)\}$, then
by \cref{lem:ESapp}
\begin{align*}
    \variance{Y} \leq 2p(1-p)n^2 \expec{|D'|^2} \leq 2c n \cdot \expec{|D'|^2}.
\end{align*}
Note that \cref{obs:no_change_for_indif} also holds with $\phi$ replaced by $\psi'$. Hence similar to \cref{lem:D_cont_in_W}, we get that $D' \subseteq W^* \coloneqq W^- \cup W^+ \cup \{u,v\}$. Suppose that $D'\neq \varnothing$ and let $w \in D'$. 
Then $\psi'_{G^+}(w) \neq \psi'_{G^-}(w)$, hence $\psi'_{G^+}(w) \neq 0$ or $\psi'_{G^-}(w) \neq 0$. 
Assume that $\psi'_{G^+}(w) \neq 0$ (the proof proceeds analogously in the other case). 
It follows that $w \in (R^+ \cup P^+) \cap W^* \subseteq W^- \cup W^+$ and that $C_w^+$ is a red-purple component in $G^+$ with at least $2$ red vertices (as otherwise $\psi'(w) = 0$). 
By \cref{lem:S_and_W_containment}, we have $N_{G^+}(W^- \cup W^+) \subseteq S^+$ and thus $C_w^+ \subseteq W^- \cup W^+ \subseteq W^*$ (since $C_w^+$ is a component of $G^+ - S^+$).
Since $W^*_R \eqqcolon W^* \cap (R^- \cup R^+) \supseteq W^+ \cap R^+ \supseteq C_w^+ \cap R^+$, we have $|W^*_R| \geq |C_w^+ \cap R^+| \geq 2$.
Hence, by \cref{lem:W^*_properties} we have that if $D'\neq \varnothing$ then $G^+$ spans a tree $T$ containing the vertices $u$ and $v$, at least $\ell \geq 2$ vertices with no neighbour outside $T$, and at most $4\ell+2$ vertices in total.

Similarly to \eqref{e:W*conditionalexpecbound}, we have that
\begin{align*}
 \expec{|D'|^2} &\leq  \sum_{\ell =2}^{4\log^4 n}  
  \sum_{j=2}^{4 \ell +2} j^2 \binom{n-2}{j-2} j^{j-2} p^{j-2} \binom{j}{\ell } (1-p)^{\ell (n-j)} \nonumber\\
   &\leq  \sum_{\ell =2}^{4\log^4 n}  
  \sum_{j=2}^{4 \ell +2} j^2 \bfrac{e(n-2)jp}{\max\{j-2,1\}}^{j-2} 2^j e^{-p\ell (n-j)} \nonumber\\
   &\leq  \sum_{\ell =2}^{4\log^4 n}  
  \sum_{j=2}^{4 \ell +2} 4 j^2 (20c)^{j-2}  e^{-(1+o(1))c\ell }  = c^8e^{-(1+o(1))2c},
    \end{align*}
where at the last equality we used that $c\to \infty$ as $n \to \infty$. 
Thus,
\begin{align*}
    \variance{Y}&\leq 2p(1-p)n^2 \expec{|D|^2}\leq 2n c^9e^{-(1+o(1))2c}=o(ce^{-c}n).
\end{align*}

\section{Proof of Lemma \ref{lem:normality}}\label{a:normality}

By \cref{thm:Stein}, it suffices to show that $\delta \rightarrow 0$ as $n \rightarrow \infty$ (where $\delta$ is defined as in~\cref{eq:delta}).
Note that for any $U \in \cU$,
\begin{align} \label{eq:mean}
    \mu_U \coloneqq \expec{Y_U} \leq \expec{T_U} = O(p^{|U|-1}),
\end{align}
where we used the fact that $|U| \leq t = O(1)$.

We claim that for all $U_1, U_2, U_3 \in \cU$ with $U_1 \cap U_2 = \varnothing$ and $U_1 \cap U_3 = \varnothing$, the following inequalities hold.
\begin{align}
    0 &\leq Y_{U_1} \leq Y_{U_1}^{U_2} \leq Y_{U_1}^{U_2 \cup U_3} \leq T_{U_1}, \label{eq:Ineq1} \\
    0 &\leq Y_{U_1}^{U_2} - Y_{U_1} \leq Y_{U_1}^{U_2} \mathds{1}_{\{E_G(U_1,U_2) = \varnothing\}} \leq T_{U_1} \mathds{1}_{\{E_G(U_1,U_2) = \varnothing\}}, \text{ and } \label{eq:Ineq2} \\
    0 &\leq Y_{U_1}^{U_2 \cup U_3} - Y_{U_1}^{U_2} \leq Y_{U_1}^{U_2 \cup U_3} \mathds{1}_{\{E_G(U_1,U_3) = \varnothing\}} \leq T_{U_1} \mathds{1}_{\{E_G(U_1,U_2) = \varnothing\}}. \label{eq:Ineq3}
\end{align}
All these inequalities except the middle inequalities in \cref{eq:Ineq2} and \cref{eq:Ineq3} follow directly from \cref{eq:Y_U_ineq}. Note that if $E_G(U_1, U_2) = \varnothing$, then, by the definition of $Y_{U_1}^{U_2}$ and the fact that $Y_{U_1}$ is semi-induced with respect to $U_1$, we have $Y_{U_1} = Y_{U_1}^{U_2}$. Hence $Y_{U_1}^{U_2} - Y_{U_1} \leq Y_{U_1}^{U_2} \mathds{1}_{\{E_G(U_1,U_2)= \varnothing\}}$. The middle inequality in \cref{eq:Ineq3} follows analogously. 

For $U_1, U_2\in \cU$ with $U_1 \cap U_2 \neq \varnothing$, if both $G[U_1]$ and $G[U_2]$ are connected then so is $G[U_1\cup U_2]$. Thus,  
\begin{align} \label{eq:2T}
        \expec{T_{U_1}T_{U_2}} &\leq K^2 \expec{\ind{G[U_1\cup U_2] \text{ is connected}}}  = O(p^{|U_1 \cup U_2| - 1}).
    \end{align}
Similarly, for $U_1, U_2, U_3 \in \cU$ with $U_1 \cap U_2 \neq \varnothing$ and $(U_1 \cup U_2) \cap U_3 \neq \varnothing$, if each of $G[U_1]$, $G[U_2]$, and $G[U_3]$ is connected then so is $G[U_1\cup U_2\cup U_3]$. Thus
\begin{align} \label{eq:3T}
    \begin{split}
        \expec{T_{U_1}T_{U_2}T_{U_3}} &= O(p^{|U_1 \cup U_2 \cup U_3| - 1}).
    \end{split}  
\end{align}

We deal with the three terms in \cref{eq:delta} separately.
\begin{claim} \label{claim:bound_term_1}
   We have $\sum_{U \in \cU} \expec{|X_U| Z_U^2} = o(1)$.
\end{claim}
\begin{proofclaim}
    Note that 
    \begin{align*}
        \begin{split}
            \sigma_Y^2 Z_{U_1}^2 &= \sum_{\substack{U_2, U_3 \in \cU \\ U_2 \cap U_1 \neq \varnothing,\, U_3 \cap U_1 \neq \varnothing}} Y_{U_2} Y_{U_3} + 
            \sum_{\substack{U_2, U_3 \in \cU \\ U_2 \cap U_1 \neq \varnothing,\, U_3 \cap U_1 = \varnothing}}  2 Y_{U_2} (Y_{U_3} - Y_{U_3}^{U_1}) \\
            &\quad\quad + \sum_{\substack{U_2, U_3 \in \cU \\ U_2 \cap U_1 = \varnothing,\, U_3 \cap U_1 = \varnothing}} (Y_{U_2}^{U_1} - Y_{U_2}) (Y_{U_3}^{U_1} - Y_{U_3})
       \end{split}
       \\
       &\leq \sum_{\substack{U_2, U_3 \in \cU \\ U_2 \cap U_1 \neq \varnothing,\, U_3 \cap U_1 \neq \varnothing}} Y_{U_2} Y_{U_3} + \sum_{\substack{U_2, U_3 \in \cU \\ U_2 \cap U_1 = \varnothing,\, U_3 \cap U_1 = \varnothing}}  (Y_{U_2}^{U_1} - Y_{U_2}) (Y_{U_3}^{U_1} - Y_{U_3}),
    \end{align*}
    since $Y_{U_3} \leq Y_{U_3}^{U_1}$.
    Together with $|X_U| \leq \sigma_Y^{-1}(Y_U + \mu_U)$, this implies that 
    \begin{align} \label{eq:expand_1}
        \begin{split}
            \sigma_Y^3 &\smashoperator[l]{\sum_{U \in \cU}} \expec{|X_U| Z_U^2} \leq \\
            &\smashoperator[l]{\sum_{\substack{U_1, U_2, U_3 \in \cU \\ U_2 \cap U_1 \neq \varnothing,\, U_3 \cap U_1 \neq \varnothing}}} \left(\expec{Y_{U_1} Y_{U_2} Y_{U_3}} + \mu_{U_1} \expec{ Y_{U_2} Y_{U_3}}\right) \\
            +\quad\quad\quad&\smashoperator[l]{\sum_{\substack{U_1, U_2, U_3 \in \cU \\ U_2 \cap U_1 = \varnothing,\, U_3 \cap U_1 = \varnothing}}}  \left(\expec{Y_{U_1}(Y_{U_2}^{U_1} - Y_{U_2}) (Y_{U_3}^{U_1} - Y_{U_3})}  + \mu_{U_1}\expec{(Y_{U_2}^{U_1} - Y_{U_2}) (Y_{U_3}^{U_1} - Y_{U_3})}\right).
        \end{split}
    \end{align}

    Let $U_1, U_2, U_3 \in \cU$ with $U_2 \cap U_1 \neq \varnothing$ and $U_3 \cap U_1 \neq \varnothing$.
    If $U_2 \cap U_3 = \varnothing$, then
    \begin{align*}
        \mu_{U_1} \expec{T_{U_2} T_{U_3}} &= O(p^{|U_1|-1 + |U_2|-1 + |U_3|-1}) \\ &= 
        O(p^{|U_1|-1 + |U_2|-|U_1 \cap U_2| + |U_3|-|(U_1 \cup U_2) \cap U_3|}) = O(p^{|U_1 \cup U_2 \cup U_3| -1}).  
    \end{align*}
    If $U_2 \cap U_3 \neq \varnothing$, then, by \cref{eq:2T}, we have
    \begin{align*}
        \mu_{U_1} \expec{T_{U_2} T_{U_3}} &= O(p^{|U_1|-1 + |U_2\cup U_3| -1}) \\ &= 
        O(p^{|U_1|-|U_1 \cap (U_2 \cup U_3)| + |U_2\cup U_3| -1}) = O(p^{|U_1 \cup U_2 \cup U_3| -1}).  
    \end{align*}
    Hence, using \cref{eq:3T}, we have
    \begin{align*}
        \expec{Y_{U_1} Y_{U_2} Y_{U_3}} + \mu_{U_1} \expec{ Y_{U_2} Y_{U_3}} \leq \expec{T_{U_1} T_{U_2} T_{U_3}} + \mu_{U_1} \expec{T_{U_2} T_{U_3}} = O(p^{|U_1 \cup U_2 \cup U_3|-1}).
    \end{align*}
    It follows that 
    \begin{align} \label{eq:bound_1}
        \begin{split}
            \sum_{\substack{U_1, U_2, U_3 \in \cU \\ U_2 \cap U_1 \neq \varnothing,\, U_3 \cap U_1 \neq \varnothing}} \left(\expec{Y_{U_1} Y_{U_2} Y_{U_3}} + \mu_{U_1} \expec{ Y_{U_2} Y_{U_3}}\right) 
            \\ = \sum_{\substack{U_1, U_2, U_3 \in \cU \\ U_2 \cap U_1 \neq \varnothing,\, U_3 \cap U_1 \neq \varnothing}} O(p^{|U_1 \cup U_2 \cup U_3|-1}) = O(n).
        \end{split}
    \end{align}
    
    Now let $U_1, U_2, U_3 \in \cU$ with $U_2 \cap U_1 = \varnothing$ and $U_3 \cap U_1 = \varnothing$. 
    If $U_2 \cap U_3 = \varnothing$, then, by  \cref{eq:Ineq1}, \cref{eq:Ineq2}, \cref{eq:mean}, and independence, we have 
    \begin{align*}
        &\expec{Y_{U_1}(Y_{U_2}^{U_1} - Y_{U_2}) (Y_{U_3}^{U_1} - Y_{U_3})}  + \mu_{U_1}\expec{(Y_{U_2}^{U_1} - Y_{U_2}) (Y_{U_3}^{U_1} - Y_{U_3})} \\ & \quad\leq
        \expec{T_{U_1}T_{U_2}T_{U_3} \ind{E_G(U_1,U_2) \neq \varnothing} \ind{E_G(U_1,U_3) \neq \varnothing}} \\
        &\quad\quad+ \expec{T_{U_1}} \expec{T_{U_2}T_{U_3} \ind{E_G(U_1,U_2) \neq \varnothing} \ind{E_G(U_1,U_3) \neq \varnothing}} \\
        &\quad = 2 \expec{T_{U_1}}\expec{T_{U_2}}\expec{T_{U_3}}\prob{E_G(U_1,U_2) \neq \varnothing}\prob{E_G(U_1,U_3) \neq \varnothing} \\
        &\quad = O(p^{|U_1|+|U_2|+|U_3|-1}) = O(p^{|U_1 \cup U_2 \cup U_3|-1}),
    \end{align*}
    since $\prob{E_G(U_1,U_2) \neq \varnothing},\prob{E_G(U_1,U_3) \neq \varnothing} = O(p)$.
    If $U_2 \cap U_3 \neq \varnothing$, then, by  \cref{eq:Ineq1}, \cref{eq:Ineq2}, \cref{eq:mean}, \cref{eq:2T}, and independence, we have  
    \begin{align*}
        &\expec{Y_{U_1}(Y_{U_2}^{U_1} - Y_{U_2}) (Y_{U_3}^{U_1} - Y_{U_3})}  + \mu_{U_1}\expec{(Y_{U_2}^{U_1} - Y_{U_2}) (Y_{U_3}^{U_1} - Y_{U_3})} \\ & \quad\leq
        \expec{T_{U_1}T_{U_2}T_{U_3}\ind{E_G(U_1,U_2) \neq \varnothing} } + \mu_{U_1}\expec{T_{U_2}T_{U_3}\ind{E_G(U_1,U_2) \neq \varnothing} } \\
        &\quad\leq 2 \expec{T_{U_1}} \prob{E_G(U_1,U_2) \neq \varnothing} \expec{T_{U_2}T_{U_3}} \\
        &\quad = O(p^{|U_1|+|U_2 \cup U_3|-1}) = O(p^{|U_1 \cup U_2 \cup U_3|-1}).
    \end{align*}
    It follows that 
    \begin{align} \label{eq:bound_2}
        \begin{split}
            &\sum_{\substack{U_1, U_2, U_3 \in \cU \\ U_2 \cap U_1 = \varnothing \\ U_3 \cap U_1 = \varnothing}}  \left(\expec{Y_{U_1}(Y_{U_2}^{U_1} - Y_{U_2}) (Y_{U_3}^{U_1} - Y_{U_3})}  + \mu_{U_1}\expec{(Y_{U_2}^{U_1} - Y_{U_2}) (Y_{U_3}^{U_1} - Y_{U_3})}\right) \\
            &=\sum_{\substack{U_1, U_2, U_3 \in \cU \\ U_2 \cap U_1 = \varnothing \\ U_3 \cap U_1 = \varnothing}} O(p^{|U_1 \cup U_2 \cup U_3|-1})
            = O(n).
        \end{split}
    \end{align}
    Hence, by \cref{eq:expand_1}, we have 
    \[
        \sigma_Y^3 \sum_{U \in \cU} \expec{|X_U| Z_U^2} = O(n).
    \]
    Since, $\sigma_Y = \omega(n^{\frac{1}{3}})$, we have $\sum_{U \in \cU} \expec{|X_U| Z_U^2} = o(1)$.
\end{proofclaim}
Next we bound the second term in \cref{eq:delta}.
    \begin{claim} \label{claim:bound_term_2}
        We have $\sum_{U_1, U_2 \in \cU} \expec{|X_{U_1} Z_{U_1U_2}V_{U_1U_2}|} = o(1)$.
    \end{claim}
    \begin{proofclaim}
        Note that, by definition, if $U_1 = U_2$, then $V_{U_1U_2} = 0$. Hence, we may sum only over $U_1, U_2 \in \cU$ with $U_1 \neq U_2$. By substituting the definitions of $Z_{U_1U_2}$ and $V_{U_1U_2}$, using $|X_{U_1}| \leq \sigma_Y^{-1}(Y_{U_1}+\mu_{U_1})$, and applying the triangle inequality, we have
        \begin{align} \label{eq:expand_2}
            \begin{split}
                \sigma_Y^3\quad &\smashoperator[l]{\sum_{U_1, U_2 \in \cU}} \expec{|X_{U_1} Z_{U_1U_2}V_{U_1U_2}|} \leq \\
                &\smashoperator[l]{\sum_{\substack{U_1, U_2, U_3 \in \cU \\ U_1 \cap U_2 \neq \varnothing, U_1 \neq U_2 \\ U_3 \cap U_1 = \varnothing, U_3 \cap U_2 \neq \varnothing}}} \left(\expec{Y_{U_1}Y_{U_2}Y_{U_3}^{U_1}} + \mu_{U_1} \expec{Y_{U_2}Y_{U_3}^{U_1}}\right)\\
                +\quad\quad &\smashoperator[l]{\sum_{\substack{U_1, U_2, U_3 \in \cU \\ U_1 \cap U_2 \neq \varnothing, U_1 \neq U_2 \\ U_3 \cap U_1 = \varnothing, U_3 \cap U_2 = \varnothing}}} \left(\expec{Y_{U_1}Y_{U_2} |Y_{U_3}^{U_1}-Y_{U_3}^{U_1 \cup U_2}|} + \mu_{U_1} \expec{Y_{U_2} |Y_{U_3}^{U_1}-Y_{U_3}^{U_1 \cup U_2}|}\right) \\
                +\quad\quad &\smashoperator[l]{\sum_{\substack{U_1, U_2, U_3 \in \cU \\ U_1 \cap U_2 = \varnothing \\ U_3 \cap U_1 = \varnothing, U_3 \cap U_2 \neq \varnothing}}} \left(\expec{Y_{U_1}  |Y_{U_2}-Y_{U_2}^{U_1}|   Y_{U_3}^{U_1}} + \mu_{U_1} \expec{|Y_{U_2}-Y_{U_2}^{U_1}|  Y_{U_3}^{U_1}}\right)\\
                +\quad\quad &\smashoperator[l]{\sum_{\substack{U_1, U_2, U_3 \in \cU \\ U_1 \cap U_2 = \varnothing \\ U_3 \cap U_1 = \varnothing, U_3 \cap U_2 = \varnothing}}} \left(\expec{Y_{U_1}|Y_{U_2}-Y_{U_2}^{U_1}|  |Y_{U_3}^{U_1}-Y_{U_3}^{U_1 \cup U_2}|} + \mu_{U_1} \expec{|Y_{U_2}-Y_{U_2}^{U_1}|  |Y_{U_3}^{U_1}-Y_{U_3}^{U_1 \cup U_2}|}\right).
            \end{split}
        \end{align}
        Since, $\sigma_Y = \omega(n^{\frac{1}{3}})$, it suffices to show that each of the four sums in \cref{eq:expand_2} is $O(n)$. 
        
        For the first sum in \cref{eq:expand_2}, let $U_1, U_2, U_3 \in \cU$ such that $U_1 \cap U_2 \neq \varnothing$, $U_1 \neq U_2$, $U_3 \cap U_1 = \varnothing$, and $U_3 \cap U_2 \neq \varnothing$.
        By \cref{eq:Ineq1}, \cref{eq:mean}, \cref{eq:2T}, and \cref{eq:3T}, we have
        \begin{align*}
            \expec{Y_{U_1}Y_{U_2}Y_{U_3}^{U_1}} + \mu_{U_1} \expec{Y_{U_2}Y_{U_3}^{U_1}} &\leq \expec{T_{U_1}T_{U_2}T_{U_3}} + \expec{T_{U_1}} \expec{T_{U_2}T_{U_3}} \\
            &= O(p^{|U_1 \cup U_2 \cup U_3| -1}) + O(p^{|U_1|-1}p^{|U_2 \cup U_3| -1}) \\
            &= O(p^{|U_1 \cup U_2 \cup U_3| -1}) + O(p^{|U_1|-1 + |U_2 \cup U_3| -|U_1 \cap (U_2 \cup U_3)|}) \\
            &= O(p^{|U_1 \cup U_2 \cup U_3| -1}).
        \end{align*}
        Hence, the first sum in \cref{eq:expand_2} is $O(n)$. 
        
        For the second sum, let $U_1, U_2, U_3 \in \cU$ such that $U_1 \cap U_2 \neq \varnothing$, $U_1 \neq U_2$, $U_3 \cap U_1 = \varnothing$, and $U_3 \cap U_2 = \varnothing$. By \cref{eq:Ineq1}, \cref{eq:Ineq3}, \cref{eq:mean}, \cref{eq:2T}, and independence, we have
        \begin{align*}
            &\expec{Y_{U_1}Y_{U_2} |Y_{U_3}^{U_1}-Y_{U_3}^{U_1 \cup U_2}|} + \mu_{U_1} \expec{Y_{U_2} |Y_{U_3}^{U_1}-Y_{U_3}^{U_1 \cup U_2}|} \\
            &\quad \leq \expec{T_{U_1}T_{U_2}T_{U_3} \ind{E_G(U_2,U_3) \neq \varnothing}} + \expec{T_{U_1}} \expec{T_{U_2}T_{U_3} \ind{E_G(U_2,U_3) \neq \varnothing}} \\
            &\quad \leq \expec{T_{U_1}T_{U_2}}\expec{T_{U_3}}\prob{E_G(U_2, U_3) \neq \varnothing} + \expec{T_{U_1}}\expec{T_{U_2}}\expec{T_{U_3}}\prob{E_G(U_2, U_3) \neq \varnothing} \\
            &\quad = O(p^{|U_1 \cup U_2| -1 +|U_3|}) +O(p^{|U_1| + |U_2| + |U_3| -2}) \\
            &\quad = O(p^{|U_1 \cup U_2 \cup U_3|-1}).
        \end{align*}
        Hence, the second sum is $O(n)$.

        For the third sum, let $U_1, U_2, U_3 \in \cU$ such that $U_1 \cap U_2 = \varnothing$, $U_3 \cap U_1 = \varnothing$, and $U_3 \cap U_2 \neq \varnothing$.
        By \cref{eq:Ineq1}, \cref{eq:Ineq2}, \cref{eq:mean}, \cref{eq:2T}, and independence, we have
        \begin{align*}
            &\expec{Y_{U_1}  |Y_{U_2}-Y_{U_2}^{U_1}|   Y_{U_3}^{U_1}} + \mu_{U_1} \expec{|Y_{U_2}-Y_{U_2}^{U_1}|  Y_{U_3}^{U_1}} \\
            &\quad \leq \expec{T_{U_1}T_{U_2}\ind{E_G(U_1, U_2) \neq \varnothing} T_{U_3}} + \expec{T_{U_1}} \expec{T_{U_2}\ind{E_G(U_1, U_2) \neq \varnothing} T_{U_3}} \\
            &\quad = 2 \expec{T_{U_1}} \prob{E_G(U_1, U_2) \neq \varnothing} \expec{T_{U_2}T_{U_3}} \\
            &\quad = O(p^{|U_1|+|U_2 \cup U_3| -1}) = O(p^{|U_1 \cup U_2 \cup U_3|-1}).
        \end{align*}
        Hence, the third sum is $O(n)$.

        For the fourth sum, let $U_1, U_2, U_3 \in \cU$ such that $U_1 \cap U_2 = \varnothing$, $U_3 \cap U_1 = \varnothing$, and $U_3 \cap U_2 = \varnothing$. By \cref{eq:Ineq1}, \cref{eq:Ineq2}, \cref{eq:Ineq3}, \cref{eq:mean}, and independence, we have
        \begin{align*}
            &\expec{Y_{U_1}  |Y_{U_2}-Y_{U_2}^{U_1}|   |Y_{U_3}^{U_1}-Y_{U_3}^{U_1 \cup U_2}|} + \mu_{U_1} \expec{|Y_{U_2}-Y_{U_2}^{U_1}|  |Y_{U_3}^{U_1}-Y_{U_3}^{U_1 \cup U_2}|} \\
            &\quad \leq \expec{T_{U_1}T_{U_2}T_{U_3} \ind{E_G(U_1,U_2) \neq \varnothing} \ind{E_G(U_2,U_3) \neq \varnothing}} \\
            &\quad\quad+ \expec{T_{U_1}}\expec{T_{U_2}T_{U_3} \ind{E_G(U_1,U_2) \neq \varnothing} \ind{E_G(U_2,U_3) \neq \varnothing}} \\
            &\quad \leq 2 \expec{T_{U_1}}\expec{T_{U_2}}\expec{T_{U_3}} \prob{E_G(U_1, U_2) \neq \varnothing} \prob{E_G(U_2, U_3) \neq \varnothing} \\
            &\quad = O(p^{|U_1| + |U_2| + |U_3| -1}) = O(p^{|U_1 \cup U_2 \cup U_3|-1}).
        \end{align*}
        Hence, the fourth sum is $O(n)$. This completes the proof of the claim.
    \end{proofclaim}
    Finally we bound the third term in \cref{eq:delta}.
    \begin{claim} \label{claim:bound_term_3}
        We have $\sum_{U_1, U_2 \in \cU} \expec{|X_{U_1}Z_{U_1U_2}|} \expec{|Z_{U_1} + V_{U_1U_2}|} = o(1)$.
    \end{claim}
    \begin{proofclaim}
        We begin by showing that for any $U_1, U_2 \in \cU$, we have $\expec{|Z_{U_1}|} = O(\sigma_Y^{-1})$ and $\expec{|V_{U_1,U_2}|} = O(\sigma_Y^{-1})$.        
        By  the triangle inequality, \cref{eq:Ineq2}, \cref{eq:mean}, and independence   we have
        \begin{align*}
            \sigma_Y \expec{|Z_{U_1}|} &\leq \sum_{\substack{U_3 \in \cU \\ U_3 \cap U_1 \neq \varnothing}} \expec{Y_{U_3}} + \sum_{\substack{U_3 \in \cU \\ U_3 \cap U_1 = \varnothing}} \expec{|Y_{U_3} - Y_{U_3}^{U_1}|} \\ 
            &\leq \sum_{\substack{U_3 \in \cU \\ U_3 \cap U_1 \neq \varnothing}} O(p^{|U_3|-1}) + \sum_{\substack{U_3 \in \cU \\ U_3 \cap U_1 = \varnothing}} \expec{T_{U_3} \ind{E_G(U_1, U_3 ) \neq \varnothing}} \\
            &= O(1) + \sum_{\substack{U_3 \in \cU \\ U_3 \cap U_1 = \varnothing}} \expec{T_{U_3}} \expec{\ind{E_G(U_1, U_3) \neq \varnothing}} \\
            &= O(1) + \sum_{\substack{U_3 \in \cU \\ U_3 \cap U_1 = \varnothing}} O(p^{|U_3|}) = O(1).
        \end{align*}
        Hence, $\expec{|Z_{U_1}|} = O(\sigma_Y^{-1})$. Similarly, $\expec {|V_{U_1, U_2}|} = O(\sigma_Y^{-1})$.
        
        Since,  $\sigma_Y = \omega(n^{\frac{1}{3}})$, it now suffices to show that 
        \begin{align} \label{eq:sum_bound_3}
            \sigma_Y^2 \sum_{U_1, U_2 \in \cU} \expec{|X_{U_1}Z_{U_1U_2}|} = O(n).
        \end{align}
        Let $U_1, U_2 \in \cU$.
        If $U_1 = U_2$, then we have 
        \begin{align*}
            \sigma_Y^2 \expec{|X_{U_1}Z_{U_1U_2}|} &\leq \expec{Y_{U_1}^2} + \mu_{U_1}^2 \leq \expec{T_{U_1}^2} + \mu_{U_1}^2\\
            &\leq K^2\expec{\ind{G[U_1] \text{ is a tree}}} + \mu_{U_1}^2 = O(p^{|U_1| - 1}).
        \end{align*}      
        Hence, the contribution of terms of this type to the sum in \cref{eq:sum_bound_3} is $O(n)$.
        If $U_1 \cap U_2 \neq \varnothing$ and $U_1 \neq U_2$, then by \cref{eq:Ineq1} and \cref{eq:2T},  we have 
        \begin{align*}
            \sigma_Y^2 \expec{|X_{U_1}Z_{U_1U_2}|} &\leq \expec{Y_{U_1}Y_{U_2}} + \mu_{U_1}\mu_{U_2} \\
            &\leq \expec{T_{U_1}T_{U_2}} + \mu_{U_1}\mu_{U_2} \\
            &\leq O(p^{|U_1 \cup U_2|-1}) + O(p^{|U_1|+|U_2|-2}) \\
            &= O(p^{|U_1 \cup U_2| -1}).
        \end{align*}
        Hence, the contribution of terms of this type to the sum in \cref{eq:sum_bound_3} is $O(n)$.
        If $U_1 \cap U_2 = \varnothing$, then  by \cref{eq:Ineq1}, \cref{eq:Ineq2} \cref{eq:mean} and independence, we have 
        \begin{align*}
            \sigma_Y^2 \expec{|X_{U_1}Z_{U_1U_2}|} &\leq \expec{Y_{U_1}   |Y_{U_2}- Y_{U_2}^{U_1}|} + \mu_{U_1} \expec{|Y_{U_2}- Y_{U_2}^{U_1}|} \\
            &\leq \expec{T_{U_1} T_{U_2} \ind{E_G(U_1,U_2) \neq \varnothing}} + \expec{T_{U_1}} \expec{T_{U_2}\ind{E_G(U_1,U_2) \neq \varnothing}} \\
            &= 2 \expec{T_{U_1}} \expec{T_{U_2}} \prob{E_G(U_1,U_2) \neq \varnothing} \\
            &=  O(p^{|U_1| + |U_2| -1}) = O(p^{|U_1 \cup U_2|-1}).
        \end{align*}
        Hence, the contribution of terms of this type to the sum in \cref{eq:sum_bound_3} is $O(n)$. This completes the proof of the claim.
    \end{proofclaim}
    By \cref{claim:bound_term_1}, \cref{claim:bound_term_2}, and \cref{claim:bound_term_3}, we have $\delta \rightarrow 0$ as $n \rightarrow \infty$ (where $\delta$ is defined as in \cref{eq:delta}). Hence, by \cref{thm:Stein}, it follows that 
    \[
        \frac{Y - \expec{Y}}{\sigma_Y} \quad\indist\quad \mathcal{N}(0,1) \quad \text{ as } n \rightarrow \infty,
    \]  
    as desired.

    \section{Proof of Lemma \ref{lem:no_large_red_purple_comps}} \label{a:comp_lem}
    For $\circ \in \{-,+\}$, let $\cE^\circ$ be the event that $G^\circ - S^\circ$ contains a component of size larger than $\log^4n$.
    Note that $\prob{\cE_\mathrm{comp}^c} \leq \prob{\cE^+} + \prob{\cE^-}$. We show that $\prob{\cE^+} = o(n^{-2})$. The proof that $\prob{\cE^-} = o(n^{-2})$ is analogous. Note that we may assume that $n$ is large enough.

    Let $R_t^+$ and $P_t^+$ be the sets of red and purple vertices after $t$ iterations of the while-loop in \cref{alg:strong_4_core_alg} applied to $G^+$, respectively. Moreover, let $m_t^+$ be the size of the largest component in $G^+[R_t^+ \cup P_t^+]$.
    
    Every component of $G^+[R_{t+1}^+ \cup P_{t+1}^+]$ consists of a union of components of $G^+[R_t^+ \cup P_t^+]$ or vertices in $S_t^+$, that are connected by edges incident to the vertices that were recoloured in the $(t+1)$-th iteration. Since in  every iteration a single vertex is recoloured red and at most $3$ of its neighbours are recoloured purple, it follows that $m_{t+1}^+ \leq 4 \Delta(G^+) \max(m_t^+,1)$. 
    
    Suppose that $G^+ - S^+ = G^+[R^+ \cup P^+]$ contains a component of size larger than $\log^4 n$. Then at least one of the following holds.
    \begin{enumerate}[label = \upshape{(\roman*)}]
        \item $\Delta(G^+) \geq \frac{1}{4}\log^2n$,
        \item There exists $t \geq 1$ such that $\log^2n \leq m_t^+ \leq \log^4n$.
    \end{enumerate}
    It follows that 
    \begin{align}\label{e:E+twoterms}
        \prob{\cE^+} &\leq \prob{\Delta(G^+) \geq \frac{1}{4}\log^2n} + \prob{\exists t \geq 1 \colon \log^2n \leq m_t^+ \leq \log^4n}.
    \end{align}
    It now suffices to show that both of these terms are $o(n^{-2})$.
    We bound the first term in \eqref{e:E+twoterms} as follows.
    \begin{align*}
        \prob{\Delta(G^+) \geq \frac{1}{4}\log^2n} &\leq \prob{\Delta(G) \geq \frac{1}{4}\log^2n - 1} \leq n \prob{\text{Bin}\left(n-1, \frac{c}{n}\right) \geq \frac{1}{4}\log^2n -1} \nonumber\\
        &\leq \exp\left( -\frac{1}{4}\log^2n +1\right) = o(n^{-2}), 
    \end{align*}
    where the last inequality follows from \cref{thm:ex_tail_Chernoff}.
    
    To bound the second term in \eqref{e:E+twoterms} we use the following facts. For any component $C$ of $G^+[R_t^+ \cup P_t^+]$, the graph $G^+[C]$ is connected and by \cref{obs:red_vxs_in_comp}, $|R_t^+ \cap C| \geq \frac{|C|}{4}$.
    Moreover, since in $G^+$ there are no edges from $R_t^+$ to $V(G) \setminus (R_t^+ \cup P_t^+)$, there are no edges in $G^+$ from $C \cap R_t^+$ to $V(G) \setminus C$.

    If particular, given a subset $ C \subseteq V(G)$ of size $i\coloneqq |C| = o(n)$, if $C$ is a component of $G^+$, then~$G$ must contain all but one edge of a spanning tree of $C$ and no edges from $C \cap R_t^+$ to $V(G) \setminus C$ and we can calculate
    \begin{align*}
    \prob{C \text{ is a component of $G^+[R_t^+\cup P_t^+]$ for some $t \geq 1$}} &\leq  i^{i-2} p^{i-2} \binom{i}{\lceil \frac{i}{4} \rceil} (1-p)^{\frac{i(n-i)}{4}}\\
    &\leq  4 \left( 2ip \exp\left( - \frac{c}{4}\right)\right)^{i-2},
    \end{align*}
    using the fact that $ \binom{i}{\lceil \frac{i}{4} \rceil} \leq 2^i$, $1-p \leq e^{-p}$ and $i=o(n)$. 
    
    It follows that 
    \begin{align*}
        &\prob{\exists t \geq 1 \colon \log^2n \leq m_t^+ \leq \log^4n} \\
        &\quad\leq \sum_{\log^2n \leq i \leq \log^4n} \sum_{\substack{C \subseteq V(G) \\ |C| =i}} \prob{C \text{ is a component of $G^+[R_t^+\cup P_t^+]$ for some $t \geq 1$}} \\
        &\quad \leq \sum_{\log^2n \leq i \leq \log^4n} 4\binom{n}{i}\left( 2ip \exp\left( - \frac{c}{4}\right)\right)^{i-2}  \\
        &\quad \leq \sum_{\log^2n \leq i \leq \log^4n} \left(\frac{2en}{i}\right)^2 \left( 2ec \exp\left( - \frac{c}{4}\right)\right)^{i-2} \\
        &\quad = o(n^{-2}),
    \end{align*}
    bounding the second term in \cref{e:E+twoterms} and completing the proof.

\section{Proof of Lemma \ref{lem:NUNUCexpectation}} \label{a:neighbourhood_approx}
    In order to prove \cref{lem:NUNUCexpectation}, we need the following bound on the probability that the $k$-th neighbourhood-ball is large. 

    \begin{proposition} \label{prop:kth_neigh_bound}
    Let $c > 0$ and let $G \sim G\left(n,\frac{c}{n}\right)$. Then, for each vertex $w \in V(G)$ and $s > (7c)^k$,
        \begin{align*}
            \prob{|B_{G}(w,k)| \geq s} \leq k\exp\left(-s^{\frac{1}{k}}\right).
        \end{align*}
    \end{proposition}
    \begin{proof}
        We can think of exposing $B_{G}(w,k)$ by running a breadth-first-search algorithm in $G$ rooted at $w$ to depth $k$. A standard argument shows that we can couple the tree formed in this way from below by a Galton--Watson tree with child distribution $\mathrm{Bin}\left(n, \frac{c}{n}\right)$. Let $Z_0,Z_1,\dots, Z_k$ denote the sizes of the $k+1$ layers of the Galton--Watson tree (in particular $Z_0 =1$). We see that by \Cref{thm:ex_tail_Chernoff},
         \begin{align*}
               \prob{|B_{G}(w,k)| \geq s} &\leq \prob{Z_k \geq s} \leq \sum_{i \in [k]}\condprob{Z_i \geq s^{\frac{1}{k}}Z_{i-1}}{Z_{i-1} \geq 1} \leq k \exp\left(-s^{\frac{1}{k}}\right), 
            \end{align*}
            since conditioned on the value of $Z_{i-1}$, we have that $Z_i$ has distribution $\mathrm{Bin}\left(nZ_{i-1}, \frac{c}{n}\right)$.
    \end{proof}

    In the resulting calculations we will also need the following proposition.
    \begin{proposition} \label{prop:near_geometric_sum_2}
    Let $a \geq 1$ and $m_1, m_2, s_0\in \mathbb{N}$ with $s_0 \geq \max\{a(2(m_1+1)(m_2+1))^{2m_2}, a^{m_2+1}\}$. Then \begin{align*}
        \sum_{s \geq s_0 +1} s^{m_1} \exp\left(-\left(\frac{s}{a}\right)^{\frac{1}{m_2}}\right) \leq \exp\left(-\frac{1}{2}\left(\frac{s_0}{a}\right)^{\frac{1}{m_2}}\right).
    \end{align*} 
    \end{proposition}
    \begin{proof}
    Fix $a \geq 1$, $m_1, m_2 \in \mathbb{N}$ and $s_0 \in \mathbb{N}$ with $s_0 \geq \max\{a(2(m_1+1)(m_2+1))^{2m_2}, a^{m_2+1}\}$.
    Consider the function $g(x) \coloneqq x^{m_1}\exp\left(-\left(\frac{x}{a}\right)^{\frac{1}{m_2}}\right)$. Since 
    \[
        g'(x) = \frac{1}{m_2}x^{m_1-1}\exp\left(-\left(\frac{x}{a}\right)^{\frac{1}{m_2}}\right) (m_1m_2 - \left(\frac{x}{a}\right)^{\frac{1}{m_2}}),
    \]
    $g(x)$ is decreasing for $x \geq a(m_1m_2)^{m_2}$. Since $a(m_1m_2)^{m_2} \leq a(2(m_1+1)(m_2+1))^{2m_2} \leq s_0$, we have
    \begin{align*}
        \sum_{s \geq s_0 +1} s^{m_1} \exp\left(-\left(\frac{s}{a}\right)^{\frac{1}{m_2}}\right) \leq \int_{s_0}^\infty x^{m_1} \exp\left(-\left(\frac{x}{a}\right)^{\frac{1}{m_2}}\right) dx. 
    \end{align*}
    Substituting $u = \left(\frac{x}{a}\right)^{\frac{1}{m_2}}$ gives
    \begin{align*}
        \int_{s_0}^\infty x^{m_1} \exp\left(-\left(\frac{x}{a}\right)^{\frac{1}{m_2}}\right) dx = a^{m_1+1}m_2\int_{\left(\frac{s_0}{a}\right)^{\frac{1}{m_2}}}^\infty u^{(m_1+1)m_2-1}e^{-u} du.
    \end{align*}
    Note that $u \geq \left(\frac{s_0}{a}\right)^{\frac{1}{m_2}} \geq \max\{(2(m_1+1)(m_2+1))^2, a\}$. In particular, $\log a \leq \log u$, $\log (2m_2)\leq \log u$,   
    and $(m_1+1)(m_2+1) \sqrt{u} \leq \frac{u}{2}$. 
    Hence
    \begin{align*}
        (m_1+1)\log a + \log(2m_2) +((m_1+1)m_2-1)\log u &\leq (m_1+1)(m_2+1)\log u \\ 
        &\leq (m_1+1)(m_2+1) \sqrt{u} \leq \frac{u}{2}.
    \end{align*}
    Thus, $a^{m_1+1}m_2 u^{(m_1+1)m_2-1} \leq \frac{1}{2}e^{\frac{u}{2}}$. It follows that 
    \begin{align*}
        a^{m_1+1}m_2\int_{(s_0/a)^{\frac{1}{m_2}}}^\infty u^{(m_1+1)m_2-1}e^{-u} du \leq a^{m_1+1}m_2\int_{(s_0/a)^{\frac{1}{m_2}}}^\infty \frac{1}{2}e^{-\frac{u}{2}} du = \exp\left(-\frac{1}{2}\left(\frac{s_0}{a}\right)^{\frac{1}{m_2}}\right).
    \end{align*}
    This completes the proof. 
\end{proof}

    We are now ready to prove \cref{lem:NUNUCexpectation}.
    \begin{proof}[Proof of \cref{lem:NUNUCexpectation}]
    Note that, by symmetry,
    \[
        \prob{|N^*| = s} \leq \prob{|N^*| \geq s} \leq 2 \prob{|B_{G}(u,k)| \geq \frac{s}{2}} \leq 2 k \exp\left(-\left(\frac{s}{2}\right)^{\frac{1}{k}}\right),
    \]
    where the last inequality follows from \cref{prop:kth_neigh_bound}.
    It follows that
    \begin{align} \label{eq:expec_N^2_bound}
        \expec{|N^*|^2 \ind{|N^*| > t_k}} &=  \sum_{s = t_k+1}^n s^2 \prob{|N^*| = s} 
        \leq \sum_{s = t_k+1}^n 2s^2 k \exp\left(-\left(\frac{s}{2}\right)^{\frac{1}{k}}\right) \nonumber\\
        &\leq 2k\exp\left(-2c^2k^2\right),
    \end{align}
    where the last inequality follows by \cref{prop:near_geometric_sum_2} with $a=2, m_1=2$ and $m_2=k$, since $t_k \coloneqq \tbound \geq 2(6(k+1))^{2k}$. This establishes \cref{i:N*large}.
    
    Fix a vertex $w \in V(G) \setminus \{u,v\}$ and note that
    \begin{align} \label{eq:expec_split}
            \expec{|N^* \cap (U_N \cup U_C)|} &\leq 2\prob{u \in U_N} + 2\prob{u \in U_C \setminus U_N} + (n-2) \prob{w \in N^* \cap U_N} \nonumber\\
            &\quad\quad + (n-2)\prob{w \in N^* \cap (U_C \setminus U_N)}.
      \end{align}
    We bound each of these four terms separately.

    To begin, we note that, for any vertex $x \in V(G)$, we have
    \[
    B_{G^+}(x,k) \subseteq B_{G}(x,k) \cup B_{G}(u,k) \cup B_{G}(v,k)
    \]
    and so by \cref{prop:kth_neigh_bound} for any $s \geq 3(7c)^k$,
    \begin{align} \label{eq:N_bound_x}
        \prob{|B_{G^+}(x,k)| \geq s} \leq 3\prob{|B_{G}(x,k)| \geq \frac{s}{3}} \leq 3k\exp\left(-\left(\frac{s}{3}\right)^{\frac{1}{k}}\right).
    \end{align}
    
    In particular,
    \begin{equation}\label{eq:term1UN}
    2\prob{u \in U_N} = 2\prob{|B_{G^+}(u,k)| > t_k} \leq 6k\exp\left(-\left(\frac{t_k}{3}\right)^{\frac{1}{k}}\right) \leq 6k\exp\left(-2c^2k^2\right),
    \end{equation}
    where the last inequality follows from the definition of $t_k$ in \cref{eq:t_k_def}. This bounds the first term of \eqref{eq:expec_split}.

    Furthermore,
    \begin{align}
        \prob{u \in U_C \setminus U_N} &= \prob{|B_{G^+}(u,k)| \leq t_k \text{ and } G^+[B_{G^+}(u,k)] \text{ contains a cycle}} \nonumber \\
        &\leq \prob{|B_{G^+}(u,k)| \leq t_k \text{ and } |E(G^+[B_{G^+}(u)],k)| \geq |B_{G^+}(u,k)|} \nonumber\\
        &\leq \prob{|B_{G^+}(u,k)| \leq t_k \text{ and }|E(G[B_{G^+}(u)],k)| \geq |B_{G^+}(u,k)|-1} \nonumber \\
        &\leq \sum_{s=3}^{t_k} \binom{n-2}{s-2} 2^{s^2} p^{s-1} =O\left(\frac{1}{n}\right) \label{e:eqterm2UN},
    \end{align}
 where we have used that $v \in B_{G^+}(u,k)$ since $uv \in E(G^+)$, bounding the second term of \eqref{eq:expec_split}.

  To bound the third term in \cref{eq:expec_split}, we note, recalling that $N^* \coloneqq B_{G}(u,k) \cup B_{G}(v,k)$, that
    \begin{align} \label{eq:prob_split}
            \prob{w \in N^* \cap U_N} &=  \prob{|B_{G^+}(w,k)| > t_k \text{ and } w \in N^*} \nonumber\\
            &\leq \prob{|B_{G}(w,k)| > \frac{t_k}{2} \text{ and } w \in N^*} + \prob{|N^*| > \frac{t_k}{2} \text{ and } w \in N^*}.
    \end{align}
For the first term in \eqref{eq:prob_split}, we note that
    \begin{align*} 
            &\prob{|B_{G}(w,k)| > \frac{t_k}{2} \text{ and } w \in N^*} = \prob{|B_{G}(w,k)| > \frac{t_k}{2} \text{ and } B_{G}(w,k) \cap \{u,v\} \neq \varnothing} \\
            &\quad\leq \sum_{s = \frac{t_k}{2}+1}^n \prob{|B_{G}(w,k)| = s} \condprob{B_{G}(w,k) \cap \{u,v\} \neq \varnothing}{|B_{G}(w,k)| = s}.
    \end{align*}
   However, conditional on $|B_G(w,k)| = s$, we have that $B_G(w,k) \setminus \{w\}$ is uniformly distributed in $\binom{V(G) \setminus \{w\}}{s-1}$. Hence, by \eqref{eq:N_bound_x}, we obtain
    \begin{align} \label{eq:prob_bound_1}
            \prob{|B_G(w,k)| > \frac{t_k}{2} \text{ and } w \in N^*} &\leq \sum_{s = \frac{t_k}{2} +1}^n \prob{|B_G(w,k)| = s} \frac{2 \binom{n-2}{s-2}}{\binom{n-1}{s-1}} \nonumber\\
            &\leq \sum_{s = \frac{t_k}{2}+1}^n \prob{|B_G(w,k)| \geq s}\frac{2(s-1)}{n-1} \nonumber\\
            &\leq \frac{2}{n-2} \sum_{s = \frac{t_k}{2}+1}^\infty ks\exp\left(-s^{\frac{1}{k}}\right)\nonumber\\
            &\leq  \frac{2k}{n-2} \exp\left(-2c^2k^2\right), 
    \end{align}
   where the last inequality follows by \cref{prop:near_geometric_sum_2} with $a=m_1=1$ and $m_2=k$, since $t_k \coloneqq \tbound \geq \max\{(4(k+1))^{2k},1\}$. 

    A similar calculation, noting that conditional on $|N^*| = s$, we have that $N^* \setminus \{u,v\}$ is uniformly distributed in $\binom{V(G) \setminus \{u,v\}}{s-2}$, leads to
    \begin{align} \label{eq:prob_bound_2}
            \prob{|N^*| > \frac{t_k}{2} \text{ and } w \in N^*}
            &\leq \sum_{s = \frac{t_k}{2}+1}^n \prob{|N^*| = s} \condprob{w \in N^*}{|N^*| = s}. \nonumber\\
            &\leq \sum_{s = \frac{t_k}{2}+1}^n \prob{|N^*| \geq s} \frac{\binom{n-3}{s-3}}{\binom{n-2}{s-2}} \nonumber\\
            &\leq \sum_{s = \frac{t_k}{2}+1}^n 2\prob{|B_G(u,k)| \geq \frac{s}{2}}\frac{s-2}{n-2} \nonumber\\
            &\quad \leq \frac{2}{n-2} \sum_{s = \frac{t_k}{2}+1}^n ks \exp\left(-\left(\frac{s}{2}\right)^{\frac{1}{k}}\right) \nonumber \\
            &\leq \frac{2k}{n-2} \exp\left(-2c^2k^2\right),
    \end{align}
    where the last inequality follows from \cref{prop:near_geometric_sum_2} with $a=2, m_1=1$ and $m_2=k$,  since $t_k \coloneqq \tbound \geq \max\{2(4(k+1))^{2k},2^{k+1}\}$.

    Hence, we can bound the third term in \cref{eq:expec_split}, using \cref{eq:prob_split,eq:prob_bound_1,eq:prob_bound_2}, as
    \begin{align} \label{eq:second_term_bound}
        &(n-2)\prob{w \in N^* \cap U_N} \leq 4k \exp\left(-2c^2k^2\right).
    \end{align}
    
    Finally, we bound the last term in \cref{eq:expec_split}. 
    Suppose that $w \in N^* \cap (U_C \setminus U_N)$. Since $w \in N^*$, $|B_{G^+}(w,k) \cap \{u,v\}| \neq \varnothing$ and since $w \in U_C$, we have that $G^+[B_{G^+}(w,k)]$ contains a cycle. 
    
    If $|B_{G^+}(w,k) \cap \{u,v\}| =1$, then $G^+[B_{G^+}(w,k)] = G[B_{G}(w,k)]$ is connected and has a cycle, thus $|E(G[B_{G^+}(w)],k)| \geq |B_{G^+}(w,k)|$. Let $\mathcal{E}_1$ be the event that $|E(G[B_{G^+}(w)],k)| \geq |B_{G^+}(w,k)|$.
    If instead $|B_{G^+}(w,k) \cap \{u,v\}| =2$, then $G^+[B_{G^+}(w,k)]$ is connected and has a cycle. Thus, $|E(G^+[B_{G^+}(w)],k)| \geq |B_{G^+}(w,k)]|$ which implies that $|E(G[B_{G^+}(w)],k)| \geq |B_{G^+}(w,k)|-1$. Let $\mathcal{E}_2$ be the event that $|E(G[B_{G^+}(w)],k)| \geq |B_{G^+}(w,k)|-1$.
    
    Hence,
    \begin{align} \label{eq:bound_third_term}
        \prob{w \in N^* \cap (U_C \setminus U_N)}
        &\leq \prob{|B_{G^+}(w,k) \cap \{u,v\}| =1 \text{ and } |B_{G^+}(w,k)| \leq t_k \text{ and } \mathcal{E}_1} \nonumber\\ 
        &\quad  + \prob{|B_{G^+}(w,k) \cap \{u,v\}| =2 \text{ and } |B_{G^+}(w,k)| \leq t_k \text{ and } \mathcal{E}_2} \nonumber\\
        &\leq \sum_{s = 3}^{t_k} 2\binom{n}{s-2} 2^{s^2} p^{s} +  \sum_{s = 3}^{t_k} \binom{n}{s-3} 2^{s^2} p^{s-1} \nonumber\\
        &= O\left(\frac{1}{n^2}\right).
    \end{align}

    Finally, by \cref{eq:expec_split,eq:term1UN,e:eqterm2UN,eq:second_term_bound,eq:bound_third_term},
    \begin{align*}
        \expec{|N^* \cap (U_N \cup U_C)|} &\leq 6k\exp\left(-2c^2k^2\right) + O\left(\frac{1}{n}\right) + 4k \exp\left(-2c^2k^2\right) + (n-2) \cdot  O\left(\frac{1}{n^2}\right)\nonumber\\
        &\quad \leq 10k \exp\left(-2c^2k^2\right) +O\left(\frac{1}{n}\right),
    \end{align*}
    establishing \cref{i:N*capUNUC}.
    \end{proof}

\section{Proof of Lemma \ref{lem:balls_in_bins}} \label{a:balls_in_bins}

    All asymptotics in this proof are as $N \rightarrow \infty$.
    Let $Y$ be the number of bins with at most $1$ ball and note that $Y = N-Z$ and so $\variance{Y} = \variance{Z}$.
    For $i \in [N]$, let \[I_i \coloneqq \ind{\text{Bin $i$ has at most $1$ ball}}, \quad\text{ so that }\quad Y = \sum_{i \in [N]}I_i.\]
    Hence
    \begin{align} \label{eq:ball_bin_var_split}
        \variance{Z} = \variance{Y} = \sum_{i \in [N]} \left( \expec{I_i} - \expec{I_i}^2 \right) + \sum_{i \neq j} \left( \expec{I_iI_j} - \expec{I_i}\expec{I_j} \right).
    \end{align}
    Note that, for each $i \in [N]$, we have
    \begin{align} \label{eq:expec_i}
        \expec{I_i} &= \left(1-\frac{1}{N}\right)^m + \frac{m}{N}\left(1-\frac{1}{N}\right)^{m-1} = \left(1-\frac{1}{N}\right)^m\left(1+\frac{m}{N}\left(1-\frac{1}{N}\right)^{-1}\right) \nonumber\\
        &= \left(1-\frac{1}{N}\right)^m\left(1+\frac{m}{N}\left(1+\frac{1}{N}\right)\right) +O(N^{-2}).
    \end{align}
    Let $i, j \in [N]$ with $i \neq j$. We have
    \begin{align} \label{eq:expec_squared_simple}
        \expec{I_i}\expec{I_j} = \expec{I_i}^2 = \left(1-\frac{1}{N}\right)^{2m}\left(1+\frac{2m}{N}\left(1+\frac{1}{N}\right)+\frac{m^2}{N^2}\left(1+\frac{2}{N}\right)\right) +O(N^{-2}).
    \end{align}
    Using that 
    \[
        \frac{\left(1-\frac{1}{N}\right)^{2m}}{\left(1-\frac{2}{N}\right)^m} = \left(1+\frac{\frac{1}{N^2}}{1-\frac{2}{N}}\right)^m = \left(1+\frac{1}{N(N-2)}\right)^m = 1+\frac{m}{N^2} +O(N^{-2}),
    \]
    we obtain
    \begin{align} \label{eq:expec_squared}
        \expec{I_i}\expec{I_j} &= \expec{I_i}^2 \nonumber\\
        &= \left(1-\frac{2}{N}\right)^m\left(1+\frac{m}{N^2}\right)\left(1+\frac{2m}{N}\left(1+\frac{1}{N}\right)+\frac{m^2}{N^2}\left(1+\frac{2}{N}\right)\right) +O(N^{-2}) \nonumber\\
        &= \left(1-\frac{2}{N}\right)^m\left(1+\frac{2m}{N} +\frac{3m}{N^2}+\frac{m^2}{N^2}+\frac{4m^2}{N^3}+\frac{m^3}{N^4}\right) +O(N^{-2}).
    \end{align}
    Moreover,
    \begin{align} \label{eq:expec_i_j}
        \expec{I_i I_j} &= \left(1-\frac{2}{N}\right)^m +  \frac{2m}{N}\left(1-\frac{2}{N}\right)^{m-1} +  \frac{m(m-1)}{N^2}\left(1-\frac{2}{N}\right)^{m-2} \nonumber\\ 
        &= \left(1-\frac{2}{N}\right)^m \left( 1+ \frac{2m}{N}\left(1-\frac{2}{N}\right)^{-1} +  \frac{m^2}{N^2} \left(1-\frac{1}{m}\right)\left(1-\frac{2}{N}\right)^{-2}\right) \nonumber\\
        &= \left(1-\frac{2}{N}\right)^m \left( 1+ \frac{2m}{N}\left(1+\frac{2}{N}\right) +  \frac{m^2}{N^2} \left(1-\frac{1}{m}\right)\left(1+\frac{4}{N}\right)\right) +O(N^{-2}) \nonumber\\
        &= \left(1-\frac{2}{N}\right)^m \left( 1+ \frac{2m}{N} + \frac{3m}{N^2} + \frac{m^2}{N^2} + \frac{4m^2}{N^3}\right) +O(N^{-2}).
    \end{align}
    From \cref{eq:expec_i_j,eq:expec_squared}, we obtain
    \begin{align*}
        \expec{I_i I_j} - \expec{I_i}\expec{I_j} = -\frac{m^3}{N^4}\left(1-\frac{2}{N}\right)^m +O(N^{-2}).
    \end{align*}
    Together with \cref{eq:ball_bin_var_split,eq:expec_i,eq:expec_squared_simple}, we have
    \begin{align*}
        \variance{Z} &= N\left(1-\frac{1}{N}\right)^m\left(1+\frac{m}{N}\right) - N\left(1-\frac{1}{N}\right)^{2m}\left(1+\frac{2m}{N}+\frac{m^2}{N^2}\right) \nonumber\\ 
        &\qquad- (N-1)\frac{m^3}{N^3}\left(1-\frac{2}{N}\right)^m  +O(1) \nonumber\\
        &= h\left(\frac{m}{N}\right)N +O(1).
    \end{align*}
    Finally $\variance{Z} \geq \frac{1}{3}\left(\frac{m}{N}\right)^2 N$, follows since $h(x) > \frac{1}{3}x^2$ for $0 < x < \frac{1}{100}$ (proof omitted).

\section{Proof of Lemma \ref{lem:Var_is_Kn_general}}\label{a:VarhLk}
    As in \cref{eq:Y_U_Def}, we can write $\widehat{M}_k = \sum_{U \in \cU_k} Y_U$, where  $\cU \coloneqq \{U \subseteq V(G) \colon 1 \leq |U| \leq t_k\}$,
    \[
        Y_U \coloneqq \sum_{v \in U} \sum_{\substack{(T,r) \in \mathcal{T} \\ v(T) = |U|}} \beta_{(T,r)} I(U,v,T,r),
    \]
    and $I(U,v,T,r)\coloneqq \mathds{1}_{\{B_G(v,k) = U\}} \mathds{1}_{\{(G[U],v) \cong (T,r)\}}$.
    We have
    \begin{align} \label{eq:var_sum_split}
            \variance{\widehat{M}_k} = &\sum_{U \in \cU_k} \variance{Y_U} + \sum_{\substack{U_1, U_2 \in \cU_k \\ U_1 \neq U_2}} \mathrm{cov}(Y_{U_1}, Y_{U_2}) \nonumber\\
            = &\sum_{U \in \cU_k} \variance{Y_U} + \sum_{\substack{U_1, U_2 \in \cU_k \\ U_1 \cap U_2 = \varnothing}} \mathrm{cov}(Y_{U_1}, Y_{U_2}) + \sum_{\substack{U_1, U_2 \in \cU_k \\ |U_1 \cap U_2| = 1}} \mathrm{cov}(Y_{U_1}, Y_{U_2})  \nonumber\\ 
            &\quad+\sum_{i =2}^{t_k} \sum_{\substack{j_1, j_2 \in \{i, \dots, t_k\} \\ (j_1, j_2) \neq (i,i)}} \sum_{\substack{U_1, U_2 \in \cU_k \\ |U_1| = j_1, |U_2| = j_2 \\ |U_1 \cap U_2| = i}} \mathrm{cov}(Y_{U_1}, Y_{U_2}).
    \end{align}

    We will show that each of these four sums is linear in $n$. 

    Let us first fix some $1 \leq j \leq t_k$ and $U \in \cU_k$ with $|U|=j$ and estimate $\variance{Y_U}$. Note that, since $p = \frac{c}{n}$, we have for any constants $a_1, a_2 >0$, that 
    \[
        (1-p)^{a_1 n + a_2} = e^{-a_1 c} + O(p).
    \]
    
    For $(T,r) \in \cT$, we denote by $\mathrm{aut}(T,r)$ the number of automorphisms of $T$ that map $r$ to itself. Let $\cT_j \coloneqq \{(T,r) \in \cT \colon v(T)=j\}$. We can calculate
    \begin{align}\label{eq:Y_U_expec}
        \expec{Y_U} &= 
        \sum_{v \in U} \sum_{\substack{(T,r) \in \mathcal{T} \\ v(T) = |U|}} \beta_{(T,r)} \prob{B_G(v,k) = U \text{ and } (G[U],v) \cong (T,r)} \nonumber\\
        &=\sum_{(T,r) \in \cT_j} \frac{\beta_{(T,r)}j!}{\mathrm{aut}(T,r)} p^{j-1} (1-p)^{(n-j)|B_T(r,k-1)| +\binom{j}{2}-j+1} \nonumber\\
        &= O(p^{j-1}),
    \end{align}
    where we used that, for $(T,r) \in \cT_j$,
    \[
         \left|\left\{E\subseteq \binom{[j]}{2} \colon (([j], E),1) \cong (T,r)\right\}\right|=\frac{(j-1)!}{\mathrm{aut}(T,r)} .
    \]

     On the other hand we can write
    \begin{align*}
        Y_U^2 &= \sum_{\substack{v_1 \in U \\ v_2 \in U}} \sum_{\substack{(T_1, r_1) \in \cT_j \\ (T_2, r_2) \in \cT_j}} \beta_{(T_1, r_1)}\beta_{(T_2, r_2)} I(U,v_1,T_1,r_2)I(U,v_2,T_2,r_2)\\
        &= \sum_{v \in U} \sum_{(T,r) \in \cT_j} \beta_{(T,r)}^2 I(U,v,T,r) + \sum_{\substack{v_1, v_2 \in U \\ v_2 \neq v_1}} \sum_{\substack{(T,r_1) \in \cT_j \\ r_2 \in V(T) \\ r_2 \neq r_1}} \beta_{(T, r_1)}\beta_{(T, r_2)}  I(U,v_1,T_1,r_2)I(U,v_2,T_2,r_2).
    \end{align*}

Hence, it follows that
    \begin{align}\label{eq:Y_U^2_expec}
        \expec{Y_U^2} &= \sum_{(T,r) \in \cT_j} \frac{\beta_{(T,r)}^2 j!}{\mathrm{aut}(T,r)}p^{j-1} (1-p)^{(n-j)|B_T(r,k-1)| +\binom{j}{2}-j+1} \nonumber\\
        &\quad + \sum_{(T,r_1) \in \cT_j} \sum_{\substack{r_2 \in V(T) \\ r_2 \neq r_1}} \frac{\beta_{(T,r_1)}\beta_{(T,r_2)} j!}{\mathrm{aut}(T,r_1,r_2)} p^{j-1} (1-p)^{(n-j)|B_T(r_1,k-1) \cup B_T(r_2,k-1)| + \binom{j}{2}-j+1} \nonumber\\
        &= p^{j-1} \sum_{(T,r) \in \cT_j} \left(\frac{\beta_{(T,r)}^2 j!}{\mathrm{aut}(T,r)} e^{-c|B_T(r,k-1)|} + \sum_{\substack{r' \in V(T) \\ r' \neq r}} \frac{\beta_{(T,r)}\beta_{(T,r')} j!}{\mathrm{aut}(T,r,r')} e^{-c|B_T(r,k-1) \cup B_T(r',k-1)|}\right) \nonumber\\
        &\quad+O(p^{j}) \nonumber\\
        &= p^{j-1} K^{(1)}_j(k,c) +O(p^j),
    \end{align}
    where $\mathrm{aut}(T,r,r')$ is the number of automorphisms that map $r$ to $r$ and $r'$ to $r'$ and we have set
    \begin{align*}
        K^{(1)}_j(k,c) \coloneqq  \sum_{(T,r) \in \cT_j} \left(\frac{\beta_{(T,r)}^2 j!}{\mathrm{aut}(T,r)} e^{-c|B_T(r,k-1)|} + \sum_{\substack{r' \in V(T) \\ r' \neq r}} \frac{\beta_{(T,r)}\beta_{(T,r')} j!}{\mathrm{aut}(T,r,r')} e^{-c|B_T(r,k-1) \cup B_T(r',k-1)|}\right).
    \end{align*}

Hence, by \eqref{eq:Y_U_expec} and \eqref{eq:Y_U^2_expec},
    \begin{align*}
        \variance{Y_U} = \expec{Y_U^2} - \expec{Y_U}^2 = p^{j-1} K^{(1)}_j(k,c) +O(p^j) +  O(p^{2j-2}) = p^{j-1} K^{(1)}_j(k,c) +O(p^j).
    \end{align*}
    Hence,
    \begin{align} \label{eq:K_1_sum}
        \sum_{U \in \cU_k} \variance{Y_U} &= \sum_{j = 2}^{t_k} \binom{n}{j} \left(p^{j-1} K^{(1)}_j(k,c) +O(p^j)\right) \nonumber\\
        &= n \sum_{j=2}^{t_k} \frac{c^{j-1}K^{(1)}_j(k,c)}{j!} +O(1) \nonumber\\
        &= n K^{(1)}(k,c) +O(1),
    \end{align}
    where we have set
    \begin{align*}
        K^{(1)}(k,c) \coloneqq \sum_{j=2}^{t_k} \frac{c^{j-1}K^{(1)}_j(k,c)}{j!}.
    \end{align*}

Next, let us fix some $1 \leq j_1, j_2 \leq t_k$ and $U_1, U_2 \in \cU_k$ with $|U_1| = j_1$, $|U_2| = j_2$, and $U_1 \cap U_2 = \varnothing$ and estimate $\mathrm{cov}(Y_{U_1}, Y_{U_2})$.
    We have
    \begin{align*}
        \mathrm{cov}(Y_{U_1}, Y_{U_2}) &= \sum_{\substack{v_1 \in U_1 \\ v_2 \in U_2}} \sum_{\substack{(T_1,r_1) \in \cT_{j_1} \\ (T_2, r_2) \in T_{j_2}}} \beta_{(T_1, r_1)}\beta_{(T_2, r_2)} \cdot \\ 
        &\quad \cdot\left(\expec{I(U_1,v_1,T_1,r_1) I(U_2,v_2,T_2,r_2)} - \expec{I(U_1,v_1,T_1,r_1)}\expec{I(U_2,v_2,T_2,r_2)} \right)\\
        &= \sum_{\substack{(T_1,r_1) \in \cT_{j_1} \\ (T_2, r_2) \in T_{j_2}}} \frac{\beta_{(T_1, r_1)} \beta_{(T_2, r_2)}j_1!\,j_2!}{\mathrm{aut}(T_1, r_1)\mathrm{aut}(T_2, r_2)} \cdot p^{j_1+j_2-2} (1-p)^{(n-j_1)|B_{T_1}(r_1,k-1)| +\binom{j_1}{2}-j_1 +1} \cdot \\
        &\quad \cdot (1-p)^{(n-j_2)|B_{T_2}(r_2,k-1)| +\binom{j_2}{2}-j_2 +1} \left((1-p)^{-|B_{T_1}(r_1,k-1)|\cdot |B_{T_2}(r_2,k-1)|} -1\right).
    \end{align*}
    Since 
    \[
        (1-p)^{-|B_{T_1}(r_1,k-1)|\cdot |B_{T_2}(r_2,k-1)|} -1 = |B_{T_1}(r_1,k-1)|\cdot |B_{T_2}(r_2,k-1)| \cdot p +O(p^2),
    \]
    we have
    \begin{align*}
        \mathrm{cov}(Y_{U_1}, Y_{U_2}) &= p^{j_1+j_2-1} \sum_{\substack{(T_1,r_1) \in \cT_{j_1} \\ (T_2, r_2) \in T_{j_2}}} \frac{\beta_{(T_1, r_1)} \beta_{(T_2, r_2)}j_1!\,j_2!}{\mathrm{aut}(T_1, r_1)\mathrm{aut}(T_2, r_2)}|B_{T_1}(r_1,k-1)|\cdot |B_{T_2}(r_2,k-1)| \cdot \\
        &\qquad \cdot e^{-c(|B_{T_1}(r_1,k-1)|+|B_{T_2}(r_2,k-1)|)} +O(p^{j_1+j_2}) \\
        &= p^{j_1 + j_2 -1} K^{(2)}_{j_1,j_2}(k,c) + O(p^{j_1 + j_2}),
    \end{align*}
    where we have set
    \begin{align*}
        K^{(2)}_{j_1,j_2}(k,c) &\coloneqq \sum_{\substack{(T_1,r_1) \in \cT_{j_1} \\ (T_2, r_2) \in T_{j_2}}} \frac{\beta_{(T_1, r_1)} \beta_{(T_2, r_2)}j_1!\,j_2!}{\mathrm{aut}(T_1, r_1)\mathrm{aut}(T_2, r_2)}|B_{T_1}(r_1,k-1)|\cdot |B_{T_2}(r_2,k-1)| \cdot \\ &\quad\quad\cdot e^{-c(|B_{T_1}(r_1,k-1)|+|B_{T_2}(r_2,k-1)|)}.
    \end{align*}
    It follows that 
    \begin{align} \label{eq:K_2_sum}
        \sum_{\substack{U_1, U_2 \in \cU_k \\ U_1 \cap U_2 = \varnothing}} \mathrm{cov}(Y_{U_1}, Y_{U_2}) &= \sum_{1 \leq j_1, j_2 \leq t_k} \binom{n}{j_1} \binom{n-j_1}{j_2} \left(p^{j_1 + j_2 -1} K^{(2)}_{j_1,j_2}(k,c) + O(p^{j_1 + j_2})\right) \nonumber\\
        &= n \sum_{1 \leq j_1, j_2 \leq t_k} \frac{c^{j_1+j_2-1} K^{(2)}_{j_1,j_2}(k,c)}{j_1! \, j_2!} +O(1) \nonumber\\
        &= n K^{(2)}(k,c) +O(1),
    \end{align}
    where we have set
    \begin{align*}
        K^{(2)}(k,c) \coloneqq \sum_{1 \leq j_1, j_2 \leq t_k} \frac{c^{j_1+j_2-1} K^{(2)}_{j_1,j_2}(k,c)}{j_1! \, j_2!}.
    \end{align*}

    Next, let us fix some $1 \leq j_1, j_2 \leq t_k$ and $U_1, U_2 \in \cU_k$ with $|U_1| = j_1$, $|U_2| = j_2$, $|U_1 \cap U_2| = 1$, and $U_1 \neq U_2$ and estimate  $\mathrm{cov}(Y_{U_1}, Y_{U_2})$. 
    Let $(T_1, r_1) \in \cT_{j_1}$ and $(T_2,r_2) \in \cT_{j_2}$ and suppose that $U_1 = B_G(v_1,k)$, $(U_1,v_1) \cong (T_1, r_1)$, $U_2 = B_G(v_2,k)$, and $(U_2,v_2) \cong (T_2, r_2)$. 
    Suppose first that $|U_1| = 1$. Then since $G[U_2]$ is connected we have that $B_G(v_1,k) \neq U_1$, a contradiction. Hence we may assume $|U_1| \geq 2$ and similarly $|U_2| \geq 2$.
    Let $x$ be the vertex in $U_1 \cap U_2$. Note that since $T_1$ and $T_2$ are connected, $x$ is adjacent to a vertex in both $U_1 \setminus U_2$ and $U_2 \setminus U_1$. In particular, since $U_1 = B_G(v_1,k)$ and $U_2 = B_G(v_2,k)$, it follows that $x \in \partial_G(v_1,k) \cap \partial_G(v_2,k)$. 
    Hence, by a similar calculation as before, noting that any automorphism of $T$ which maps $r$ to itself also maps $\partial_T(r,k)$ to itself,
    \begin{align*}
        &\mathrm{cov}(Y_{U_1}, Y_{U_2}) = \\
        &\sum_{\substack{(T_1, r_1) \in \cT_{j_1} \\ (T_2, r_2) \in \cT_{j_2}}} \frac{\beta_{(T_1, r_1)} \beta_{(T_2, r_2)}j_1! \, j_2!}{\mathrm{aut}(T_1,r_1)\mathrm{aut}(T_2,r_2)} \left(\frac{|\partial_{T_1}(r_1,k)| \cdot |\partial_{T_2}(r_2,k)|}{j_1 j_2} (1-p)^{-|B_{T_1}(r_1,k-1)|\cdot |B_{T_2}(r_2,k-1)|} -1 \right) \cdot \\ 
        &\qquad \cdot p^{j_1+j_2-2} (1-p)^{(n-j_1)|B_{T_1}(r_1,k-1)| + \binom{j_1}{2}-j_1+1} (1-p)^{(n-j_2)|B_{T_2}(r_2,k-1)| + \binom{j_2}{2}-j_2+1} \\
        &= p^{j_1 + j_2 -2} \sum_{\substack{(T_1, r_1) \in \cT_{j_1} \\ (T_2, r_2) \in \cT_{j_2}}} \frac{\beta_{(T_1, r_1)} \beta_{(T_2, r_2)}j_1! \, j_2!}{\mathrm{aut}(T_1,r_1)\mathrm{aut}(T_2,r_2)} \left(\frac{|\partial_{T_1}(r_1,k)| \cdot |\partial_{T_2}(r_2,k)|}{j_1 j_2} -1 \right) \cdot \\
        &\qquad \cdot e^{-c(|B_{T_1}(r_1,k-1)| + |B_{T_2}(r_2,k-1)|)}  +O(p^{j_1 + j_2 -1})\\
        &=p^{j_1 + j_2 -2}K^{(3)}_{j_1,j_2}(k,c) +  O(p^{j_1 + j_2 -1}) ,
    \end{align*}
    where we have set
    \begin{align*}
        &K^{(3)}_{j_1,j_2}(k,c) \coloneqq \\
        &\sum_{\substack{(T_1, r_1) \in \cT_{j_1} \\ (T_2, r_2) \in \cT_{j_2}}} \frac{\beta_{(T_1, r_1)} \beta_{(T_2, r_2)} j_1! \, j_2!}{\mathrm{aut}(T_1,r_1)\mathrm{aut}(T_2,r_2)} \left(\frac{|\partial_{T_1}(r_1,k)| \cdot |\partial_{T_2}(r_2,k)|}{j_1 j_2} -1 \right) \cdot e^{-c(|B_{T_1}(r_1,k-1)| + |B_{T_2}(r_2,k-1)|)}.
    \end{align*}
    It follows that 
    \begin{align} \label{eq:K_3_sum}
            \sum_{\substack{U_1, U_2 \in \cU_k \\ |U_1 \cap U_2| = 1}} \mathrm{cov}(Y_{U_1}, Y_{U_2}) &= \sum_{1 \leq j_1, j_2 \leq t_k} \binom{n}{j_1}j_1\binom{n-j_1}{j_2}\left(p^{j_1 + j_2 -2}K^{(3)}_{j_1,j_2}(k,c) +  O(p^{j_1 + j_2 -1})\right) \nonumber\\
            &= n\sum_{1 \leq j_1, j_2 \leq t_k} \frac{c^{j_1+j_2-2}j_1 K^{(3)}_{j_1,j_2}(k,c)}{j_1!\, (j_2-1)!} + O(1) \nonumber\\
            &= n K^{(3)}(k,c) +O(1),
    \end{align}
    where we have set
    \begin{align*}
        K^{(3)}(k,c) \coloneqq \sum_{1 \leq j_1, j_2 \leq t_k} \frac{c^{j_1+j_2-2}j_1 K^{(3)}_{j_1,j_2}(k,c)}{j_1!\, (j_2-1)!}.
    \end{align*}

    Finally, let us fix some $j_1, j_2 \in [t_k]$ and $2 \leq i \leq \min\{j_1, j_2\}$. Let $V_1 \coloneqq [j_1]$
    and $V_2 \coloneqq \{j_1-i +1, \dots, j_1 + j_2 - i\}$. For $(T_1, r_1), (T_2, r_2) \in \cT$ with $v(T_1) = j_1$ and $v(T_2) = j_2$, we let $\cF_i(T_1, r_1, T_2, r_2)$ be the set of tuples $(E_1, v_1, E_2, v_2)$ with $E_1 \subseteq \binom{V_1}{2}$, $v_1 \in V_1$, $E_2 \subseteq \binom{V_2}{2}$, and $v_2 \in V_2$ such that following hold with $H_1 = (V_1, E_1)$ and $H_2 = (V_2, E_2)$.
    \begin{enumerate}[label = \upshape{(\roman*)}]
        \item $(H_1, v_1) \cong (T_1, r_1)$ and $(H_2, v_2) \cong (T_2, r_2)$,
        \item $E_1 \cap \binom{V_1 \cap V_2}{2} = E_2 \cap \binom{V_1 \cap V_2}{2} = E_1 \cap E_2$,
        \item for all $w \in B_{H_1}(v_1,k-1) \cap V_1 \cap V_2$ with $N_{H_1}(w) \setminus V_2 \neq \varnothing$, we have $w \in \partial_{H_2}(v_2,k)$,
        \item for all $w \in B_{H_2}(v_2,k-1) \cap V_1 \cap V_2$ with $N_{H_2}(w) \setminus V_1 \neq \varnothing$, we have $w \in \partial_{H_1}(v_1,k)$, and
        \item $v_1 \neq v_2$.
    \end{enumerate}
    Let $U_1, U_2 \subseteq V(G)$ with $U_1 \neq U_2$, $|U_1| = j_1$, $|U_2| = j_2$, and $|U_1 \cap U_2| = i$. Without loss of generality assume that $U_1 = V_1$ and $U_2 = V_2$. Now note that, for $(T_1, r_1), (T_2, r_2) \in \cT$, $v_1 \in U_1$, and $v_2 \in U_2$, we have that if
    \[
U_i = B_G(v_i,k) \text{ and } (G[U_i],v_i) \cong (T_i, r_i) \quad \text{ for } i=1,2,
\]
then $(E(G[U_1]), v_1, E(G[U_2]), v_2) \in \cF_i(T_1, r_1, T_2, r_2)$.
    Furthermore, if $E_i = E(G[B_G(v_i,k)])$ for $i \in \{1,2\}$, then the converse also holds. Hence, we have
    \begin{align*}
        &\expec{Y_{U_1} Y_{U_2}} =  
        \sum_{\substack{(T_1, r_1) \in \cT_{j_1} \\ (T_2, r_2) \in \cT_{j_2}}} \beta_{(T_1,r_1)}\beta_{(T_2,r_2)} \cdot \\
        &\cdot \sum_{\substack{(E_1, v_1, E_2, v_2) \\ \in \cF_i(T_1,r_1,T_2,r_2)}} p^{j_1 + j_2 -2 - |E_1 \cap E_2|} (1-p)^{\binom{j_1+j_2-i}{2} - |E_1 \cup E_2| + |B_{H_1}(v_1,k-1) \cup B_{H_2}(v_2,k-1)| (n - j_1 - j_2 +i)}.
    \end{align*}
    It is easy to verify that the above sum is dominated by the contribution from the set of tuples $(E_1, v_1, E_2, v_2) \in \cF_i(T_1, r_1, T_2, r_2)$ with $|E_1 \cap E_2| = i-1$, which we denote by $\tcF_i(T_1, r_1, T_2, r_2)$. In particular, we have
    \begin{align*}
        \expec{Y_{U_1} Y_{U_2}} &= p^{j_1 + j_2 -i -1} \sum_{\substack{(T_1, r_1) \in \cT_{j_1} \\ (T_2, r_2) \in \cT_{j_2}}} \beta_{(T_1,r_1)}\beta_{(T_2,r_2)} \sum_{\substack{(E_1, v_1, E_2, v_2) \\ \in \tcF_i(T_1,r_1,T_2,r_2)}} e^{-c |B_{H_1}(v_1,k-1) \cup B_{H_2}(v_2,k-1)|} \\ &\quad\quad+ O(p^{j_1 + j_2 -i}) \\
        &= p^{j_1 + j_2 -i -1} K^{(4)}_{j_1, j_2, i}(k,c) + O(p^{j_1 + j_2 -i}),
    \end{align*}
    where we have set 
    \begin{align*}
        K^{(4)}_{j_1, j_2, i}(k,c) \coloneqq \sum_{\substack{(T_1, r_1) \in \cT_{j_1} \\ (T_2, r_2) \in \cT_{j_2}}} \beta_{(T_1,r_1)}\beta_{(T_2,r_2)} \sum_{\substack{(E_1, v_1, E_2, v_2) \\ \in \tcF_i(T_1,r_1,T_2,r_2)}} e^{-c |B_{H_1}(v_1,k-1) \cup B_{H_2}(v_2,k-1)|}.
    \end{align*}
    Since $i \geq 2$, we have $\expec{Y_{U_1}} \expec{Y_{U_2}} = O(p^{j_1 +j_2 -2}) = O(p^{j_1 + j_2 -i})$.
    Hence, $\mathrm{cov}(Y_{U_1}, Y_{U_2}) = O(p^{j_1+j_2-i}) + p^{j_1 + j_2 -i -1} K^{(4)}_{j_1, j_2, i}(k,c)$.
    It follows that 
    \begin{align} \label{eq:K_4_sum}
            &\sum_{i =2}^{t_k} \sum_{\substack{j_1, j_2 \in \{i, \dots, t_k\} \\ (j_1, j_2) \neq (i,i)}} \sum_{\substack{U_1, U_2 \in \cU_k \\ |U_1| = j_1, |U_2| = j_2 \\ |U_1 \cap U_2| = i}} \mathrm{cov}(Y_{U_1}, Y_{U_2}) \nonumber\\
            = &\sum_{i =2}^{t_k} \sum_{\substack{j_1, j_2 \in \{i, \dots, t_k\} \\ (j_1, j_2) \neq (i,i)}} \binom{n}{j_1} \binom{j_1}{i} \binom{n-i}{j_2} \left( p^{j_1 + j_2 -i -1} K^{(4)}_{j_1, j_2, i}(k,c) +O(p^{j_1 + j_2 -i})\right) \nonumber\\
            = &\sum_{i =2}^{t_k} \sum_{\substack{j_1, j_2 \in \{i, \dots, t_k\} \\ (j_1, j_2) \neq (i,i)}} n \cdot \frac{c^{j_1+j_2-i-1}\binom{j_1}{i}}{j_1! \, j_2!} K^{(4)}_{j_1, j_2, i}(k,c) +O(1) \nonumber\\
            = &K^{(4)}(k,c) n +O(1), 
    \end{align}
    where we have set
    \begin{align*}
        K^{(4)}(k,c) \coloneqq \sum_{i =2}^{t_k} \sum_{\substack{j_1, j_2 \in \{i, \dots, t_k\} \\ (j_1, j_2) \neq (i,i)}}\frac{c^{j_1+j_2-i-1}\binom{j_1}{i}}{j_1! \, j_2!} K^{(4)}_{j_1, j_2, i}(k,c).
    \end{align*}
    It now follows from \cref{eq:var_sum_split,,eq:K_1_sum,,eq:K_2_sum,,eq:K_3_sum,,eq:K_4_sum}, that 
    \begin{align*}
        \variance{\widehat{M}_k} = (K^{(1)}(k,c) + K^{(2)}(k,c) + K^{(3)}(k,c) + K^{(4)}(k,c))n +O(1).
    \end{align*}
    Hence, the statement of the lemma is satisfied with $$K(k,c,\bm{\beta}) \coloneqq K^{(1)}(k,c) + K^{(2)}(k,c) + K^{(3)}(k,c) + K^{(4)}(k,c).$$

\end{document}